\documentclass[a4paper]{amsart}
\usepackage[utf8]{inputenc}
\usepackage[T1]{fontenc}
\usepackage{lmodern, enumerate}
\usepackage{amssymb,amsxtra}
\usepackage{amscd}
\usepackage[mathscr]{euscript}
\usepackage{stmaryrd}
\usepackage[all]{xy}
\usepackage{xcolor}

\usepackage{nicefrac,mathtools}
\usepackage{microtype}
\usepackage{tikz-cd}
\usepackage{tkz-graph}
\usetikzlibrary{arrows}
\usepackage{comment}
\usetikzlibrary{positioning}
\usepackage[pdftitle={Inverse Semigroups},
 pdfauthor={Pere Ara, Alcides Buss, Ado Dalla Costa},
 pdfsubject={Mathematics}
]{hyperref}
\usepackage[lite]{amsrefs}
\renewcommand*{\MR}[1]{ \href{http://www.ams.org/mathscinet-getitem?mr=#1}{MR #1}}

\numberwithin{equation}{section}
\theoremstyle{plain}
\newtheorem{theorem}[equation]{Theorem}
\newtheorem{lemma}[equation]{Lemma}
\newtheorem{proposition}[equation]{Proposition}
\newtheorem{corollary}[equation]{Corollary}
\theoremstyle{definition}
\newtheorem{definition}[equation]{Definition}
\newtheorem{notation}[equation]{Notation}
\theoremstyle{remark}
\newtheorem{remark}[equation]{Remark}
\newtheorem{example}[equation]{Example}

\newcommand*{\SI}{\mathfrak{Fr}} %
\newcommand*{\IS}{\mathcal{S}} %
\renewcommand*{\S}{\tilde{\mathcal{S}}} %

\newcommand*{\FIM}{\mathfrak{F_{im}}} 

\newcommand*{\C}{\mathbb C}

\newcommand*{\N}{\mathbb N}
\newcommand*{\F}{\mathbb F} 
\newcommand*{\FC}{\mathbb F_{C}} 

\newcommand*{\Group}{\Gamma} 
\newcommand*{\G}{\mathcal{G}} 
\newcommand*{\R}{\mathcal{R}} 

\newcommand*{\alg}{\mathrm{alg}}

\newcommand*{\FG}{\mathfrak{F_{gr}}}

\newcommand*{\FIS}{\mathfrak{F_{is}}}

\newcommand*{\FIC}{\mathfrak{F_{ic}}}

\newcommand*{\f}{\mathfrak f}

\newcommand{\e}{\mathbf{e}}

\newcommand*{\nb}{\nobreakdash}

\renewcommand*{\L}{\mathcal L}
\newcommand*{\Co}{\mathcal C}
\newcommand*{\Le}{\mathcal L}
\newcommand*{\To}{\mathcal{TP}}
\newcommand*{\T}{\mathcal{T}}

\newcommand*{\cont}{C}
\newcommand*{\contz}{\cont_0}

\newcommand{\FP}{\mathbb{FP}}

\newcommand{\Cfin}{C^{\mathrm{fin}}}

\newcommand*{\U}{\mathcal U}
\newcommand*{\E}{\mathcal E}
\newcommand*{\Y}{\mathcal X}%
\newcommand*{\YY}{\mathcal Y}%
\newcommand*{\OEC}{\mbox{$\mathcal{O}(E,C)$}} 

\newcommand*{\OECR}{\mbox{$\mathcal{O}_r(E,C)$}}

\newcommand{\SInf}{E^{0,r}_{\infty}}

\newcommand{\OO}{\mathcal{O}}

\newcommand*{\congto}{\xrightarrow\sim}

\newcommand*{\PI}[1]{\langle#1\rangle}

\newcommand*{\s}{\textup s}
\newcommand*{\rg}{\textup r}
\newcommand*{\sbe}{\subseteq} 
\newcommand*{\Free}{\mathbb F}

\newcommand*{\cstar}{\texorpdfstring{$C^*$\nobreakdash-\hspace{0pt}}{*-}}

\newcommand*{\onto}{\twoheadrightarrow}
\newcommand*{\red}{\text{red}}
\renewcommand*{\max}{\mathrm{max}}
\newcommand*{\ab}{\mathrm{ab}}

\newcommand*{\tight}{\mathrm{tight}}

\newcommand*{\dual}[1]{\widehat{#1}}

\newcommand{\Etight}{\dual\E_\tight}

\newcommand{\Aut}{\text{Aut}}

\newcommand{\idealin}{\mathrel{\trianglelefteq}} 

\begin{document}
\title[Inverse semigroups of separated Graphs]{Inverse semigroups of separated graphs\\ and associated algebras}

\author{Pere Ara}
\address{Departament de Matem\`atiques, Edifici Cc, Universitat Aut\`onoma de Barcelona, 08193 Cerdanyola del Vall\`es (Barcelona), Spain, and}
\email{pere.ara@uab.cat}

\author{Alcides Buss}
\address{Departamento de Matem\'atica\\
  Universidade Federal de Santa Catarina\\
  88.040-900 Florian\'opolis-SC\\
  Brazil}
\email{alcides.buss@ufsc.br}

\author{Ado Dalla Costa}
\address{Setor de Matem\'atica do Departamento de Administração Empresarial \\
	Universidade Estadual de Santa Catarina\\
	88.035-001 Florian\'opolis-SC\\
	Brazil}
 \email{adodallacosta@hotmail.com \\ ado.costa@udesc.br}

\begin{abstract}
In this paper we introduce an inverse semigroup $\IS(E,C)$ associated to a separated graph $(E,C)$ and describe its 
internal structure. In particular we show that it is strongly $E^*$-unitary and can be realized as a partial 
semidirect product of the form $\YY\rtimes\F$ for a certain partial action of the free group 
$\F=\F(E^1)$ on the edges of $E$ on a semilattice $\YY$ realizing the idempotents of $\IS(E,C)$. 
In addition we also describe the spectrum as well as the tight spectrum of $\YY$.

We then use the inverse semigroup $\IS(E,C)$ to describe several ``tame'' algebras associated to $(E,C)$, including 
its Cohn algebra, its Leavitt-path algebra, and analogues in the realm of \cstar{}algebras, like the tame 
\cstar{}algebra $\OO(E,C)$ and its Toeplitz extension $\T(E,C)$, proving that these algebras are canonically 
isomorphic to certain algebras attached to $\IS(E,C)$. Our structural results on $\IS(E,C)$ imply that these algebras can be realized as partial crossed products, revealing a great portion of their structure.

\end{abstract}

\subjclass[2020]{46L55, 20M18}

\keywords{Inverse semigroup, Separated Graph, C*-Algebras}

\thanks{The first author was partially supported by the Spanish State Research Agency (grants No.\ PID2020-113047GB-I00/AEI/10.13039/501100011033, PID2023-147110NB-I00,  and CEX2020-001084-M), and by the Comissionat per Universitats i Recerca de la Generalitat de Catalunya (grant No.\ 2021-SGR-01015). The second author was supported by CNPq, CAPES/Humboldt and FAPESC, and the third author by CNPq - 402924/2022-3.}

\maketitle

\tableofcontents

\section{Introduction}

Inverse semigroups provide important tools in the study of several types of \cstar{}algebras. Close connections between inverse semigroups, étale groupoids and operator algebras were established in the work of Paterson \cite{Paterson:Groupoids}, and this was later extended by Exel \cite{Exel:Inverse_combinatorial} with the introduction of the tight spectrum and the tight \cstar{}algebra of an inverse semigroup. 

Graph \cstar{}algebras are usually interpreted as the ``\cstar{}algebras one can see'', as several of their properties or invariants (like simplicity or their $K$-theory) can be read off directly from the graphs that originate them. On the other hand, this class of \cstar{}algebras offers some limitations, as they are always nuclear, their odd-degree $K$-theory can only give free abelian groups, and their non-stable $K$-theory is always unperforated and separative. For this reason, several  variations on the standard theory have been considered in order to get more general models for \cstar{}algebras. Among these, one can consider, for instance, the labelled graphs \cite{BatesPask}, ultragraphs \cite{Tomforde} and higher-rank graphs \cite{KumjianPask}, and their associated \cstar{}algebras.

The purpose of this paper is to introduce the inverse semigroup $\IS (E,C)$ of a separated graph $(E,C)$ and to develop Exel's program \cite{Exel:Inverse_combinatorial} in order to describe the structure of the tight algebras associated with this inverse semigroup. To achieve this goal, we will examine the internal structure of the semigroup $\IS (E,C)$, obtaining a detailed description of its semilattice of idempotents, which is a crucial component in our analysis. 
The inverse semigroup $\IS(E,C)$ generalizes the graph inverse semigroup $\mathcal S (E)$ first introduced by Ash and Hall in \cite{AshHall}, which has been widely analyzed by several authors, see e.g. \cites{JonesLawson, LaLonde-ConditionK, meakin-wang-2021, meakin-milan-wang-2021, mesyan-mitchell-2016}.

Recall that a separated graph is a pair $(E,C)$ consisting of a (directed) graph $E=(s,r\colon E^1\to E^0)$ and a separation $C$ on $E$, meaning a certain partition of its edges. We can attach to such a pair several 
abstract $*$-algebras (over some base commutative unital ring with involution $K$) as well as \cstar{}algebras. These are algebras (or \cstar{}algebras) generated by projection copies of the vertices $E^0$ and partial isometry copies of the edges $E^1$, satisfying some natural relations coming from the graph structure, and also taking into account its separation. In addition to the canonical path relations, we consider the (SCK1)-relation:
$$e^*f=\delta_{e,f}r(e),\quad e,f\in X,\quad\mbox{ for }X\in C.$$
Taking into account only these relations, we get the so-called Cohn path algebra of the separated graph over $K$, which we denote by $\Co_K(E,C)$. 
Its enveloping \cstar{}algebra is the so-called Toeplitz path \cstar{}algebra of $(E,C)$, denoted by $\To(E,C)$.
Adding the extra (SCK2)-relation:
$$\sum_{e\in X}ee^*=v\quad\mbox{for all }X\in C,\, X\sbe s^{-1}(v) \mbox{ with } 0<|X|<\infty$$
one obtains the Leavitt path algebra $\Le_K(E,C)$, introduced in \cite{AG12}, and its companion \cstar{}algebra $C^*(E,C)$, which is the main \cstar{}algebra originally studied in \cite{Ara-Goodearl:C-algebras_separated_graphs}. 

The copy of the graph into the Cohn path algebra $\Co_K(E,C)$ generates a $*$-semigroup $\S(E,C)$, which can also be viewed as the universal $*$-semigroup with zero generated by projections $v\in E^0$ and partial isometries $e\in E^1$ satisfying the same relations (in the category of $*$-semigroups with zero) that define the Cohn path algebra, or its Toeplitz \cstar{}algebra. Indeed, it is a simple observation to deduce the isomorphisms
\begin{equation}\label{eq:iso-Cohn-Cst}
    K[\S(E,C)]\cong\Co_K(E,C),\qquad C^*(\S(E,C))\cong \To(E,C),
\end{equation}
where $K[S]$ denotes the (zero-preserving) $*$-algebra over $K$ of a $*$-semigroup with zero $S$, and $C^*(S)$ denotes the enveloping \cstar{}algebra of $\C[S]$. In particular it follows that the Leavitt path $*$-algebra $\Le_K(E,C)$ or its enveloping \cstar{}algebra $C^*(E,C)$ can be realized as a quotient of $K[\S(E,C)]$ or $C^*(\S(E,C))$, respectively. However this quotient cannot be directly obtained from the $*$-semigroup $\S(E,C)$ as the (SCK2)-relation involves a sum and does not make sense in this category.

Another, more basic, structural problem is that $\S(E,C)$ is not, in general, an inverse semigroup. This happens because, although its generators are partial isometries (i.e. satisfy the relation $xx^*x=x$), this is no longer true, in general, for their products. One can fix this by adding the extra commutation relations $e(x)e(y)=e(y)e(x)$, where $e(x):=xx^*$. More precisely, defining $\IS(E,C)$ as the quotient of $\S(E,C)$ by the congruence generated by these commutation relations for all $x,y\in \S(E,C)$, we can prove that $\IS(E,C)$ is an inverse semigroup. Indeed, it turns out to be the universal inverse semigroup generated by $v\in E^0$ and $e\in E^1$ satisfying the same relations that define the $*$-semigroup $\S (E,C)$, or its associated Cohn $*$-algebra or Toeplitz path \cstar{}algebra, as explained above.

The main goal of this paper is to study the internal structure of the inverse semigroup $\IS(E,C)$ and its associated $*$-algebras, as well as \cstar{}algebras. As should be expected, the $*$-algebra $K[\IS(E,C)]$ is a ``tame'' quotient of $K[\S(E,C)]\cong \Co_K(E,C)$, namely the quotient by the ideal generated by the commutators $[e(x),e(y)]=e(x)e(y)-e(y)e(x)$ for $x,y\in \S(E,C)\sbe K[\S(E,C)]$. We denote this quotient $*$-algebra by $\Co_K^\ab(E,C)$. Similarly, we can define a tame quotient \cstar{}algebra of $\To(E,C)$, which we denote by $\T(E,C)$. Using these notations, a standard argument using their universal defining relations shows that we have canonical isomorphisms factoring those in~\eqref{eq:iso-Cohn-Cst}:
\begin{equation}\label{eq:tame-iso-Cohn-Cst}
    K[\IS(E,C)]\cong\Co^\ab_K(E,C),\qquad C^*(\IS(E,C))\cong \T(E,C).
\end{equation}
In a similar way, we can define tame versions of the Leavitt path $*$-algebra $\Le_K(E,C)$ and its enveloping \cstar{}algebra $C^*(E,C)$, and we denote these tame quotients by $\Le_K^\ab(E,C)$ and $\OO(E,C)$, respectively. For certain special classes of separated graphs, these algebras have been already introduced and studied before in several papers, see for instance \cites{Ara-Exel:Dynamical_systems, AraLolk, Lolk:tame, ABPS, AEK}. 
We show in Theorem \ref{thm:iso-tight-algebras} that these algebras can be described in terms of the tight algebras of the inverse semigroup $\IS(E,C)$, more precisely we have canonical isomorphisms factoring~\eqref{eq:tame-iso-Cohn-Cst}:
\begin{equation}\label{eq:tame-tight-iso-Cohn-Cst}
    K_\tight[\IS(E,C)]\cong\Le^\ab_K(E,C),\qquad C^*_\tight(\IS(E,C))\cong \OO(E,C).
\end{equation}
The tight algebra of an inverse semigroup $S$ can be understood once we have a good knowledge of its canonical action on the spectrum $\dual\E$ (i.e. the space of filters) of the semilattice of idempotents $\E= \E (S)$, and on the tight spectrum $\dual\E_\tight$, which is the closure in $\dual\E$ of the space $\dual\E_\infty$ of maximal filters. Taking germs of these actions of $S$ on $\dual\E$ and $\dual\E_\tight$, one gets the universal groupoid $\G(S)$ of $S$, and its tight quotient $\G_\tight(S)$, with associated \cstar{}algebras $C^*(\G(S))\cong C^*(S)$ and $C^*(\G_\tight(S))\cong C^*_\tight(S)$, respectively, and similarly for their reduced \cstar{}algebras, or their abstract $*$-algebras over the ring $K$.

We will obtain a detailed description of the semilattice $\E=\E(\IS(E,C))$ and its (tight) spectrum, along with the canonical action of $\IS(E,C)$, leading to a complete description of their associated groupoids and algebras.
Moreover, we show that elements of $\IS(E,C)$ have a standard form, which we call the \emph{Scheiblich normal form}, because it is similar to the well-known Scheiblich normal form for the elements of the free inverse semigroup \cite{lawson}*{Chapter~6}. More precisely, we show that every non-trivial element of $\IS(E,C)$ can be represented as an expression of the form
\begin{equation}\label{eq:Scheiblich-normal-form}
    (\gamma_1\gamma_1^{-1})\cdots (\gamma_n \gamma _n^{-1})\lambda 
\end{equation} 
for certain paths $\gamma_i, \lambda$ in the fundamental groupoid $\FG(E)$ of $E$. Moreover, we show that this representation is unique if we add some natural conditions on the words appearing in the above representation. This can be geometrically expressed in terms of Munn $(E,C)$-trees, which are certain subtrees of the Cayley graph $\Gamma_E$ of $\FG(E)$ (see Section \ref{sect:Munntrees}). 

The Scheiblich normal form is essentially equivalent to a representation of the inverse semigroup $\IS(E,C)$ as a (restricted) semidirect product of the form
\begin{equation}\label{eq:semi-direct-representation}
    \IS(E,C)\cong \YY\rtimes_\theta^r \F,
\end{equation}
where $\YY$ is an isomorphic copy of the semilattice $\E=\E(\IS(E,C))$, which can be described in different forms in terms of certain special paths in $\FG(E)$ that take into account the structure of the separated graph $(E,C)$. Here $\theta$ is a partial action of the free group $\F=  \F(E^1)$ on the edge set $E^1$ of $E$ on $\YY$ by partial semilattice isomorphisms between ideals of $\YY$. In particular this shows that $\IS(E,C)$ is a well-behaved inverse semigroup: it is strongly $E^*$-unitary, which means that it admits an idempotent pure partial homomorphism $\f\colon \IS(E,C)^\times\to \F$, which we view as a `grading'. 
This map reads off the element $\lambda\in \F$ in the Scheiblich normal form~\eqref{eq:Scheiblich-normal-form}. 

The isomorphism~\eqref{eq:semi-direct-representation} then leads immediately to a similar representation of the groupoid $\G(E,S):=\G(\IS(E,S))$ and its tight quotient as semidirect products of the form
$$\G(E,S)\cong \dual\YY\rtimes\F,\qquad\G_\tight(E,S)\cong\dual\YY_\tight\rtimes\F$$
with respect to the (dual) partial action of $\F$ on $\dual\YY$ induced from $\theta$. We also  describe both $\dual\YY$ and $\dual\YY_\tight$ in terms of certain subsets of $\FG(E)$, or, equivalently, in terms of certain infinite Munn $(E,C)$-trees. All this implies analogous decompositions of the several tame algebras associated with $(E,C)$ as partial crossed products by partial actions of $\F$:
\begin{equation*}\label{eq:alg-tame-iso-Cohn-Cst}
    \Co^\ab_K(E,C)\cong C_K(\dual\YY)\rtimes\F,\qquad \T(E,C)\cong C_0(\dual\YY)\rtimes\F
\end{equation*}
and
\begin{equation*}\label{eq:alg-tametight-iso-Cohn-Cst}
    \Le^\ab_K(E,C)\cong C_K(\dual\YY_\tight)\rtimes\F,\qquad \OO(E,C)\cong C_0(\dual\YY_\tight)\rtimes\F.
\end{equation*}
Here $C_K(X)$ denotes the commutative $K$-algebra of compactly supported locally constant functions $X\to K$ from a totally disconnected locally compact Hausdorff space $X$, and $\contz(X)$ the commutative \cstar{}algebra of continuous functions $X\to \C$ vanishing at infinity. In the \cstar{}algebra case, we can also look at the reduced crossed product, and this corresponds to the reduced tame \cstar{}algebra of $(E,C)$:
$$\OO_r(E,C)\cong C_0(\dual\YY_\tight)\rtimes_r\F.$$
This is, in general, a proper quotient of $\OO(E,C)$: indeed, we have $\OO(E,C)=\OO_r(E,C)$ if and only if the partial action of $\F$ on $\dual\YY_\tight$ is amenable, and this is equivalent to nuclearity of either $\OO(E,C)$ or $\OO_r(E,C)$, or amenability of the underlying groupoid $\G_\tight(E,C)$.

The paper is organized as follows. After a section on preliminaries, we introduce in Section \ref{sect:InvSemSG} the inverse semigroup $\IS(E,C)$ associated to a separated graph $(E,C)$. We show in Theorem \ref{thm:ECMunntrees} the basic representation result of $\IS(E,C)$ in terms of Munn $(E,C)$-trees. This gives immediately the main structural result concerning the representation of $\IS(E,C)$ as a partial semidirect product and the uniqueness of the Scheiblich normal form for elements of $\IS (E,C)$ (Theorem \ref{thm:uniqueness-SNF}). Using the Scheiblich normal form, we also show a rigidity result for the automorphism group of $\IS(E,C)$ in Theorem 
\ref{thm:automorphisms}. We characterize in Section \ref{sect:tightsection} various spaces of filters of interest. We first identify in Proposition \ref{prop:characterizing-filt-ultra-tight} the space of filters $\dual\E$ on the semilattice of idempotents $\E$ of the inverse semigroup $\IS (E,C)$ as a space consisting of some specified subsets of the set $\FC $ of $C$-separated paths (see Definition \ref{def:Cseparatedword-and-compatible} for the definition of the latter). Using this identification, we provide the characterization of the space of ultrafilters $\dual\E_\infty$ in Theorem \ref{cor:ultrafilters}, and of the space of tight filters $\dual\E_\tight$ in Theorem \ref{thm:charac-tightfilters}. As an immediate consequence of these characterizations, we obtain that $\dual\E_\infty$ is closed in $\dual\E$ if and only if $(E,C)$ is a finitely separated graph (Corollary \ref{cor:ultrafilters-closed}). Finally we deal in Section \ref{sect:algebras} with the algebras associated to separated graphs in complete generality. The main technical result in this section is Theorem \ref{thm:Leavitt-universal-tight}, which shows that the tame Leavitt path algebra $\Le^\ab_K(E,C)$ is the universal algebra for tight representations of the inverse semigroup $\IS(E,C)$. This is a principal ingredient for the proof of Theorem \ref{thm:iso-tight-algebras}, which contains our main results on the structure of the tame Leavitt path algebra $\Le^\ab_K(E,C)$ and the tame $C^*$-algebra $\OO (E,C)$ of a separated graph.

The structural results developed in this work not only clarify the internal properties of $\IS(E,C)$, but also provide a natural framework to analyze the associated algebras -- such as Leavitt path algebras and graph C$^*$-algebras -- via partial crossed product decompositions. This perspective connects the combinatorial nature of separated graphs with the dynamics of \'etale groupoids, offering tools to investigate key properties of the associated algebras, including simplicity, pure infiniteness, and K-theory. Such decomposition techniques have already proven effective in the study of these algebras in earlier works (e.g., \cites{Ara-Exel:Dynamical_systems, AraLolk, Lolk:tame, Lolk:nuclearity}), and continue to play a central role in ongoing developments. While a full exploration of these implications lies beyond the scope of the present paper, we expect that the framework presented here will support and inspire further investigations in this direction.

\section{Preliminaries}\label{sec:Preliminaries}

In this section we summarize some of the main tools we are going to use throughout the paper. This will also consolidate the notation we use.

\subsection{Directed graphs}
	A \emph{directed graph} $E$ is a quadruple of the form $E=(E^0,E^1,s,r)$ consisting of two sets $E^0$,$E^1$ and two maps $s,r: E^1 \rightarrow E^0$. The elements of $E^0$ and $E^1$ are called vertices and edges and the maps $s,r$ are called the source and range maps, respectively. Although in the literature this is a common assumption, we will not need to make any general assumptions on the cardinality of our graphs; in particular, we do not require our graphs to be finite or countable.
	
	The graph is called \emph{row-finite} if every vertex emits at most finitely many edges. A \emph{finite path} in $E$ is a sequence of edges of the form $\mu:= e_1 \ldots e_n$ with $r(e_i)=s(e_{i+1})$ for all $i \in \{1, \ldots, n-1\}$. The length of $\mu$ is $|\mu|:=n$ and paths with length $0$ are identified with the vertices of $E$ (we set $s(v)=r(v)=v$). We denote by $E^n$ the set of all finite paths with length $n$ and $\text{Path}(E) := \displaystyle\cup_{n=0}^{\infty} E^n$ denotes the set of all paths of $E$. We can extend the source and range maps to $\text{Path}(E)$ in the obvious way: if $\mu=e_1\ldots e_n \in \text{Path}(E)$, then $s(\mu) = s(e_1)$ and $r(\mu) = r(e_n)$. Given two paths $\mu,\nu \in \text{Path}(E)$ with $r(\mu) = s(\nu)$, one obtains a new path $\mu\nu$ by concatenation with $|\mu\nu| = |\mu|+|\nu|$. 
 
	Given a graph $E$, we define its \emph{extended graph} (also called the \emph{double graph} of $E$)
as the new graph $\hat{E} = (E^0,E^1 \cup  E^{-1}, r, s)$, containing $E$ as a subgraph, and where we set
$s(e^{-1}) := r(e)$, $r(e^{-1}) := s(e)$ for all $e \in  E^1$.
\subsection{Separated graphs}
A \emph{separated graph} is a pair $(E,C)$ consisting of a graph $E=(E^0,E^1,s,r)$ and a \emph{separation} $C=\bigsqcup_{v\in E^0}C_v$ on $E$, consisting of partitions $C_v$ of $s^{-1}(v)\sbe E^1$ into pairwise disjoint nonempty subsets (with $C_v= \emptyset$ if $v$ is a sink). The \emph{trivial separation} is the separation with $C_v=\{ s^{-1}(v)\}$ for all non-sinks $v\in E^0$; a graph with the trivial separation is also called trivially separated or a non-separated graph. The \emph{free separation} of $E$ is the finest separation, where each $s^{-1}(v)$ is separated into singletons, that is, $C_v=\{\{e\}:e\in s^{-1}(v)\}$. 

If all the sets in $C$ are finite, we say that $(E,C)$ is a \emph{finitely separated graph}.

\subsection{Semigroups with involution and inverse semigroups} Recall that a semigroup is a set $S$ endowed with an associative multiplication $S\times S\to S$, usually written as concatenation $(s,t)\mapsto st$. 
An \emph{idempotent} of $S$ is an element $e\in S$ satisfying $e^2=e$; we write $\E(S)$ for the subset of idempotents of $S$. An involution on $S$ is an involutive and anti-multiplicative map $S\to S$, which we usually write as a pseudo-inverse map $s\mapsto s^{-1}$, so it satisfies $(s^{-1})^{-1}=s$ and $(st)^{-1}=(t^{-1})(s^{-1})$. When $S$ is a semigroup endowed with an involution, we also call it a $*$-semigroup. An inverse semigroup is a $*$-semigroup satisfying 
$$(1)\,\, ss^{-1}s=s\quad \mbox{and}\quad  (2)\,\, (ss^{-1})(tt^{-1})=(tt^{-1})(ss^{-1})\quad\mbox{for all }s,t\in S.$$ 
In this case, the pseudo-inverse $s^{-1}$ of $s\in S$ is uniquely determined by the relations $ss^{-1}s=s$ and $s^{-1}ss^{-1}=s^{-1}$. Moreover $\E(S)$ is a commutative inverse subsemigroup consisting solely of idempotents. Those are exactly the semilattices, with meet given by the multiplication: $e\wedge f= ef$. Moreover, inverse semigroups are naturally endowed with a partial order:
$$s\leq t\Leftrightarrow ts^{-1}s=s\Leftrightarrow ss^{-1}t=s.$$
Given a semigroup $S$, we may always add a (formal) zero element $0$, getting a semigroup with zero $S_0=S\sqcup \{0\}$, as well as a unit $1$, obtaining a semigroup with unit (i.e. a monoid) $S_1=S\sqcup\{1\}$. 
We shall also write $S_u$ for the minimal unitization of $S$, that is, $S_u=S$ if $S$ is already unital and $S_u=S_1$ (formal unitization) otherwise. Our semigroups of interest will usually have already a zero, sometimes a unit. If $S$ already has a unit, we shall generally write $S^*=S\backslash\{1\}$, and if $S$ has a zero, $S^\times =S\backslash\{0\}$. Of course, these are not semigroups, in general.

\subsection{Strongly E*-unitary inverse semigroups and partial actions} A homomorphism $\pi \colon S \to \Group $ from an inverse semigroup $S$ to a group $\Group$ is said to be {\it idempotent pure} if $\pi^{-1}(1)=\E(S)$.  
An inverse semigroup $S$ is {\it $E$-unitary} is there is an idempotent pure homomorphism $\pi \colon S\to \Group$, where $\Group$ is a group.

Since we will work with semigroups with $0$, we need the corresponding notion in this setting. Let $S$ be an inverse semigroup with zero and $\Group$ a group. A {\it partial homomorphism} from $S$ to $\Group$ is a map $\pi \colon S^\times \to \Group$ such that $\pi (st) =\pi (s)\pi (t)$ whenever $st\ne 0$. A partial homomorphism $\pi$ is {\it idempotent pure} if $\pi^{-1}(1)= \E(S)^\times$. 
If such an idempotent pure partial homomorphism exists, we say that $S$ is \emph{strongly $E^*$-unitary}.
 
The above concepts are intimately related to the notion of a semidirect product associated to a partial action of a group on a semilattice, see \cite{KelLawson} and \cite{MilanSteinberg}. We briefly recall the relevant definitions.  

 If $\Group$ is a group acting partially on a semilattice $\E$ via partial isomorphisms $\theta_g\colon D_{g^{-1}}\to D_g$ between ideals $D_g\idealin \E$, then we define its {\it semidirect product} as $\E\rtimes_\theta\Group=\{(e,g): e\in D_g, g\in \Group\}$. We write elements of $\E\rtimes_\theta\Group$ as $e\delta_g$ with $e\in D_g$, $g\in \Group$. This becomes an inverse semigroup when endowed with the following product and inverse:
$$(e\delta_g)\cdot (f\delta_h):=\theta_g(\theta_g^{-1}(e)f)\delta_{gh},\quad (e\delta_g)^{-1}:=\theta_g^{-1}(e)\delta_{g^{-1}}.$$
Of course, this is similar to the definition of crossed products of group partial actions on algebras, see \cite{Exel:Partial_dynamical}, and a similar proof shows that the above product is indeed associative and turns $\E\rtimes_\theta\Group$ into an inverse semigroup. The inverse semigroups of the form $\E\rtimes_\theta \Group$ are $E$-unitary inverse semigroups (see \cite{KelLawson} and \cite{MilanSteinberg}) with semilattice of idempotents $\E (\E\rtimes_\theta\Group)\cong \E$ and the canonical grading homomorphism $\E\rtimes_\theta\Group\to\Group$ induces an isomorphism from the maximal group image of $\E\rtimes_\theta\Group$ onto a subgroup of $\Group$.

As our inverse semigroups usually have a zero, the above construction is not suitable for us, and we modify it as follows. If $\E$ has a zero, then all non-empty ideals $D_g$ necessarily contain the zero $0\in D_g$, and the subset $J_0:=\{0\delta_g: D_g\not=\emptyset\}$ is an ideal of $\E\rtimes_\theta\Group$. We define the \emph{restricted} semidirect product as the Rees quotient inverse semigroup $$\E\rtimes_\theta^r\Group:=(\E\rtimes_\theta\Group)/J_0=\big((\E\rtimes_\theta\Group)\backslash J_0\big)\sqcup \{0\}.$$
In other words, $\E\rtimes_\theta^r\Group$ is the quotient of $\E\rtimes_\theta\Group$ where all elements $0\delta_g$ are identified with zero. The grading homomorphism on $\E\rtimes_\theta\Group$ factors through an idempotent pure partial grading homomorphism $(\E\rtimes_\theta^r\Group)^\times\to \Group$. In particular this implies that $\E\rtimes_\theta^r\Group$ is strongly $E^*$\nb-unitary. 
The converse also holds, as follows. Indeed, the following result is essentially well known: it appears in \cite[Theorem 2.5]{BulFouGo}, \cite{Steinberg2003}, \cite{MilanSteinberg}*{Section~5} and \cite{Li-K-theory}*{Section~2} in disguised forms.

\begin{proposition}
	\label{prop:KelLawson-result}
	Let $S$ be a strongly $E^*$-inverse semigroup with $0$, and let $\pi \colon S^\times \to \Group$ be an idempotent pure partial homomorphism. Then $\pi$ induces a partial action $\theta=(\theta_g\colon D_{g^{-1}}\to D_g)_{g\in \Group}$ of $\Group$ on $\E = \E(S)$ and a natural isomorphism $S\cong \E \rtimes_\theta^r \Group$. The ideal $D_g$ is defined as the set of all the elements $ss^{-1}$, where $s\in S^\times$ and $\pi (s) = g$, together with $0$. The partial action is defined by $\theta_g(s^{-1}s) = ss^{-1}$ for $s^{-1}s\in D_{g^{-1}}$, and the isomorphism $\E \rtimes ^r_\theta  \Group \to S$ sends $e\delta_g$ to $s$, where $e=ss^{-1}$ and  $\pi (s) = g$.   
\end{proposition}

\begin{proof}
	The case where $S$ is $E$-unitary is shown in \cite{KelLawson}*{Theorem~3.7}. Notice that adding a zero to $S$ in this case we get a strongly $E^*$-unitary inverse semigroup $S\cup \{0\}$. 
In this case we get from the results in \cite{KelLawson} an isomorphism $\E \rtimes_\theta \Group\to S $, and the partial action of $\Group $ on $\E$ and the isomorphism $\E\rtimes _\theta \Group \to S$ are defined as in the statement. 
 
 To obtain the result in the general strongly $E^*$-unitary case, we can use the above case and the fact that every such inverse semigroup is a Rees quotient of an $E$-unitary inverse semigroup; this is observed in \cite{MilanSteinberg}*{Section~5}\footnote{We thank Benjamin Steinberg for pointing out this result to us.}.
	
More precisely, let $S$ be a strongly $E^*$-inverse semigroup, with an idempotent pure partial homomorphism $\pi \colon S^\times \to \Group$. Observe that $S$ is the ideal quotient of the $E$-unitary inverse semigroup 
	$$T:= \{ (s,g)\in S\times \Group : (s= 0) \text{ or } (s\ne 0 \text{ and } g= \pi (s))\}$$
	by the ideal $I=\{0\}\times \Group$. Using this, the result now follows from the $E$-unitary case.
\end{proof}

\section{The inverse semigroup of a separated graph}
\label{sect:InvSemSG}

In this section we generalize to the realm of separated graphs the notion of the inverse semigroup of a graph introduced by Ash and Hall \cite{AshHall}.  

\begin{definition}
	\label{def:graphsemigroup}
    Let $(E,C)$ be a separated graph. The \emph{semigroup} of $(E,C)$ is the universal semigroup $\S(E,C)$ 
    generated by $E^0\cup \hat{E}^1=E^0\cup E^1\cup E^{-1}$ subject to the following relations:
	\begin{enumerate}
		\item $vw\equiv\delta_{v,w}v$ for all $v,w\in E^0$;
		\item $s(x)x\equiv x$ for all $x\in \hat{E}^1$;
		\item $xr(x)\equiv x$ for all $x\in \hat{E}^1$;
		\item $e^{-1}f\equiv \delta_{e,f} r(e)$ for all $e,f\in X$ with $X\in C$.  
	\end{enumerate}
    The \emph{inverse semigroup} of $(E,C)$ is the universal inverse semigroup $\IS(E,C)$ generated by $E^0\cup \hat{E}^1$ 
    subject to the same above relations (1)--(4).
\end{definition}

We defined the semigroups $\S(E,C)$ and $\IS(E,C)$ above via presentations, that is, by specifying generators and relations. The existence of those follows by standard universal algebra, see \cite{lawson}*{Section~2.3} for more details on this. Concretely, one can describe the above semigroups as quotients of free semigroups. More specifically, let $\SI(E)$ be the free semigroup with zero on the set $E^0\cup \hat{E}^1$. Notice that $\SI(E)$ carries a unique involution that extends the identity map on $E^0$, i.e. $v^{-1}=v$ are self-adjoint for all $v\in E^0$, as well as the zero $0=0^{-1}$, and sends $E^1\ni e\to e^{-1}\in E^{-1}$. The semigroup $\S(E,C)$ is then the quotient semigroup $\SI(E)/\rho$, where $\rho $ is the congruence on $\SI(E)$ generated by the relations given by (1)-(4) above; the involution on $\SI(E)$ factors through $\S(E,C)$ so that $\S(E,C)$ is also a $*$-semigroup. And the inverse semigroup $\IS(E,C)$ is the quotient $\SI(E)/{\sigma}$, where $\sigma$ is the congruence on $\SI(E)$ generated by the relations (1)--(4) plus the following extra relation:
\begin{equation}\label{def:inversegraphsemigroup}
	 \tag{5} (xx^{-1})(yy^{-1})\equiv (yy^{-1})(xx^{-1})\quad\mbox{ for all }x,y\in \SI(E). 
\end{equation}
In principle, in order to get an inverse semigroup, we would also need to mod out by the relation
\begin{equation}\label{def:inversegraphsemigroup1}
	\tag{5'} xx^{-1}x\equiv x\quad\mbox{ for all }x\in \SI(E). 
\end{equation}
But this turns out to follow from (5), the reason being that relation (5') holds on the generators $x\in E^0\cup \hat{E}^1$ from (1)--(4), and then a standard argument shows that the commutativity relation (5) implies (5') for all $x\in \SI(E)$. 

It will be useful to work within the path *-semigroup $\mathcal P (E)$, which is the quotient $*$-semigroup $\SI (E)/\theta$, where $\theta$ is the congruence on $\SI (E)$ generated by the relations (1)-(3). The nonzero elements of $\mathcal P (E)$ are exactly the paths on the extended graph $\hat{E}$, and the product of two paths $\lambda $ and $\mu$ is given by its concatenation $\lambda \mu$ if $r(\lambda)= s(\mu)$, or $0$ is $r(\lambda) \ne s(\mu)$.  
In what follows, we will be mainly interested in the inverse semigroup $\IS(E,C)$ and will generally work within the path $*$-semigroup $\mathcal P (E)$, and thus $\lambda\equiv \mu$ will indicate that the elements $\lambda,\mu\in \mathcal P (E)$ represent the same element of $\IS(E,C)$. Clearly, $\IS(E,C)$ is also a quotient $*$-semigroup of $\S(E,C)$, namely, it is the quotient of $\S(E,C)$ by the congruenge generated by (5).

Let us consider some basic examples. 

\begin{example}
\label{exam:non-separated-first-one}
    First of all, if a graph $E$ is trivially separated, i.e., if $C$ consists of $C_v=s^{-1}(v)$ for $v\in E^0$, then $\S (E,C)=\IS (E,C)$ is already an inverse semigroup; it coincides with the {\it graph inverse semigroup} already studied before by many authors, see e.g. \cites{AshHall, JonesLawson, LaLonde-ConditionK, meakin-wang-2021, meakin-milan-wang-2021, mesyan-mitchell-2016}. We will denote in this paper the graph inverse semigroup of $E$ by $\mathcal S(E)$. For instance, for the graph $E$ with a single vertex $v$ and a single edge $e$ endowed (necessarily) with the trivial separation $C$, the inverse semigroup $\mathcal S(E)$ is the universal semigroup with zero $0$ and with $1=v\in \mathcal S(E)$ playing the role of the unit, and generated by a single element $e$ satisfying the relation $e^{-1}e=1$. In this case, there are no orthogonality relations involved, and $\mathcal S(E)^\times$ is also an inverse semigroup, known as the bicyclic monoid. It can be realized concretely as the semigroup of operators on $\ell^2(\N)$ generated by the shift operator $s\in \L(\ell^2(\N))$.  
\end{example}

For non-trivially separated graphs, $\S(E,C)$ is usually not an inverse semigroup and $\IS(E,C)$ is a proper quotient of $\S(E,C)$; indeed, the difference between both semigroups is usually huge. Let us now consider a class of non-trivial examples to get a better feeling for this.

\begin{example}
 \label{exam:free-separation-first-one}
Let $X$ be any non-empty set, and let $E_X$ be the graph with just one vertex $v$ and with $E_X^1 = X$. We consider the free separation $C=C_v= \{\{e\} : e\in X \}$ on $E_X$. 
In this case we do not need the zero element, since $\IS (E_X,C)^{\times}$ and  $\S (E_X,C)^\times $ are already semigroups.
    If $|X| >1$, we have $\S (E_X,C) \ne \IS (E_X,C)$. Indeed take two different edges $e,f\in X$. In the semigroup $\S (E_X,C)$ we have the relations $e^{-1}e=f^{-1}f=1$, but the elements $ee^{-1}ff^{-1}$ and  $ff^{-1}ee^{-1}$ are different elements in $\S (E,C)$, although they become equal in the inverse semigroup $\IS(E,C)$, by definition. To see that $ee^{-1}ff^{-1}\not= ff^{-1}ee^{-1}$ in $\S (E,C)$ one can use for instance a concrete representation of $\S (E_X,C)$ on a Hilbert space $H$, sending $e$ and $f$ to two isometries $s$ and $t$ respectively, such that their range projections $p=ss^*$, $q=tt^*$ do not commute. 

We therefore have surjective $*$-homomorphisms
$$\S (E_X,C)^\times \longrightarrow \IS (E_X,C)^\times \longrightarrow \F (X),$$
where $\F (X)$ is the free group on $X$. The first homomorphism $\S (E_X,C)^\times \to \IS (E_X,C)^\times$ is injective if and only if $|X| = 1$. The second homomorphism 
 $\IS (E_X,C)^\times \to\F (X)$ is never injective. We will give a description of this homomorphism in Example \ref{exam:free-separation-second-one}. 
\end{example}

We now show how to realize the free inverse monoid from the inverse semigroup of a certain separated graph. Figure~\ref{fig0} illustrates the case of the free inverse monoid on two generators.  
    
\begin{example}
\label{exam:free-inverse-monoid}
Let $X$ be any non-empty set. Here we realize the free inverse monoid $\FIM (X)$ as a corner of the inverse semigroup of a certain separated graph $(F_X, D)$. 
Let $F_X$ be the graph with $F_X^0 = \{ v\} \sqcup X$ and $F_X^1 = \{e_x,f_x : x\in X \}$, with $s(e_x) = s(f_x) = v$ and $r(e_x) = r(f_x) = x$ for all $x\in X$.
We take the free separation $D$ on $F_X$, so that $D_v = \{\{ e_x\},\{f_x\} : x\in X\}$ and $D_x= \emptyset$ for all $x\in X$. 
By the universal property of the free inverse monoid, there exists a unique monoid homomorphism  
$$\varphi \colon  \FIM (X) \longrightarrow v\IS (F_X,D)v, \quad \varphi (x) =  e_xf_x^{-1} \,\, \text{ for all }\, x\in X .$$
We will see in Example~\ref{exam:free-inverse-two} that the map $\varphi$ is an isomorphism onto
$(v\IS (F_X,D)v)^\times $, so that $(v\IS (F_X,D)v)^\times $ is a monoid (with identity $v$) isomorphic to the free inverse monoid $\FIM (X)$. 
\end{example}
   
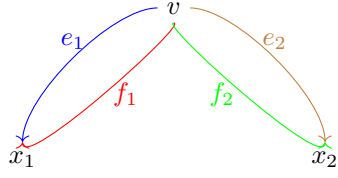
\begin{figure}[htb]
	\begin{tikzpicture}[scale=2]
		\node (v) at (1,1)  {$v$}; \node (x_1) at (0,0) {$x_1$};
		\node (x_2) at (2,0) {$x_2$};
		\draw[->,blue]  (v.west) ..  node[above]{$e_1$} controls+(left:3mm) and +(up:3mm) ..
		(x_1.north) ;
		\draw[->,red] (v.south) .. node[below, right]{$f_1$}  controls+(right:1mm) and +(down:2mm) ..
		(x_1.north);
		\draw[->,green] (v.south) .. node[below, left]{$f_2$}
		controls+(left:1mm) and +(down:2mm) ..
		(x_2.north);
		\draw[->,brown] (v.east) .. node[above]{$e_2$}
		controls+(right:3mm) and +(up:3mm) ..
		(x_2.north);
 \end{tikzpicture}
\caption{The separated graph for the free inverse monoid $\FIM (X)$ in the case where $|X| = 2$.}
        \label{fig0}
\end{figure}

We now proceed to develop a theory of ``Scheiblich normal form'' for the inverse semigroup 
$\IS(E,C)$ of a separated graph $(E,C)$ which is closely related to the similar theory for free inverse semigroups, see \cite{lawson}*{Section~6.2}. This will be connected with the geometric concept of a Munn $(E,C)$-tree in Section \ref{sect:Munntrees}.

Let $(E,C)$ be a separated graph, and let $\F = \F(E^1)$ be the free group on $E^1$. This is the set of all {\it reduced words} in $\SI (\hat{E}^1) \subseteq \SI (E)$ plus the empty word, that plays the role of the identity $1\in \F$.
Here $\SI (\hat{E}^1)$  is the free semigroup on $\hat{E}^1= E^1\cup E^{-1}$, and a reduced word is an element of $\SI (\hat{E}^1)$  which does not contain subwords of the form $xx^{-1}$, where $x\in E^1\cup E^{-1}$.
The product in $\F$ is defined by $s\cdot t = \red (st)$, where $\red (z)$ is the (unique) reduction of an element $z\in \SI (\hat{E}^1)$ to a reduced word. 

As said before, it will be useful here to work within the path $*$-semigroup $\mathcal P (E)$, whose nonzero elements are exactly the paths in $\hat{E}$. We now introduce an object that is a path analogue of the free group $\F = \F(E^1)$ that we considered before.
The construction resembles very much the one of $\F$ as reduced words in $\SI(\hat{E})$, but this time we will get only an inverse semigroup as the target of our construction.

\begin{definition}
    \label{def:FP} Let $E$ be a directed graph and let $\mathcal P (E)$ be the path $*$-semigroup of $E$. Let $\FP (E)$ be the set of all reduced paths in $\mathcal P(E)$, including the vertices of $E$, together with $0$. Here a reduced path is any path not containing subpaths of the form $xx^{-1}$, where $x\in E^1\cup E^{-1}$. We endow $\FP(E)$ with a product $\cdot$ given by 
    $$g\cdot h = \mbox{red} (gh),$$
    where for a path $z$, $\mbox{red} (z)$ is the unique path obtained from $z$ by using the reductions $xx^{-1} \rightarrow  s(x)$ for $x\in E^1\cup E^{-1}$, and $\mbox{red}(0)=0$. With this product $\FP (E)$ is an inverse semigroup with zero, and its nonzero idempotents are exactly the vertices of $E$. We have $g\cdot g^{-1} = s(g)$ and $g^{-1} \cdot g = r(g)$ for all $g\in \FP(E)^{\times}$. We call $\FP (E)$ the {\it fundamental inverse semigroup} of $E$. Observe that $\FP (E)^{\times}$, endowed with the partial product $g \ast  h = g\cdot h$ if $r(g)= s(h)$, is just the {\it fundamental groupoid} $\FG(E)$ of $E$. Hence $\FP (E)$ is the inverse semigroup obtained from the fundamental groupoid of $E$ by declaring any non-defined product to be $0$.   
    We have an obvious idempotent pure partial homomorphism $\omega  \colon \FG (E)= \FP(E)^\times \to \F$, hence we get from Proposition \ref{prop:KelLawson-result} a decomposition $\FP (E) = (E^0\cup \{0\})\rtimes^r_\mathfrak{t} \F$ as a restricted semidirect product of a partial action $\mathfrak{t}$ of $\F$ on $E^0\cup \{0\}$.
    \end{definition}

General nonzero elements of $\mathcal P(E)$ will be called \emph{paths}, and nonzero elements of $\FP(E)$ will be called \emph{reduced paths}. 

A basic ingredient for the Scheiblich normal form is the notion of a $C$-separated path: 

\begin{definition}
	\label{def:Cseparatedword-and-compatible}
	A {\it $C$-separated path} is a reduced path $w=y_1\cdots y_n\in \FP(E)$, with $y_i\in \hat{E}^1$, that does not contain any subpath of the form $x^{-1}y$ with $x,y\in X$ for $X\in C$. We denote the set of all $C$-separated paths of $\FP(E)$ by $\FC$. For $v\in E^0$, set
	$$\FC (v) = \{ w\in \F_C \mid s(w)= v \} .$$
	Note that $\FC = \bigsqcup_{v\in E^0} \FC (v)$. We define the \emph{prefix order} $\le_p$ on $\FC$ by $\mu\leq _p\nu$ if there is $\gamma\in \mathcal P (E)$ such that $\nu=\mu\gamma$. Note that in this case we automatically have $\gamma \in \FC (r(\mu))$. 
\end{definition}

We now come to the crucial concept of $C$-compatibility of $C$-separated paths.
Intuitively two $C$-separated paths $\gamma, \nu \in \FC(v)$ are $C$-compatible if the corresponding idempotents $\gamma \gamma^{-1}$ and $\nu\nu^{-1}$ have a nonzero product in $\IS(E,C)$. 

\begin{definition}
	\label{def:Ccompatible}
	Let $v\in E^0$ and $\gamma,\nu\in \FC (v)\setminus \{v\}$. We say that $\gamma $ and $\nu$ are {\it $C$-compatible} in case $\red (\nu^{-1}\gamma) \in \FC$. Otherwise $\gamma $ and $\nu$ are said to be {\it $C$-incompatible}. In addition, we say that the vertex $v\in \FC (v)$ is $C$-compatible with all the elements of $\FC (v)$.
	A subset $T$ of $\FC(v)$ is said to be $C$-compatible if each pair of elements  of $T$ is $C$-compatible. 
	\end{definition}
	
\begin{remark}
	\label{rem:onCcompatibility}
Note that $\gamma $ and $\nu$ are $C$-incompatible if and only if 
	we can write $\gamma = ux\gamma '$ and $\nu = uy\nu'$, where $u$ is the largest common prefix of $\gamma $ and $\nu$, and $x,y$ are distinct elements of some set $X\in C$. It is clear that $C$-compatibility is a reflexive and symmetric relation. In terms of the Cayley graph of $\FG (E)$, two paths $\gamma $ and $\nu$ are $C$-compatible if the geodesic path joining $\gamma $ and $\nu$ is a $C$-separated path.   
	\end{remark}

Note also that two elements $g,h\in \FC (v)$ are $C$-compatible exactly when the product $g^{-1}\cdot h$ belongs to $\FC$.   	 

	For a non-empty subset $A$ of $\FC$, we shall write $A^{\downarrow}=\{x\in \FC: x\leq_p a,\, \mbox{for some }a\in A\}$ for the lower subset  generated by $A$ with respect to the prefix order. If $A=\{a\}$ is a singleton, we shall also write $a^\downarrow:=\{a\}^\downarrow$ to simplify the notation.

We also need the concept of $C$-separated string.

\begin{definition}
	\label{def:Cseparatedstring}
	A \emph{$C$-separated string} is a path $\gamma = y_1y_2\cdots y_n$ in $\mathcal P (E)$, where $y_i\in E^1\cup E^{-1}$, such that it does not contain any subpath of the form $e^{-1}f$, where $e,f\in X$ for $X\in C$. Paths of length $0$, i.e., vertices $v\in E^0$ are also considered as $C$-separated strings.
	
	Note that the only difference between a non-trivial $C$-separated string and a non-trivial $C$-separated path is that a non-trivial $C$-separated string may contain subwords of the form $ee^{-1}$, where $e\in E^1$.  
	\end{definition}

The following lemma gives a first approximation to the normal form of elements in $\IS(E,C)$ that we will soon obtain. 

\begin{lemma}
	\label{lem:Scheiblich-form}
	Let $\gamma$ be a path in $\mathcal P (E)$ representing an element in $\IS(E,C)^\times$. Then 
	$$\gamma \equiv (\gamma_1\gamma_1^{-1})\cdots (\gamma_n \gamma _n^{-1})\lambda$$
	where $\gamma_i$ and $\lambda$ are $C$-separated paths, $\lambda\lambda^{-1}\gamma_j\gamma_j^{-1}= \gamma_j\gamma_j^{-1}$ for some $j$, and $\{\gamma_1,\dots ,\gamma_n\}$ is a $C$-compatible family.
\end{lemma}

\begin{proof} Let $\gamma $ be a path in $\mathcal P (E)$.
	Applying the reductions (4) in Definition \ref{def:graphsemigroup}, we will arrive, after a finite number of steps, to either a $C$-separated string or to $0$.  
	
	We can thus assume that $\gamma $ is a $C$-separated string in $\mathcal P (E)$.  
	 If $\gamma $ is a vertex $v\in E^0$, then we arrive at the desired form, namely $v\equiv (vv^{-1})v$, with $vv^{-1}(vv^{-1}) = vv^{-1}$. If $\gamma$ is already a $C$-separated path, then $\gamma \equiv  (\gamma \gamma^{-1})\gamma$ is the desired decomposition. If $\gamma $ is not reduced, then let $\gamma_1$ be the longest prefix of $\gamma$ which is reduced as it stands. Obviously, $\gamma_1$ is a $C$-separated path.
	Then we have $\gamma =\gamma_1\nu$, where $\gamma_1,\nu$ are both non-trivial strings. Let $\delta$ be the longest suffix of $\gamma_1$ such that $\delta^{-1}$ is a prefix of $\nu$. Then $\gamma_1 = \gamma_1' \delta$ and $\nu= \delta^{-1}\nu'$ for some paths $\gamma_1 ',\nu'$, so that $\gamma= \gamma_1'\delta\delta^{-1}\nu'$. But $ \gamma_1'\delta\delta^{-1}\nu'\equiv (\gamma_1'\delta \delta^{-1} (\gamma_1')^{-1})(\gamma_1'\nu')= \gamma_1\gamma_1^{-1}(\gamma_1'\nu')$. Now observe that $\gamma_1'\nu'$ is a string whose length is strictly less than the length of $\gamma$, and moreover $s(\gamma_1'\nu')= s(\gamma)$. The string  $\gamma_1'\nu'$ is not necessarily $C$-separated but, in case it is not $C$-separated, we can reduce it using relation (4) in Definition \ref{def:graphsemigroup} either to a shorter $C$-separated string or to $0$. Hence, proceeding by induction, and assuming that $\gamma $ does not represent $0$ in $\IS(E,C)$, we will obtain a decomposition
	$$\gamma \equiv (\gamma_1\gamma _1^{-1})\cdots (\gamma _n\gamma_n^{-1})\lambda,$$ 
	where $\gamma_1,\dots , \gamma_n$ are non-trivial mutually distinct $C$-separated paths, with $s(\gamma _i)= s(\gamma_j)$ for all $i,j$, and with $\lambda = \gamma_n$ or $\lambda = s(\gamma)$. Suppose that $\gamma _i$ and $\gamma _j$ are not $C$-compatible. Then $\gamma_i = \eta x\tau_i$ and $\gamma_j = \eta y\tau_j$, for some paths $\eta,\tau_i,\tau_j$, with $x,y\in X$ for some $X\in C$ and $x\ne y$. But then
	\begin{align*}
	(\gamma_i\gamma_i^{-1})(\gamma_j\gamma_j^{-1}) & \equiv  \eta (x\tau_i \tau_i^{-1}x^{-1}) (\eta^{-1}\eta) y \tau_j \gamma_j^{-1} \\
	& \equiv \eta (\eta^{-1}\eta)(x\tau_i \tau_i^{-1}x^{-1}) y \tau_j \gamma_j^{-1} \\
	& \equiv  \eta (x\tau_i \tau_i^{-1})(x^{-1} y) \tau_j \gamma_j^{-1}  \\
	& \equiv 0,
\end{align*}
so that $\gamma \equiv 0$. Therefore the family $\{ \gamma_1,\dots, \gamma_n\} $ must be $C$-compatible.       	  
\end{proof}

We can thus obtain the Scheiblich normal form as follows.

\begin{proposition}
	\label{prop:Scheiblich-form}
	Let $\gamma$ be a string in $\mathcal P (E)$ representing an element in $\IS(E,C)^\times$. Then we have 
	$$\gamma \equiv (\gamma_1\gamma_1^{-1})\cdots (\gamma_n \gamma _n^{-1})\lambda$$
	where $\{\gamma_1,\dots ,\gamma_n\}$ is an incomparable $C$-compatible family of $C$-separated paths, $\lambda$ is a $C$-separated path, the paths $\gamma_j$ do not end in $E^{-1}$, $s(\gamma_j)= s(\gamma)$ for $j=1,\dots ,n$,  and $\gamma_i\gamma_i^{-1}\equiv (\lambda \lambda^{-1}) (\gamma_i\gamma_i^{-1})$ for some $i$.
\end{proposition}

\begin{proof}
	Let 
	$$\gamma \equiv (\nu_1\nu_1^{-1}) \cdots (\nu_r \nu_r^{-1}) \lambda$$
	be the decomposition obtained in 
	Lemma \ref{lem:Scheiblich-form}. Write $\nu_i = \nu_i'w_i$, where $\nu_i'$ does not end in $E^{-1}$, and $w_i$ is a (possibly empty) product of inverse edges $x^{-1}\in E^{-1}$. Note that $\nu_i \nu_i^{-1} \equiv  (\nu_i')(\nu_i')^{-1}$ for all $i$. Moreover, since each $\nu_i'$ is an initial segment of $\nu_i$, the family $\{\nu_1',\dots , \nu_r'\}$ is $C$-compatible. Observe that if $\nu_i'\le_p \nu_j'$, then with $\nu_j'= \nu_i'\nu_j''$, we have
	$$(\nu_i'(\nu_i')^{-1})(\nu_j'(\nu_j')^{-1})= (\nu_i'(\nu_i')^{-1}\nu_i')(\nu_j''(\nu_j')^{-1}) \equiv \nu_i'\nu_j''(\nu_j')^{-1}= \nu_j'(\nu_j')^{-1}$$
	Let $\{\gamma_1,\dots ,\gamma _n\}= \max \{\nu_1',\dots , \nu_r'\}$, where the max is taken in the prefix order $\le_p$, that is, $\{ \gamma_1,\dots ,\gamma_n\}$ is the family of maximal elements of the family $\{ \nu_1',\dots , \nu_r'\}$ with respect to $\le_p$. Then for each $j=1,\dots , r$ there exists $i\in \{1,\dots ,n\}$ such that $\nu_j' \le_p \gamma_i$. Hence, using the above computations, we get
	$$\gamma \equiv (\nu_1\nu_1^{-1})\cdots (\nu_r\nu_r^{-1})\lambda \equiv 
	(\nu_1'(\nu_1')^{-1})\cdots (\nu_r'(\nu_r')^{-1})\lambda \equiv (\gamma_1\gamma_1^{-1})\cdots (\gamma_n\gamma_n^{-1}) \lambda.$$
	Finally there is some $j$ such that $\lambda = \nu_j$ and there is some $i$ such that $\nu_j'\le_p \gamma _i$, hence
	$$(\lambda \lambda^{-1})(\gamma _i\gamma_i^{-1})= (\nu_j\nu_j^{-1}) (\gamma _i\gamma_i^{-1}) \equiv  (\nu_j'(\nu_j')^{-1}) (\gamma _i\gamma_i^{-1}) \equiv \gamma_i \gamma_i^{-1},$$
	as desired. 
\end{proof}

 Note that, combining the proofs of Lemma \ref{lem:Scheiblich-form} and of Proposition \ref{prop:Scheiblich-form} we get an algorithm to find a Scheiblich normal form. We will see soon that this form is unique, and that all elements in Scheiblich normal form are nonzero in $\IS(E,C)$, thus solving the word problem for these inverse semigroups. 

	We will now construct a set $\Y$ that will parametrize the idempotents of the inverse semigroup $\IS(E,C)$; it is the analogue of the set denoted in the same way considered in the study of the free inverse semigroup in \cite{lawson}*{Section~6.2}. 
		
	For $v\in E^0$, consider, with respect to the prefix order, the set $\mathcal I (v)$ of non-empty finite subsets $I$ of $\FC (v)$ such that the elements of $I$ are  pairwise incomparable and pairwise $C$-compatible. Consider also the set $\Y (v)$ of all non-empty finite lower $C$-compatible subsets of $\FC (v)$. (Here a lower subset of $\FC(v)$ is a subset $A$ of $\FC (v)$ such that $\lambda \in A$ whenever $\lambda \le _p \mu$ and $\mu \in A$.) We define
	$$\Y:= \bigsqcup _{v\in E^0} \Y(v) \bigsqcup \{0\}.$$
	Then $\Y$ is a semilattice with respect to the partial order $\le $ defined by $A\le B$ if $A,B\in \Y (v)$ for some $v\in E^0$ and $B\subseteq A$, and $0\le A$ for each $A\in \Y$. Notice that the meet in $\Y$ is given by
	$$
	A\wedge B = \begin{cases} A\cup B & \text{ if } A,B\in \Y(v) \text{ for some } v\in E^0 \text{ and } A\cup B \text{ is }C \text{-compatible} \\\quad  0 & \text{ otherwise  } .	\end{cases} $$
	Observe that 
$\Y (v)\perp \Y (w)$ for distinct $v,w\in E^0$. We will denote by $\Y^\times$ the set of all nonzero elements of $\Y$.    
 
	\begin{proposition}\label{prop:bijection-max-lower-sets}
		There is a bijection between $\Y (v)$ and $\mathcal I (v)$ sending $A\in \Y (v)$ to $\max (A)\in \mathcal I (v)$, where $\max(A)$ denotes the set of maximal elements of $A$. The inverse bijection sends $A\in \mathcal I (v)$ to $A^{\downarrow}\in \Y (v)$. 
	\end{proposition}  
\begin{proof}
	It is clear that if $\gamma_1\le_p \gamma $, $\nu_1\le_p \nu$ and $\gamma $, $\nu$ are $C$-compatible, then $\gamma _1$ and $\nu_1$ are also $C$-compatible. Hence the proposition is clear, and the proof follows the lines of the one in \cite{lawson}*{Proposition 6.7}.
\end{proof}
	
For $I\in \Y$, we denote by $\PI I$ the principal order ideal of $\Y$ generated by $I$. Notice that with respect to the order of $\Y$, the principal ideal $\PI I$ also equals its associated lower set, but we avoid the use of the lower set notation here because we want to reserve that for the prefix order of $\FC$, as in the following definition, where $g^\downarrow$ means the lower set of $\{g\}$ with respect to the prefix order for $g\in \FC$. Notice that $g^\downarrow\in \Y$.

\begin{definition}
\label{def:first-expansion}
	Let $(E,C)$ be a separated graph, and consider the congruence $\sim$ on $\Y$ generated by
	\begin{equation}
		\label{eq:EXP1}
		g^\downarrow \sim (gx^{-1})^{\downarrow}
			\end{equation}
	for any $g\in \FC$ and any $x\in E^1$ such that $\red (gx^{-1})= gx^{-1}$ (that is, $g$ does not end with $x$). Note that $gx^{-1}$ also belongs to $\FC$.   
\vskip 0,5pc
We denote the quotient semilattice $\Y/{\sim}$ by $\YY$, and the class of $I\in \Y$ in $\YY$ by $[I]$. 
\end{definition}

 The idea for the above congruence relation is the following: at the end we want to describe idempotents of our inverse semigroup $\IS(E,C)$ as certain finite products of the form $\alpha_1\alpha_1^{-1}\cdots\alpha_n\alpha_n^{-1}$ for certain words $\alpha_i$. And if a word $\alpha$ has the form $\alpha=\beta e^{-1}$ with $e\in E^1$ and $r(\beta)=r(e)$, then we can simplify $\alpha\alpha^{-1}$ to $\beta\beta^{-1}$ in $\IS(E,C)$ using $e^{-1}e\equiv r(e)$.

\begin{notation}\label{not:prefix-dead-ends}
    Each element $g\in \FG(E)$ can be uniquely expressed as a reduced product $g= g_0w$, where $g_0$ is a prefix of $g$ such that $g_0$ does not end in $E^{-1}$ and $w$ is a (possibly empty) product 
	of terms in $E^{-1}$.  
    Notice that $g_0\leq_p g$, and in this situation we also say that $g$ extends $g_0$.

    If $0\not=I\in \Y$ and $g\in I$, then $g_0\in I$ as $I$ is a lower set. 
    Let $A$ be the set of maximal elements of $\{g_0 : g\in I\}$ with respect to $\le_p$, and set $I_0:= A^{\downarrow}$. If $I=0$, we set $I_0:=0$.  In any case, $I_0\in \Y$ and $\max (I_0)= A$ by Proposition~\ref{prop:bijection-max-lower-sets}, so all elements of $\max (I_0)$ do not end in $E^{-1}$.
\end{notation}

\begin{lemma}\label{lem:maxelements}
	Let $I,J\in \Y^{\times}$. Then $[I]=[J]$ in $\YY$ if and only if $I_0=J_0$, and $[I]\le [J]$ in $\YY$ 
    if and only if $J_0 \subseteq I_0$. 

	Moreover, for each non-zero equivalence class $[I]\in \YY$, the element $I_0\in \Y$ is the maximum representative of the class $[I]\in \YY$.
 
    Similarly, for each $g\in \FC$ such that $g\in I$ for some $I\in\Y$, there exists a maximum element $I_g\in [I]$ such that $g\in I_g$, namely $I_g=I_0 \wedge g^\downarrow=
	I_0 \cup   g^\downarrow$. 
	\end{lemma}
	\begin{proof}
First observe that $I\le I_0$, i.e., that $I_0 \sbe I$, is clear as $I$ is a lower set. We now check that $I_0\sim I$. 
	Indeed observe that $g^\downarrow\sim g_0^\downarrow$ for each $g\in I$, so that setting $A:=\max\{g_0:g\in I\}$, we get
	$$I= \bigcup_{g\in I} g^\downarrow \sim \bigcup_{g\in I} g_0^\downarrow =\bigcup_{h\in A} h^\downarrow = I_0.$$
	Hence $I\sim I_0$. To show that $I_0$ is the maximum representative of $[I]$, it is enough to show that if $I\in \Y$, $g\in I$, and $x\in E^{1}$ is such that $gx^{-1}$ is reduced as it stands, 
 then $(I\cup (gx^{-1})^{\downarrow})_0 = I_0$. This is clear if $gx^{-1}\in I$. If $gx^{-1}\notin I$, then 
 $I\cup (gx^{-1})^{\downarrow} = I \sqcup \{gx^{-1}\}$. If $g=g_0w$ if the decomposition of $g$ as in Notation~\ref{not:prefix-dead-ends}, then $gx^{-1} = g_0(wx^{-1})$ is the corresponding decomposition of $gx^{-1}$, so that 
 $\{h_0 :h\in I\}= \{h_0: h\in I\cup (gx^{-1})^{\downarrow}\}$. Therefore $I_0$ is the maximum representative of the class $[I]$, and in particular this shows that $[I]=[J]$ if and only if $I_0=J_0$. 
 
 We now check that $J_0\subseteq I_0$ if and only if $[I]\le [J]$. If $J_0\subseteq I_0$ then $[I]=[I_0]\le [J_0] = [J]$. Conversely if $[I]\le [J]$ then $[I] = [I]\wedge [J] = [I\wedge J]$. In particular observe that $I\cup J$ must be $C$-compatible, because $I\ne 0$ by hypothesis, and $[I] = [I\cup J]$. By our previous result, this implies that $I_0=(I\cup J)_0$, hence
 $J_0\subseteq (I\cup J)_0  = I_0$ and we get that $J_0\subseteq I_0$, as desired. 
 
Finally, let $g\in I$ for some $I\in \Y$. Then $I_0\subseteq I$ by the first part of the proof, and thus 
$$I_g= I_0\cup g^\downarrow \subseteq I,\mbox{ that is, } I\leq I_g.$$ 	
This shows the last assertion of the lemma.    
	\end{proof}

The above result shows that the quotient homomorphism $\pi\colon \Y\onto \YY$, $I\mapsto [I]$, admits a right inverse homomorphism $\sigma\colon \YY\to \Y$ given by $\sigma([I]):=I_0$. Therefore we may also view the quotient semilattice $\YY$ as a subsemilattice of $\Y$, and in particular, elements of $\YY^{\times}$ may be viewed as certain non-empty finite lower subsets of $\FC$. Or, using Proposition~\ref{prop:bijection-max-lower-sets}, we may also view elements of $\YY^{\times}$ as certain incomparable non-empty finite subsets of $\FC$.

We introduce below a useful notation for the image $\sigma (\YY)$ of the above map $\sigma$.

\begin{notation}
	\label{notati:Y-sub-0}
	Let $\Y _0(v)$ be the set of those $I\in \Y(v)$ such that each element in $\max (I)$ does not end in $E^{-1}$ (so that $I_0=I$). Set $\Y_0 = \bigcup_{v\in E^0} \Y _0(v)$.	
\end{notation}

\section{Munn trees}
\label{sect:Munntrees}

Recall that an {\it inverse category} is a category $C$ such that for each arrow $\gamma$ there is a unique arrow $\gamma^{-1}$ such that $\gamma = \gamma \gamma^{-1}\gamma$ and $\gamma^{-1} = \gamma^{-1} \gamma \gamma^{-1}$.

For each graph $E$, there is a free inverse category $\FIC(E)$. It is the quotient of the free category over $\hat{E}$ by the congruence generated by the relations $\gamma \sim \gamma \gamma ^{-1}\gamma$ and $(\gamma \gamma^{-1}) (\nu\nu^{-1}) \sim (\nu \nu^{-1})(\gamma \gamma^{-1})$ whenever $s(\gamma) = s(\nu)$, see \cite{marg-meakin-93}*{Section~4}.

Given an inverse category $C$ we can form an inverse semigroup $\IS(C)$ by adjoining a zero to $C$ and declaring any undefined product to be $0$.  We denote by $\FIS(E)$ the inverse semigroup $\IS(\FIC(E))$ associated to the free inverse category $\FIC(E)$. We call $\FIS(E)$ the {\it free inverse semigroup} of $E$.

We now relate the elements in $\FIC(E)$ with certain Munn trees. The reader is referred to \cite{Munn} and \cite{lawson}*{Section~6.4} for the basic theory of Munn trees and their connections with the free inverse monoid $\FIM (X)$ on a set $X$. Let $E$ be a directed graph and let $\FG(E)$ be the fundamental groupoid of $E$. Observe that whenever $E$ is the rose with $|E^1|$ petals, this agrees with the free group $ \mathbb F (E^1)$. Now let $\Gamma_E$ be the Cayley graph of the groupoid $\FG(E)$. We have
	$\Gamma_E ^0 =  \FG(E)$, and there is an edge $(g,e)$ from $g$ to $ge$ whenever $g, ge\in \FG(E)$, $e\in E^1$.
	
	Note that $\Gamma_E$ is a forest, that is, a disjoint union of trees. Each connected component $K$ of $\Gamma_E$ contains a unique vertex $v$, and $K$ consists of all the reduced paths $p$ such that $s(p) = v$. We will denote by $K_v$ the connected component of $v$. Observe that for each $v\in E^0$, the connected component $K_v$ is isomorphic to a certain subtree of the Cayley graph $\Gamma$ of the free group $\F$ on $E^1$, namely the subtree consisting in replacing the vertex $v$ by $1\in \F$, and identiying each other vertex of $K_v$ with the corresponding vertex in $\Gamma$.  	
	
	There is a natural partial action by left translations of $\FG(E)$ on the graph $\Gamma_E$.

Let $(E,C)$ be a separated graph. Observe that, for $v\in E^0$, elements in $\Y (v)$ can be interpreted as certain finite subtrees of $K_v$ containing the root $v$ of the connected component $K_v$. Concretely, the finite subtrees $T$ which belong to $\Y(v)$ are characterized by the property that the geodesic path between any two vertices of $T$ is a $C$-separated path (see Remark \ref{rem:onCcompatibility}).

\begin{definition}
	\label{def:Mun-EC-tree} Let $E$ be a directed graph.
	\begin{enumerate}[(a)]
		\item 	A {\it Munn $E$-tree} is a pair $(T,g)$, where $T$ is a finite connected subgraph (hence a subtree) of $\Gamma_E$ containing the unit $v$ of the connected component $C_v$ corresponding to $T$, and $g$ is a vertex of $T$.
      \item We define a {\it Toeplitz-Munn $E$-tree} as a pair $(T,g)$, where $T$ is a finite 
       connected subgraph of $\Gamma_E$ containing the corresponding unit $v$ and such that all maximal elements of $T$ do not end in $E^{-1}$, and  $g= g_0 w$ satisfies that $g_0\in T$ (see Notation \ref{not:prefix-dead-ends}).  
       \item Suppose now that $(E,C)$ is a separated graph. A {\it Munn $(E,C)$-tree} is a pair $(T,g)$, where $T\in \Y_0(v)$ (see Notation \ref{notati:Y-sub-0}) and $g= g_0 w\in \FC (v)$ satisfies that $g_0\in T$, for some $v\in E^0$ 
	\end{enumerate} 
  	\end{definition}

We now obtain a concrete description of the free inverse semigroup $\FIS(E)$ using Munn $E$-trees.

\begin{proposition}
	\label{prop:MunnE-trees}
	Let $E$ be a directed graph. Then the free inverse semigroup $\FIS(E)$ associated to $E$ is isomorphic to the semigroup $\mathcal M$ of all Munn $E$-trees, together with $0$, endowed with the product 
	$$(T_1,g_1)\cdot (T_2,g_2) = \begin{cases} (T_1\cup g_1\cdot T_2,\,  g_1\cdot g_2) & \text{ if } r(g_1) = s(g_2)\\ \qquad \,\,  \qquad \,\, 0 & \text{ if } r(g_1) \ne s(g_2) 
	\end{cases}  . $$
	The inverse of a Munn $E$-tree $(T,g)$ is given by
	$$(T,g)^{-1} = (g^{-1}\cdot T, g^{-1}).$$ 
\end{proposition}

\begin{proof}
Let $\gamma$ be a path in $\FIS(E)$. The tree $M(\gamma)$ associated to $\gamma$ is built in the same way as the tree associated to an element of the free inverse semigroup on a set $X$. Namely, starting with the vertex $v= s(\gamma)$, $M(\gamma)$ is the finite subtree of $C_v$ obtained by traversing the path $\gamma$ in $C_v$. The Munn $E$-tree $(T,g)$ of $\gamma $ is the pair consisting of $T=M(\gamma)$ and the final vertex $g\in C_v$ that we obtain once we have transversed $\gamma$. Note that $M(\gamma)$ is isomorphic to the tree denoted by $MT(\gamma)$ in \cite{marg-meakin-93}*{Section~4}. 
 Therefore, by \cite{marg-meakin-93}*{Theorem~4.1}, the paths $\gamma$ and $\nu$ represent the same element in $\FIS(E)$ if and only if they have the same Munn $E$-tree. 

Hence we get a bijection $(T,g)\mapsto (\prod_{\lambda\in \max (T)} \lambda\lambda^{-1})g$ between the set $\mathcal M $ of all Munn $E$-trees, together with $0$, and the inverse semigroup $\FIS(E)$. Consequently, we get an inverse semigroup structure in $\mathcal M$. We now check that the formula of the product in $\mathcal M$ is the one given in the statement. Let $(T_1,g_1)$ and $(T_2,g_2)$ be two Munn $E$-trees. We have to compute the element
$$A:= (\prod_{\lambda\in \max (T_1)} \lambda \lambda^{-1})g_1 (\prod_{\mu\in \max (T_2)} \mu \mu^{-1})g_2$$
in $\FIS(E)$. 
Note that $A=  (\prod_{\lambda\in \max (T_1)} \lambda \lambda^{-1}) (\prod_{\mu\in \max (T_2)} g_1\mu \mu^{-1}g_1^{-1})g_1g_2$, so that $A=0$ if $r(g_1)\ne s(g_2)$. Hence we may assume that $r(g_1)= s(g_2)$. Let $\mu$ be any path such that $r(g_1)= s(\mu)$. Write $g_1= g_1'\rho$, $\mu = \rho^{-1}\mu'$, where $g_1'\mu'$ is a reduced path as it stands. We then have
$$g_1\mu = g_1'\rho \rho^{-1}\mu' = g_1'\rho \rho^{-1}(g_1')^{-1}g_1'\mu'=
(g_1g_1^{-1})(g_1\cdot \mu)$$
Using this computation, we get 
\begin{align*}
	A & = (\prod_{\lambda\in \max (T_1)} \lambda \lambda^{-1}) (\prod_{\mu\in \max (T_2)} g_1\mu \mu^{-1}g_1^{-1})g_1g_2 \\
	& = (\prod_{\lambda\in \max (T_1)} \lambda \lambda^{-1})(g_1g_1^{-1}) (\prod_{\mu\in \max (T_2)} (g_1\cdot \mu)(g_1\cdot \mu)^{-1} )(g_1\cdot g_2) \\
	& = (\prod_{\lambda\in \max (T_1\cup g_1\cdot T_2)} \lambda \lambda^{-1} )(g_1\cdot g_2)\, ,
\end{align*}
where we have used that $g_1\in T_1$ for the last equality. This shows the desired formula for the product of Munn $E$-trees. The formula for the inverse is easy to check.
\end{proof}

We can now easily show that $\FIS(E)$ is strongly $E^*$-unitary and hence a restricted semidirect product.

\begin{corollary}
	\label{cor:FSIE-strongly-unitary}
	Let $E$ be a directed graph. Then $\FIS(E)$ is strongly $E^*$-unitary and thus it is a restricted semidirect product 
	$$\FIS(E) =  \E (\FIS(E)) \rtimes_\theta ^r \F,$$
	where $\F = \F (E^1)$ is the free group on $E^1$.  
\end{corollary}
 
\begin{proof}
	Using the picture of $\FIS(E)$ from Proposition \ref{prop:MunnE-trees}, it is clear that the map $\pi \colon \FIS(E)^\times \to  \F$ given by $\pi (T,g)= \omega (g)$, where $\omega \colon \FP(E)^\times \to \F$ is the canonical partial homomorphism (see Definition \ref{def:FP}) is an idempotent pure partial homomorphism. Hence the result follows from Proposition \ref{prop:KelLawson-result}. Note that for $g\in \F^*$, $D_g$ is the ideal of $\E (\FIS(E))$ consisting of those trees $T$ such that $g\in T$. The map $\theta_g\colon D_{g^{-1}}\to D_g$ sends a tree $T$ to its left translation $g\cdot T$ by $g$. 
	\end{proof}

From now on, we will identify the free inverse semigroup $\FIS(E)$ on $E$ with the inverse semigroup $\mathcal M $ of Munn $E$-trees, together with $0$. Let $\E$ be the semilattice of idempotents of $\mathcal M$. The nonzero elements of $\E$ are the finite subtrees of $\Gamma_E$ containing the root $v\in E^0$ corresponding to its connected component. This can be identified with the family of non-empty, finite lower subsets of $\FG(E)$. Let $\sim$ be the congruence on $\FIS(E)\cong \mathcal M$ generated by the relations $e^{-1}e \sim r(e)$ for $e\in E^1$. Following the proof of Lemma \ref{lem:maxelements}, we see that each equivalence class $[T]$, for $T\in \E^\times$, has a maximum representative $T_0$ in $\E$ (with respect to the natural order of $\E$), which is characterized by the property that all the maximal elements of $T_0$ do not end in $E^{-1}$. 

\begin{proposition}
	\label{prop:Toeplitz-Munn-trees}
	Let $\sim$ be the congruence on $\FIS(E)\cong \mathcal M$ generated by the relations $e^{-1}e \sim r(e)$ for $e\in E^1$. Then the inverse semigroup $\mathcal M/{\sim}$ is isomorphic to the inverse semigroup $\mathcal{TM}$ consisting of all the Toeplitz-Munn $E$-trees $(T,g)$, endowed with the product 
$$(T_1,g_1)\cdot (T_2,g_2) = \begin{cases} \Big( (T_1'\cup g_1\cdot T_2')_0, \,\, g_1\cdot g_2\Big) & \text{ if } r(g_1) = s(g_2)\\ \qquad \,\,  \qquad \,\, 0 & \text{ if } r(g_1) \ne s(g_2) 
\end{cases}   $$
where $T_1',T_2'\in  \E$ are representatives of $T_1$ and $T_2$ respectively, such that $g_1\in T_1'$ and $g_2\in T_2'$. The inverse is given by  	
	$(T,g)^{-1} = ((g^{-1}\cdot T')_0, g^{-1})$,
	where $T'$ is a representative of $T$ such that $g\in  T'$. 
	
	Moreover, the map $(T,g)\mapsto \omega (g)$ determines an idempotent pure partial homomorphism $\mathcal{TM} \to \F$. 
	\end{proposition}

	\begin{proof}
		It is clear that the idempotent pure partial homomorphism $\FIS(E)^\times \to \F$ factors
		through an idempotent pure partial homomorphism $\FIS(E)^\times /{\sim} \to \F$. The result follows from Proposition \ref{prop:KelLawson-result}, together with the fact that each 
		 equivalence class $[T]$, for $T\in \E$ contains a maximum representative $T_0$, so that each element in $\FIS(E)^\times/\sim$ is represented by a unique Toeplitz-Munn $E$-tree.  
	\end{proof}
    
We can finally obtain our main result for the inverse semigroup $\IS(E,C)$.

\begin{theorem}
	\label{thm:ECMunntrees}
Let $(E,C)$ be a separated graph. Then the graph inverse semigroup $\IS (E,C)$ is isomorphic to the inverse semigroup consisting of all Munn $(E,C)$-trees, endowed with the following product
$$(T_1,g_1)\cdot (T_2,g_2) = \begin{cases} \Big( (T_1'\cup g_1\cdot T_2')_0, \,\, g_1\cdot g_2\Big) & \text{ if } T_1'\cup g_1\cdot T_2' \text{ is }C\text{-compatible} \\ \qquad \,\,  \qquad \,\, 0 & \text{ otherwise }  
\end{cases}   $$
where $T_1',T_2'\in \Y_0$ are representatives of $T_1$ and $T_2$ respectively, such that $g_1\in T_1'$ and $g_2\in T_2'$.
The inverse is given $(T,g)^{-1} = ((g^{-1}\cdot T')_0,g^{-1})$, where $T'$ is a representative of $T$ such that $g\in T'$. 
The map $(T,g)\mapsto \omega (g)$ provides an idempotent pure partial homomorphism $\IS(E,C)^\times \to \F$, which determines a semidirect product decomposition
$$\IS (E,C) \cong \YY \rtimes_{\theta}^r \F.$$
For each $g\in \FC(v)\setminus \{v\}$, $D_g$ can be identified with the set of all those $T\in \Y_0(v)$ such that $g_0\in T$, and the action $\theta_g$ is defined by $\theta_g(T)= (g\cdot T')_0$, where $T'$ is a representative of $T$ containing $g^{-1}$. 
\end{theorem}

\begin{proof}
	If $S$ is any strongly $E^*$-unitary inverse semigroup with an idempotent pure partial homomorphism $\theta \colon S^\times \to \Group$, and $I$ is an ideal of $S$, then $S/I$ is also strongly $E^*$-unitary with the same group $\Group$.
	This is clear since $(S/I)^\times$ can be identified with $S\setminus I$, and thus the restriction of $\theta $ to $S\setminus I$ gives the desired idempotent pure partial homomorphism. 

Now let $(E,C)$ be a separated graph. Since the inverse semigroup $\IS (E,C)$ is the Rees quotient of the inverse semigroup $\FIS(E)/{\sim}$ by the ideal generated by all the elements of the form $e^{-1}f$, for distinct $e,f\in X$, $X\in C$, it follows from Proposition \ref{prop:Toeplitz-Munn-trees} and the above observation that $S(E,C)$ is a strongly $E^*$-unitary semigroup, with an idempotent pure partial homomorphism $\pi \colon \IS (E,C)^\times \to \F$. The result now follows from Propositions \ref{prop:KelLawson-result}  and \ref{prop:Toeplitz-Munn-trees}.  
\end{proof}

The uniqueness of the Scheiblich normal form follows immediately from the partial semidirect product structure, as follows. 

\begin{theorem}
	\label{thm:uniqueness-SNF} Let $(E,C)$ be a separated graph and let $\gamma $ be a nonzero element of $\IS(E,C)$. Then 
\begin{equation}
	\label{eq:SNF}
	\gamma = (\gamma_1\gamma_1^{-1})\cdots (\gamma_n \gamma _n^{-1})\lambda,
	\end{equation}
		 where $\{\gamma_1,\dots ,\gamma_n\}$ is an incomparable $C$-compatible family of $C$-separated paths, $\lambda \in \FC$, the words $\gamma_j$ do not end in $E^{-1}$, $s(\gamma_j)= s(\gamma)= s(\lambda)$ for $j=1,\dots ,n$,  and $\gamma_i\gamma_i^{-1} = (\lambda \lambda^{-1}) (\gamma_i\gamma_i^{-1})$ for some $i$. The latter condition is equivalent to the fact that  $\lambda_0$ is a prefix of $\gamma_i$ for some $i$, where $\lambda = \lambda_0w$, $\lambda_0$ does not end in $E^{-1}$, and $w$ is a product of elements in $E^{-1}$. Moreover the Scheiblich normal form is unique in the following sense: if $\gamma = (\gamma_1'(\gamma_1')^{-1})\cdots (\gamma_m' (\gamma _m')^{-1})\lambda'$ is another Scheiblich normal form for $\gamma$, then $\lambda= \lambda'$ and $\{\gamma_1,\dots ,\gamma_n\}= \{ \gamma_1',\dots , \gamma_m'\}$. In particular, $n=m$.
		 
		 Conversely, elements $\gamma $ in Scheiblich normal form are nonzero in $\IS(E,C)$. 
\end{theorem} 

Notice that we can describe the natural order of $\IS(E,C)$ using the Scheiblich normal form: Given $\alpha=\alpha_1\alpha_1^{-1}\cdots\alpha_n\alpha_n^{-1}\lambda$ and $\beta=\beta_1\beta_1^{-1}\cdots\beta_m\beta_m^{-1}\lambda'$ in Scheiblich normal forms, then $\alpha\leq \beta$ in $\IS(E,C)$ if and only if $\lambda=\lambda'$ and $\{\alpha_1,\ldots,\alpha_n\}^{\downarrow} \supseteq 
\{\beta_1,\ldots,\beta_m\}$.

\vskip 0,5pc

We also highlight the following consequence of Theorem \ref{thm:ECMunntrees}.

\begin{theorem}
\label{thm:semidirect-product}
Let $(E,C)$ be a separated graph and let $\mathcal E$ be the semilattice of idempotents of $\IS(E,C)$. Then there exists a (zero-preserving) isomorphism $\YY \cong \mathcal E$ given by 
$$[I]\mapsto \prod_{\lambda\in \max (I_0)} \lambda \lambda^{-1},$$
where $I_0$ is the unique maximal element of $[I]$ in the semilattice order (i.e., reverse inclusion), which is characterized by the property that each maximal element of $I_0$ does not end in $E^{-1}$.
\end{theorem}

\begin{proof}
 This follows from Lemma \ref{lem:maxelements} and Theorem \ref{thm:ECMunntrees}.
\end{proof}

We now analyze Examples \ref{exam:free-separation-first-one} and \ref{exam:free-inverse-monoid} in the light of the Munn tree picture.

\begin{example}
    \label{exam:free-separation-second-one}
    Let $(E_X,C)$ be as in Example \ref{exam:free-separation-first-one}. In this case we have $\FC = \FC (v)  = \FP (E_X)^\times =  \F (X) $, that is, all these objects coincide with the free group $\F (X)$ on $E_X^1= X$, with the obvious identification of the single vertex $v$ with $1\in \F (X)$. Moreover all non-empty lower subsets of $\F (X)$ are $C$-compatible, hence the semilattice 
    $\mathcal X$ is exactly the semilattice of all ordinary Munn trees over the free group $\F (X)$, together with $0$, and the free inverse semigroup $\FIS(E_X)$ is just $\FIM (X)\cup \{0\}$. (Note the $0$ is not really needed here, because we only have one vertex.) 
   The semilattice $\mathcal E $ of idempotents of $\IS (E_X,C)$ is therefore isomorphic to the semilattice $\mathcal Y = \mathcal X/{\sim} $, where $\sim $ is the relation on $\mathcal X$ introduced in Definition \ref{def:first-expansion}.   
    By Theorem \ref{thm:ECMunntrees}, we have $\IS (E_X,C) \cong \mathcal E \times^r_{\theta}\F (X)$ and the associated free partial grading 
    $$\IS (E_X,C)^\times \cong (\mathcal E \times^r_{\theta}\F (X))^\times \longrightarrow \F(X)$$
   is surjective and coincides with the homomorphism described in Example \ref{exam:free-separation-first-one}.
   
   The unique Scheiblich normal form of an element in $\IS(E_X,C)$ can be understood through the associated Munn tree. To illustrate this, take two different edges called $e,f \in X$ and consider the element $u=e^2e^{-1}ff^{-1}ef^2f^{-1}ee^{-1}f^{-2} \in \IS(E
   _X,C)$. 
   
   By applying the clever graphical algorithm developed by Munn, we can interpret the Scheiblich normal form as follows: begin at the vertex $v$ and label it as $\alpha$. Then, draw all the edges along the element $u$, leading to the terminal vertex labeled as $\beta$. The resulting drawing is the Munn tree of $u$ (see Figure \ref{fig2}). 
   
   Once the Munn tree for $u$ is constructed, we can easily identify the components of the Scheiblich normal form, both in the free inverse monoid $\FIM (X)$ and in the separated graph inverse semigroup $\IS(E_X,C)$. In both semigroups, the free group component $g$ is the geodesic path joining $\alpha$ and $\beta$, in our example it corresponds to $g=e^2f^{-1}$. The leaves, or maximal elements, of the tree represent the idempotent elements labeled $ef$, $e^2f^2$, $e^2fe$ and $e^2f^{-1}$. These give the idempotent part of the Scheiblich normal form in the free inverse monoid $\FIM(X)$:
   $$u = \Big[ (ef)(ef)^{-1}(e^2f^2)(e^2f^2)^{-1}(e^2fe)(e^2fe)^{-1}(e^2f^{-1})(e^2f^{-1})^{-1}\Big] \cdot  \Big[ e^2f^{-1}\Big].$$

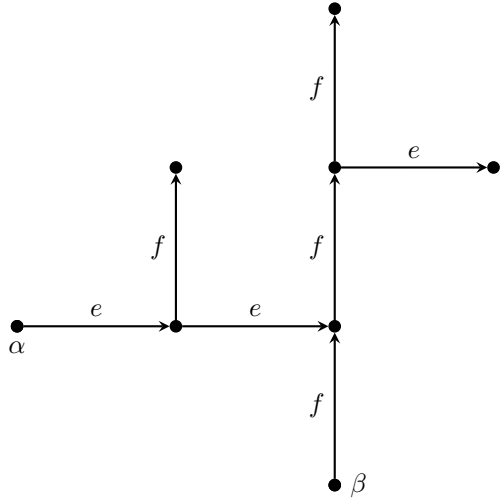
\begin{figure}[htb]
	\centering
	\begin{tikzpicture}[scale=0.7]
		\SetGraphUnit{4}
		\GraphInit[vstyle=Classic]
		\tikzset{VertexStyle/.append style={minimum size=1.5pt, inner sep=1.5pt}}
		\Vertex[Math,Lpos=-90,L={\alpha}, x=0,y=0]{a}
		\Vertex[Math,Lpos=0,L={\beta}, x=6,y=-3]{b}
		\SetVertexNoLabel
		\Vertex[Math,Lpos=-90]{v}
		\EA[Math, Lpos=0, unit=3](v){v1}
		\EA[Math, Lpos=0, unit=3](v1){v2}
		\NO[Math, Lpos=0, unit=3](v2){v3}
		\EA[Math, Lpos=0, unit=3](v3){v4}
		\NO[Math, Lpos=0, unit=3](v3){v5}
		\NO[Math, Lpos=0, unit=3](v1){v6}
        \SO[Math, Lpos=0, unit=3](v2){v7}
		\tikzset{EdgeStyle/.style = {->,>=stealth}}
		\tikzset{LabelStyle/.style = {above,fill=none}}	
		\Edge[style={->}, label=$e$](v)(v1)
		\Edge[style={->}, label=$e$](v1)(v2)
		\Edge[style={->}, label=$e$](v3)(v4)
		\tikzset{LabelStyle/.style = {left,fill=none}}
		\Edge[style={->}, label=$f$](v2)(v3)
		\Edge[style={->}, label=$f$](v3)(v5)
		\Edge[style={->}, label=$f$](v1)(v6)
        \Edge[style={->}, label=$f$](v7)(v2)
		
	\end{tikzpicture}
	\caption{Representation of Munn tree of $u=e^2e^{-1}ff^{-1}ef^2f^{-1}ee^{-1}f^{-2}$}
	\label{fig2}
	\end{figure}
In the inverse semigroup $\IS (E_X,C)$, the Scheiblich normal form is slightly different. Here, the idempotent element corresponding to the leave $g=e^2f^{-1}$ of the tree is not included
in the idempotent part of the Scheiblich normal form.
Therefore, we can express $u \in \IS (E_X,C)$ as 
$$u = \Big[ (ef)(ef)^{-1}(e^2f^2)(e^2f^2)^{-1}(e^2fe)(e^2fe)^{-1} \Big]\cdot \Big[ e^2f^{-1}\Big].$$ 
The distinction arises because $g=e^2f^{-1}$ ends in $E^{-1}$, leading the corresponding idempotent to contract to $e^2e^{-2}$ in $\IS(E_X,C)$, and $e^2$ is no longer a leave of the tree.
\end{example}

\begin{example}
    \label{exam:free-inverse-two}
Here we use the description in terms of Munn trees of the inverse semigroup $\IS(F_X,D)$ introduced in Example \ref{exam:free-inverse-monoid} to conclude that the map $\varphi\colon \FIM (X) \to v\IS (F_X,D)v$ determined by $\varphi (x) = e_xf_x^{-1}$, for $x\in X$, is an isomorphism onto $(v\IS(F_X,D)v)^\times$. 

Let $\FIS(F_X)\cong \mathcal M$ be the free inverse semigroup on $F_X$. Identify $\FIM(X)$ with the collection of Munn trees $(T,g)$ over the free group $\F (X)$. We have a semigroup homomorphism
$\phi\colon \FIM (X) \to \FIS(F_X)$ given by $\phi (x)= e_xf_x^{-1}$, and $\varphi = \rho \circ \phi$, where $\rho \colon \FIS(F_X) \to \FIS(F_X)/{\sim}= \IS (F_X,D)$ is the canonical projection, where $\sim$ is the congruence generated by the relations $e^{-1}e \sim r(e)$ for all $e\in E^1$.
(Note that every Munn $F_X$-tree is $D$-compatible, because we are dealing with the free separation on $F_X$.)

The map $\phi$ sends a Munn tree $(T,g)$ to the Munn $F_X$-tree $(\phi (T), \phi(g))$. The tree $\phi (T)$ is obtained by replacing each edge labeled by $x\in X$ by two edges labeled by $e_x$, $f_x^{-1}$, and each edge labeled by $x^{-1}$ by two edges labeled by $f_x$, $e_x^{-1}$. This is a Munn $F_X$-tree which belongs to the connected component $C_v$, and all the leaves of the tree end in either $e_x^{-1}$ or $f_x^{-1}$ for some $x\in X$. Let $(T,h)$ be a Munn $F_X$-tree such that $(T,h)\in v\mathcal M v$. Then $s(h)=r(h)=v$, hence there is a unique $g\in \F (X)$ such that $\phi (g)= h$.
In addition we have that $T$ is a finite subtree of $C_v$, the connected component of $\Gamma_{F_X}$ corresponding to $v$. Although some of the leaves of $T$ can end in $F_X$, there is a unique representative $T'$ of $T$ which has all the leaves ending in $F_X^{-1}$, and there is a unique Munn tree $(S,g)\in \FIM(X)$ such that $\phi (S,g) = (T',h)$. This shows that the composite map $\varphi = \rho\circ \phi$ is a bijection, and thus an isomorphism of semigroups, as desired.    
 \end{example}

As a first indication of the power of the Scheiblich normal form, we show next a rigidity result that generalizes \cite{mesyan-mitchell-2016}*{Corollary~26} to separated graphs.

We denote by $\Aut (E,C)$ the automorphism group of $(E,C)$. This is the subgroup of the automorphism group $\Aut(E)$ of $E$ consisting of those automorphisms $\phi= (\phi^0,\phi^1)$ of $E$ such that $\phi^1 (X) \in C$ for all $X\in C$, see \cite{ABC}*{Definition~3.2}. 
The group of all the semigroup automorphisms of $\IS(E,C)$ will be denoted by $\Aut (\IS (E,C))$. Observe that each surjective semigroup homomorphism between semigroups with zero is necessarily zero-preserving, hence all automorphisms $\varphi \in \Aut (\IS(E,C))$ satisfy $\varphi (0) = 0$.  

\begin{theorem}
    \label{thm:automorphisms}
    Let $(E,C)$ be a separated graph. Then there is a natural group isomorphism $\Aut (\IS(E,C))\cong \Aut (E,C)$.
\end{theorem}

\begin{proof}
We define a map $\delta \colon \Aut (E,C) \to \Aut (\IS(E,C))$ by
$$\delta (\phi) (v) = \phi^0(v),\quad \delta (\phi)(e) = \phi^1(e),\quad \delta(\phi) (e^{-1}) = (\phi(e))^{-1}, \quad \delta (\phi) (0)=0 , $$ 
where $\phi = (\phi^0,\phi^1) \in \Aut (E,C)$, $v\in E^0$, $e\in E^1$. We need to show that the above assignments define a unique semigroup automorphism $\delta (\phi)$ of $\IS(E,C)$. For this we need to check that $\varphi: =  \delta (\phi)$ preserves the defining relations (1)--(4) of $\IS(E,C)$ (see Definition \ref{def:graphsemigroup}). 
Indeed preservation of (2) and (3) follows immediately from the fact that $\phi $ is a graph homomorphism. Preservation of (1) follows from the fact that $\phi^0$ is injective. 
To show that relation (4) is preserved by $\varphi$, take $e,f\in X$, where $X\in C_v$
for $v\in E^0$. Then $Y:=\phi^1(X)\in C_{\phi^0(v)}$ because $\phi\in \Aut (E,C)$. 
If $e= f$ then we have
$$\varphi (e^{-1}) \varphi (f) = \phi^1 (e)^{-1} \phi^1 (e) = r(\phi^1(e))= \phi^0 (r(e)) =\varphi (r(e)) ,$$
showing the preservation of (4) in this case. If $e\ne f$ then $\phi^1(e) \ne \phi^1(f)$ because $\phi^1$ is injective, and since $\phi^1(e),\phi^1(f)\in Y\in C_{\phi^0(v)}$, we have
$$\varphi (e^{-1}) \varphi (f) = \phi^1 (e)^{-1} \phi^1 (f) = 0.$$
Hence relation (4) is preserved by $\varphi$. 

We therefore have a well-defined semigroup homomorphism $\delta (\phi)$ of $\IS(E,C)$.
It is clear that $\delta$ preserves the identity and sends the composition of two automorphisms of $(E,C)$ to the corresponding composition of homomorphisms of $\IS(E,C)$. 
It follows that each $\delta (\phi)$ belongs to $\Aut (\IS(E,C))$, and that $\delta $ is a group homomorphism from $\Aut (E,C)$ to $\Aut (\IS(E,C))$. Note that Theorem \ref{thm:uniqueness-SNF} gives in particular that the natural map  $E^0\cup E^1\to \IS (E,C)$ is injective, and it follows readily from this that $\delta$ is injective.

It remains to show that $\delta$ is surjective. This will be done by an adaptation of the argument in \cite{mesyan-mitchell-2016}*{Proposition~22} to our setting, using the uniqueness of Scheiblich normal form (Theorem \ref{thm:uniqueness-SNF}). Let $\varphi \in \Aut (\IS(E,C))$. We want to show that $\varphi (E^0) = E^0$ and $\varphi (E^1) = E^1$.
Observe that $\varphi$ restricts to an order-automorphism of the semilatice 
of idempotents $\mathcal E$ of $\IS(E,C)$. Now it is clear from Theorem \ref{thm:uniqueness-SNF} that the maximal elements of $\mathcal E$ are precisely the vertices of $E$, hence we obtain that $\varphi (E^0) = E^0$. Now it follows again from Theorem \ref{thm:uniqueness-SNF} that the maximal elements of $\mathcal E \setminus E^0$ are the idempotents of the form $xx^{-1}$, where $x= e_1^{-1} \cdots e_n^{-1} f$ is a $C$-separated path, with $e_1,\dots ,e_n,f \in E^1$, and $n\ge 0$.

Let $e\in E^1$. We want to show that $\varphi (e) \in E^1$. Since $ee^{-1}$ is a maximal element of $\mathcal E \setminus E^0$, and $\varphi$ restricts to a permutation of $\max ( \mathcal E \setminus E^0)$, it follows from the above that 
$$\varphi (ee^{-1}) = xx^{-1}, \quad \text{where} \quad  x= e_1^{-1}\cdots e_n^{-1} f \in \FC,$$
with $e_1,\dots ,e_n,f\in E^1$ and $n\ge 0$.  Now write
$$\varphi (e) = (\gamma_1\gamma_1^{-1}) \cdots (\gamma_r\gamma_r^{-1}) \lambda $$
in its unique Scheiblich normal form. Then we have
$$xx^{-1} = \varphi (ee^{-1}) = \varphi (e) \varphi (e)^{-1} = (\gamma_1\gamma_1^{-1})\cdots (\gamma_r \gamma_r^{-1}) .$$
Hence by uniqueness of the Scheiblich normal form, we have $r=1$ and $\gamma_1 = x=e_1^{-1}\cdots e_n^{-1} f$. Now we can write $\lambda = yg^{-1}$, where $g= g_1\cdots g_t$, with $g_i\in E^1$ for all $i$, $t\ge 0$, and $x= yz$, where $y$ is largest initial segment of $x$ which is also an initial segment of $\lambda$. 

We compute
\begin{align*}
    \varphi (e^{-1}e) & = \varphi(e)^{-1} \varphi (e) \\
    & = \lambda^{-1} xx^{-1} \lambda  \\
    & = gy^{-1} (yzz^{-1}y^{-1}) yg^{-1} \\
    & = [gy^{-1}yg^{-1}][gzz^{-1}g^{-1}]\\
    & = [g_1\cdots g_t y^{-1}y g_t^{-1} \cdots g_1^{-1}][g_1\cdots g_tzz^{-1} g_t^{-1}\cdots g_1^{-1} ]
\end{align*}
On the other hand $\varphi (e^{-1}e) = \varphi (r(e))$, which is a vertex of $E$. 
It is now a simple matter to show, using Theorem \ref{thm:uniqueness-SNF} that for any
possible choice of $y$, the Scheiblich normal form of $[g_1\cdots g_t y^{-1}y g_t^{-1} \cdots g_1^{-1}][g_1\cdots g_tzz^{-1} g_t^{-1}\cdots g_1^{-1} ]$ is never trivial, except in the case where $n=0=t$, in which case we get $\varphi (e) = f$. This shows that $\varphi (E^1) \subseteq E^1$, and since $\varphi$ is an automorphism we get that $\varphi (E^1) = E^1$.  

Finally we show that $\phi := (\varphi|_{E^0},\varphi|_{E^1})$ is an automorphism of $(E,C)$. As in the proof of \cite{mesyan-mitchell-2016}*{Proposition~22}, we get for 
$e\in E^1$ that 
$$ \phi^1(e) =  \varphi (e) = \varphi (s(e) e r(e)) = \phi^0(s(e)) \phi^1(e) \phi^0(r(e)).$$
Hence $s(\phi^1(e)) = \phi^0(s(e))$ and $r(\phi^1(e)) = \phi^0(r(e))$, showing that $\phi$ is a graph morphism, hence a graph automorphism. 
Suppose now that $X\in C_v$ for some $v\in E^0$. Then $\phi^1(X) =\varphi (X)$ is a subset of $s^{-1} (\phi^0(v))$. If there exist $e,f\in X$ such that $\varphi (e) \in Y_1 $ and $\varphi (f) \in Y_2$ for distinct $Y_1,Y_2\in C_{\phi^(v)}$, then 
$[\varphi (e)\varphi(e)^{-1}][\varphi (f)\varphi (f)^{-1}]\ne 0$ in $\IS(E,C)$, and so
$$0= \varphi (ee^{-1}ff^{-1}) = [\varphi (e)\varphi(e)^{-1}][\varphi (f)\varphi (f)^{-1}]\ne 0, $$
which is a contradiction. Hence there is a unique $Y\in C_{\phi^0(v)}$ such that $\varphi (X)\subseteq Y$. Using again that $\varphi$ is an automorphism we get $\varphi (X) = Y$. We conclude that $\phi\in \Aut (E,C)$, and clearly  $\delta (\phi) = \varphi$. Hence $\delta $ is surjective. This concludes the proof. 
\end{proof}

\begin{remark}
\label{remark:extension-to-isos}
Theorem \ref{thm:automorphisms} can be easily generalized to isomorphisms. Concretely, the assignment $(E,C) \mapsto \IS(E,C)$ can be extended to a full and faithful functor from the category of separated graphs with isomorphisms of separated graphs as morphisms to the category of inverse semigroups with semigroup isomorphisms as morphisms. In particular, for separated graphs $(E,C)$ and $(F,D)$, we have $(E,C)\cong (F,D)$ if and only if $\IS(E,C)\cong \IS (F,D)$. 
\end{remark}

\section{The tight spectrum}
\label{sect:tightsection}

Before we continue with our analysis of (inverse) semigroups of separated graphs, let us discard a trivial part of them, namely the \emph{isolated vertices}, that is, the vertices $v\in E^0$ with $s^{-1}(v)=r^{-1}(v)=\emptyset$. If $E^0_{\mathrm{iso}}$ denotes the set of isolated vertices, notice that these are idempotent elements of both semigroups $S=\S(E,C)$ or $S=\IS(E,C)$ that are orthogonal to every other element. This means that $S$ contains a (commutative) inverse subsemigroup $E^0_{\mathrm{iso}}\cup \{0\}$ consisting of pairwise orthogonal idempotents that do not interact with the other non-zero elements of the semigroup. This ``isolated part'' of $S$ will not be interesting in any respect, as for instance in the algebras we are going to consider later, this only contributes to a trivial commutative direct sum of some copies of the underlying ground field (or ring). Hence there is no loss of generality in assuming $E^0_{\mathrm{iso}}=\emptyset$, so we postulate, from now on, the following:

\vskip 0,5pc

{\bf Standing assumption:} \emph{we shall assume, from now on, that all our graphs $E$ have no isolated vertices, that is, $E^0_{\mathrm{iso}}=\emptyset$.}

\vskip 0,5pc

We fix for the entire section, a separated graph $(E,C)$ (without isolated vertices) and look at its associated inverse semigroup $\IS(E,C)$.
Now that we have a normal form for its elements, we develop Exel's program \cite{Exel:Inverse_combinatorial} for the inverse semigroup $\IS(E,C)$, describing its \emph{tight spectrum}.
We will use the picture of $\IS(E,C)$ as the restricted semidirect product $\YY\rtimes_\theta^r\F$ obtained in Theorem \ref{thm:ECMunntrees}. This gives, in particular, a nice parametrization of the semilattice $\E(\IS(E,C))\cong \YY$, which makes computations much easier.

We will identify the space of ultrafilters and the space of tight filters on the semilattice $\YY$. We refer to \cite{Exel:Inverse_combinatorial} for background on these concepts, but recall here quickly the main general ingredients.

Given a semilattice $\E$ with zero, we write $\dual\E$ for its spectrum. This can be viewed either as the space of characters $\dual\E$, that is, non-zero homomorphisms $\chi\colon \E\to \{0,1\}$ with $\chi(0)=0$, or equivalently, as the space of filters $\mathcal F(\E):=\{\xi\sbe \E^\times: \xi=\chi^{-1}(1), \chi\in \dual\E\}$. The map $\chi\mapsto \chi^{-1}(1)$ yields a bijection $\dual\E\congto \mathcal F(\E)$. This is a locally compact Hausdorff and totally disconnected space with the topology inherited from the compact product space $\{0,1\}^\E$; equivalently, this is the topology of pointwise convergence on characters, and it is the smallest topology that makes the sets of filters $\{\xi\in \mathcal F: e\in\xi\}$ into clopen subsets for all $e\in \E^\times$. In other words, the topology has as basis of compact open subsets the following subsets of filters:
$$\U(e;f_1,\ldots,f_n):=\{\xi\in\mathcal F: e\in \xi, f_i\notin \xi\,\forall i=1,\ldots,n, n\ge 0\}.$$

An {\it ultrafilter} is a maximal filter, and an {\it ultracharacter} is a character $\chi \in \dual\E$ such that the corresponding filter $\chi^{-1}(1)$ is an ultrafilter.
The {\it tight spectrum} $\Etight$ of $\E$ is the closure in $\dual\E$ of the space of ultracharacters. The elements of $\Etight$ are called {\it tight characters} and the corresponding filters are called {\it tight filters}. 

Recall that $\YY$ consists of certain finite lower subsets of $\FC$.
We will now also consider infinite lower subsets $Z$ in $\FC$, with respect to the prefix order $\le_p$. Geometrically, these infinite lower subsets correspond to {\it infinite Munn $(E,C)$-trees}. In the following, we say that an element $z$ of $\FG(E)$ {\it can be extended}
to another element $z'$ of $\FG(E)$ if $z$ is an initial segment of $z'$, that is, if $z\le_p z'$. 
 
\begin{definition}
	\label{def:filters-ultra-tight}
	\begin{enumerate}
		\item For $v\in E^0$, let $\mathfrak F(v)$ be the set of all (possibly infinite) non-empty $C$-compatible lower subsets $Z$ of $\FC(v)$ such that each $z\in Z$ can be extended to an element of $Z$ not ending in $E^{-1}$. Set $\mathfrak F = \bigsqcup_{v\in E^0} \mathfrak F (v)$. 
		\item We order $\mathfrak F$ by inclusion, and let $\mathfrak U$ be the set of all maximal elements $Z\in \mathfrak F$.
		 \end{enumerate}
\end{definition}

Notice that two elements of $\mathfrak F$ can only be comparable (with respect to inclusion) if both are inside $\mathfrak F(v)$ for the same vertex $v\in E^0$.

\begin{proposition}
\label{prop:characterizing-filt-ultra-tight}
With the above notation:
\begin{enumerate}
    \item The poset $\mathcal{F}$ of filters of $\YY$ is order-isomorphic to the poset $\mathfrak{F}$.
    \item In particular, this order-isomorphism restricts to a bijection between the set of ultrafilters of $\YY$ and the subset $\mathfrak{U} \subseteq \mathfrak{F}$.
\end{enumerate}
\end{proposition}

\begin{proof}
	(1) We show that the map
	$$\varphi \colon \mathfrak{F} \mapsto \mathcal F$$
given by $\varphi (Z)= \{[I]\in \YY^\times: I_0 \subseteq Z \}$
is an order-isomorphism. First we observe that $\varphi$ is well defined, that is, $\varphi (Z)$ is a filter in $\YY$. Indeed, suppose that $[I]\le [J]$ in $\YY$ and $[I]\in \varphi (Z) $. Then we have $J_0\subseteq I_0$ by 
Lemma~\ref{lem:maxelements}, and so $J_0\subseteq I_0 \subseteq Z$, showing that $[J]\in \varphi (Z)$. 
If $[I],[J]\in \varphi (Z)$, then $I_0\subseteq Z$ and $J_0\subseteq Z$, so that $I_0 \cup J_0\subseteq Z$ is $C$-compatible, because $Z$ is $C$-compatible, and we have
$$[I] \cdot [J] = [I_0 \cup J_0]\in \varphi (Z).$$ 	
	By definition $0\notin \varphi (Z)$. 

We now show that $Z\subseteq Z'$ if and only if $\varphi (Z) \subseteq \varphi (Z')$. If $Z\subseteq Z'$, and $[I]\in \varphi (Z)$, then $I_0 \subseteq Z \subseteq Z'$ so that $[I]\in \varphi (Z')$.

Suppose now that $\varphi (Z)\subseteq \varphi (Z')$, and take $h\in Z$.
By our hypotheses on the elements of $\mathfrak F$, we may find an element $g\in Z$ such that $h\le_p g$ and $g$ does not end in $E^{-1}$. Then $(g^\downarrow)_0 = g^\downarrow$, and since $(g^\downarrow)_0 \subseteq Z$, we have $[g^\downarrow]\in  \varphi (Z)\subseteq \varphi (Z')$, showing that $g^\downarrow \subseteq Z'$, and thus $g\in Z'$. Now $Z'$ is a lower set and $h\le_p g$, thus $h\in Z'$.    

This shows in particular that $\varphi$ is injective. To show that it is surjective, take any filter $\eta $ on $\YY$. Observe that since $0\notin \eta$ and $\eta$ is closed under products, there must exist $v\in E^0$ such that all elements of $\eta$ belong to $\Y (v)$.

Let $C$ be the collection of all elements of the form $g_0$, where $[g^\downarrow]\in \eta$, and $g= g_0w$ is the canonical decomposition of $g$ (see Lemma \ref{lem:maxelements}). Let $Z$ be the lower subset of $\FC (v)$ generated by $C$. We claim that $Z\in \mathfrak F$ and that $\varphi (Z)= \eta$.  Let $g,h\in \FC (v)$ such that $[g^\downarrow] , [h^\downarrow] \in \eta$. Then $[g^\downarrow] \cdot [h^\downarrow] \in \eta$, and in particular $[g^\downarrow] \cdot [h^\downarrow]  \ne 0$. It follows that $g$ and $h$ are $C$-compatible, and hence so are $g_0$ and $h_0$. It follows that any two elements of $Z$ are $C$-compatible. Take any $z\in Z$. Then there exists $g\in \FC (v)$ such that $[g^\downarrow]\in \eta$ and $z\le_p g_0$, where $g=g_0w$ is the canonical decomposition of $g$. It follows that $z$ can be extended to an element $g_0\in Z$ which does not end in $E^{-1}$. Hence we conclude that $Z\in \mathfrak F$. It remains to show that $\varphi (Z) = \eta$. We first show that $\varphi (Z)\subseteq \eta$. Let $[I]\in \varphi (Z)$. Then $I_0 \subseteq Z$ by the definition of $\varphi$. Hence for all $g\in I$ we have $g_0\in I_0 \subseteq Z$, so that $g_0\in Z$ and thus $[g_0^\downarrow]\in \eta$ by the definition of $Z$. Hence 
$$[I] = \prod _{g\in I} [g_0^\downarrow] \in \eta.$$
Suppose now that $[I]\in \eta$. We have to show that $I_0 \subseteq Z$, and for this it is enough to show that $g_0\in Z$ for all $g\in I$. But if $g\in I$, then $[I] \le [g^{\downarrow}]$ and thus $[g^{\downarrow}] \in \eta$ because $\eta$ is a filter. It follows that $g_0\in Z$, as desired.

(2) follows from (1) and the definition of ultrafilter.     
 	\end{proof}

Using the above proposition, we will identify in the following the set of filters on $\YY$ with $\mathfrak F$.

\begin{remark}
    A similar argument as above shows that the set of filters $\dual\Y$ 
    of the semilattice $\Y$ can be canonically identified with the set $\tilde{\mathcal F}$ of all non-empty lower $C$-compatible 
    subsets of $\FC$ via the map $\psi\colon \widetilde{\mathcal F}\to \dual\Y$ that sends 
    $Z\in \tilde{\mathcal F}$ to $\psi(Z):=\{I\in \Y^\times: I\sbe Z\}$. The inverse map sends a filter 
    $\xi\in \dual\Y$ to $\psi^{-1}(\xi)=\{g\in \FC: g^\downarrow \in \xi\}$.
\end{remark}

\begin{example}
    \label{exam:non-separated-filters-and-ultrafilters}
    Let $E$ be a non-separated graph, that is, a directed graph endowed with the trivial separation $C$, where $C_v= \{s^{-1}(v)\}$ for all $v\in E^0$. 
    Then the elements of $\FC$ are the vertices of $E$ and the paths of the form $$e_1\cdots e_n  f_1^{-1} \cdots f_m^{-1}, \quad \text{ where } n,m\ge 0, \, \, n+m>0 , \,\, e_i,f_j\in E^1, e_n\ne f_1.$$
    Moreover, an element of $\mathfrak F$ must be the set of all initial segments of a finite or infinite path on $E$ (not involving ghost edges $e^{-1}$).
    Hence the elements of $\mathfrak F$ can be identified with the finite or infinite paths on $E$, and the set $\mathfrak U$ can be identified with the union of the set of infinite paths and the set of finite paths ending in a sink. In particular all sinks and all infinite paths on $E$ belong to $\mathfrak U$.  
\end{example}

We now proceed to describe the space of ultrafilters $\mathfrak U$. This space has been described in \cites{Ara-Exel:Dynamical_systems, Lolk:tame} for certain special cases -- either for finite bipartite or finitely separated graphs.

It is convenient to define first \emph{local configurations} of arbitrary non-trivial lower subsets of $\FC$, as follows. In what follows, a lower subset $Z$ of $\FC(v)$, for $v\in E^0$, will be called {\it non-trivial} if it contains at least one element $z\ne v$, or equivalently $|Z| \ge 2$.   

\begin{definition}
	\label{def:local-configuration}
	A {\it local configuration $\mathfrak c$ at $v\in E^0$} is any non-empty subset
	$$ \mathfrak c \subseteq \{ x\in E^1\cup E^{-1}: s(x) =v\}=(E^1\cup E^{-1})\cap s^{-1}(v).$$
	A local configuration $\mathfrak c$ is {\it admissible} if $\mathfrak c$ is $C$-compatible, and $\mathfrak c$ is {\it maximal} if it is admissible and there does not exist any admissible local configuration $\mathfrak c'$ such that $\mathfrak c\subsetneq \mathfrak c'$. The notion of $C$-compatibility we use here is from Definition~\ref{def:Ccompatible}. Since here we are talking only about subsets $\mathfrak c\sbe (E^1\cup E^{-1})\cap s^{-1}(v)$, notice that such a subset can only be $C$-incompatible if it contains two different edges $e,f\in X$ for some $X\in C_v$.
\end{definition}
	
	Let $Z$ be a non-trivial lower subset of $\FC (v)$ with respect to the prefix order, for some $v\in E^0$. For $g\in Z$, we define the {\it local configuration} of $Z$ at $g$ as:
	$$Z_g = (E^1 \cup E^{-1})\cap g^{-1} \cdot Z:=\{ x\in E^1 \cup E^{-1}: x\in g^{-1}\cdot  Z\}=\{x\in E^1\cup E^{-1}: g\cdot x\in Z\}.$$
	If $g=v$, then $Z_v$ is a local configuration at $v$ as defined above.  If $g\ne v$, notice that $Z_g$ is a local configuration at $r(g)\in E^0$ as defined above.
 In the latter case, we define the {\it tail} of the local configuration $Z_g$ as $y^{-1}$, where $y\in E^1\cup E^{-1}$ is the last letter of $g$. Note that $y^{-1}\in Z_g$ and $r(g) = r(y)=s(y^{-1})$. All other elements of $Z_g$ are called the {\it leaves} of the local configuration $Z_g$.

\begin{figure}[htb]
    \centering
		\begin{tikzpicture}[scale=0.7, rotate=45]
		\SetGraphUnit{6}
		\GraphInit[vstyle=Classic]
		\tikzset{VertexStyle/.append style={minimum size=1.5pt, inner sep=1.5pt}}
		\Vertex[Math, Lpos=-90]{v}
        \SetVertexNoLabel
        \Vertex[x=6.7, y=-2.8]{T}.
        \Vertex[x=8, y=-2]{P}.
        \EA[Math, Lpos=0, unit=4](v){g}
        \EA[Math, Lpos=-90, unit=2](g){w}
        \tikzset{VertexStyle/.append style={minimum size=0pt, inner sep=0pt}}
        \EA[Math, Lpos = -90, unit=2.5](w){w_1}
        \SO[Math, Lpos=-90, unit=2.5](w){w_2}
        \SOEA[Math, Lpos=-90, unit=1.8](w){w_3}
        \NO[Math, Lpos=-90, unit=2.5](w){w_4}
        \NOEA[Math, Lpos=135, unit=1.8](w){w_5}
        \tikzset{EdgeStyle/.style = {->, >=stealth, scale=5, thick}}
        \Edge[style={->,out=45, in=225, looseness=1.5},labelstyle={above,fill=none}, label=\textbf{$g_1$}](v)(g)
        \Edge[style={->,out=45, in=200, looseness=1.5}, labelstyle={above,fill=none}, label=\textbf{$e$}](g)(w)
        \Edge[style={->}](w)(w_1)
        \Edge[style={->}](w)(w_4)
        \Edge[style={->}, label=\textbf{leaves}](w)(w_5)
        \Edge[style={->}](w_3)(w)
        \Edge[style={->}, label=$\textbf{leaves}$](T)(w)
        \Edge[style={->}](P)(w)
        \tikzset{VertexStyle/.append style={minimum size=0pt, inner sep=0pt}}
        \SetVertexLabel
        \SetVertexMath
        \Vertex[Math,L=\bf{tail},x=4.5,y=0]{a}
		\end{tikzpicture}
        \caption{Representation of leaves and tail of the local configuration $Z_g$ for $g$ of the form $g = g_1e$. Note that the only tail is $e^{-1}$ and the leaves are all the other edges coming to and leaving from $r(g) = r(e)$. Geometrically, $Z_g$ represents the set of vertices at distance one of the vertex $g$ in the $E$-tree corresponding to $Z$.}
        \label{fig1}
\end{figure}
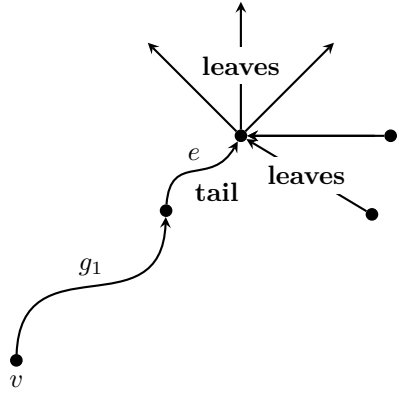

    Notice that $g\cdot Z_g\sbe Z$ are exactly the elements of $Z$ that have distance $1$ to $g\in Z$ with respect to the length metric in the Cayley graph $\Gamma_E$ of the fundamental group $\FG(E)$. Notice also that $g\in \max (Z)$ if and only if the tail of $Z_g$ is the unique element of $Z_g$. 

\begin{lemma}
	\label{lem:properties-of-configs}
	\begin{enumerate}
		\item A local configuration $\mathfrak c$ at $v\in E^0$ is maximal if and only if for each $X \in C_v$ there exists $e_X \in X$ such that 
		$$\mathfrak c = \{ e_X : X\in C_v\}\cup \{e^{-1}: e\in r^{-1}(v)\}.$$
		\item  A local configuration $\mathfrak c$ is admissible if and only if it is contained in a maximal local configuration.
		\item If $Z$ is any non-trivial lower subset of $\FC(v)$ for some $v\in E^0$, then $Z$ is $C$-compatible if and only if $Z_g$ is an admissible local configuration for each $g\in Z$.  
	\end{enumerate}
\end{lemma}

\begin{proof}
(1) and (2) are clear. 	
	
	(3) If $Z$ is $C$-compatible then, by definition, $\red (g^{-1} h)\in \FC$ for $g,h\in Z$. Suppose that we can find $u\in Z$ and $x,y\in X$ with $x,y\in Z_u$, $x\ne y$, for $X\in C_{r(u)}$. If $ux$ and $uy$ are reduced words as they stand, then $x\in u^{-1}\cdot Z$ and $y\in u^{-1}\cdot Z$ imply that $ux,uy\in Z$, and so 
	$$x^{-1}y = \red ((ux)^{-1} (uy))\in \FC ,$$
	which is a contradiction. If for instance $u=u'x^{-1}$, then $u\cdot y = uy = u'x^{-1}y \notin \FC$, and we get also a contradiction, because all elements of $Z$ belong to $\FC$. This contradction shows that the local configuration $Z_u$ at each $u\in Z$ must be admissible.
	
	Conversely suppose that the local configuration $Z_u$ is admissible for each $u\in Z$.       
		Let $g,h$ be distinct elements of $Z$. Assume first that $g,h$ are comparable with respect to $\le_p$, for instance $g\le_p h$. 
	Write $h= gh'$. Then $ g^{-1}\cdot h = \red (g^{-1} h) = h'$, and $h'\in \FC$ because $h\in Z\subseteq \FC$. 
	
	If $g,h$ are incomparable, write $g= uxg'$, $h=uyh'$, with $x\ne y$, $x,y\in E^1\cup E^{-1}$. 
	Then
	$$g^{-1} \cdot h = \red (g^{-1} h) = {g'}^{-1}x^{-1}yh',$$
	and ${g'}^{-1}x^{-1}yh'\in \FC$ because the local configuration of $Z$ at $u$ is admissible, and in addition $g',h'\in\FC$ because $g,h\in \FC$.  This implies that $Z$ is $C$-compatible. 
 \end{proof}

We can now obtain the characterization of the space of ultrafilters, as follows. Recall that the set $\mathfrak U$ has been introduced in Definition \ref{def:filters-ultra-tight}.

\begin{theorem}
	\label{cor:ultrafilters} Let $\mathfrak U'$ be the set of all non-trivial lower subsets $Z'$ of $\FC (v)$ such that $Z'_g$ is a maximal local configuration for all $g\in Z'$, where $v$ ranges over all vertices of $E$. Then there is a bijection $\mathfrak U'\cong \mathfrak U$ between $\mathfrak U '$ and the set of ultrafilters $\mathfrak U$.      
 \end{theorem}

\begin{proof}
	Take $Z'\in \mathfrak U'$, and let $v\in E^0$ be the unique vertex such that $Z'$ is a non-trivial lower subset of $\FC (v)$. By definition, for each $g\in Z'$, the set $Z'_g$ is a maximal local configuration. Since $v\in Z'$ and $E$ does not have isolated vertices, there exists $y\in E^1\cup E^{-1}$ such that $y\in Z'_v$, because $Z'_v$ is a maximal local configuration.  Let $S$ be the set of all elements $g\in Z'$ such that $g$ ends in $E^{-1}$ and $g$ cannot be extended to an element $z\in Z'$ such that $z$ ends in $E^1$. 
	
	Then the bijection $\varphi \colon \mathfrak U '\to \mathfrak U$ sends $Z'$ to $Z:= Z'\setminus S$, i.e. $\varphi (Z')= Z'\setminus S$. Let us check that $Z\in \mathfrak F (v)$. It is clear that $Z$ is a lower subset of $\mathfrak F (v)$, and $Z$ is $C$-compatible by Lemma \ref{lem:properties-of-configs}(3). We need to show that every element $g\in Z$ can be extended to an element $z\in Z$ such that $z$ does not end in $E^{-1}$. Take $g\in Z$. This is obvious if $g$ does not end in $E^{-1}$, so assume that $g$ ends in $E^{-1}$. Since $g\notin S$, we see from the definition of $S$ that $g$ can be extended to $z\in Z'$ such that $z$ does not end in $E^{-1}$. In particular $z\in Z$, so that we have shown the desired property of $Z$. 
	
	We finally show that $Z$ is maximal in $\mathfrak F$, equivalently in $\mathfrak F (v)$. By way of contradiction suppose that $Z\subsetneqq \tilde Z$ for some $\tilde Z\in \mathfrak F (v)$. Take $g\in \tilde Z\setminus Z$, and let $h$ be the largest prefix of $g$ such that $h\in Z$, so that $hy \in \tilde Z\setminus Z$ for some $y\in E^1\cup E^{-1}$.  Suppose first that $y\in E^1$. Then $y\in X$ for a unique $X\in C_{r(h)}=C_{s(y)}$. On the other hand, $h\in Z'$ and thus $Z'_h$ is a maximal local configuration. It follows that there is a unique $x\in X\cap Z'_h$, so that $hx\in Z'$. Hence $hx\in Z$, because it belongs to $Z'$ and ends in $E^1$. (Observe that $h$ cannot end in $x^{-1}$, because $hy$ is a reduced word that belongs to $\tilde{Z}$, and so it belongs to $\FC (v)$; hence $hx$ is reduced.) If $x= y$ then
	$hy \in Z$ which is not the case, hence $x\ne y$. But then, since $Z\subseteq \tilde Z$, we have $hx\in \tilde Z$.
	Since also $hy\in \tilde Z$, we obtain a contradiction, because $\tilde Z$ is $C$-compatible, and $hx$ and $hy$ are $C$-incompatible. So in any case we get a contradiction, and thus $y\in E^{-1}$. Write $y=e_1^{-1}$, where $e_1\in E^1$. We have $he_1^{-1}\in \tilde Z$. Moreover since $Z'_h$ is a maximal local configuration, $e_1^{-1} \in Z'_h$, which implies that $he_1^{-1} \in Z'$. Now since $\tilde Z\in \mathfrak F (v)$, we can extend $he_1^{-1}$ to a reduced word ending in $E^1$, so that there are $e_2,\dots ,e_n,d\in E^1$ such that 
    $$he_1^{-1}e_2^{-1} \cdots e_n^{-1} d\in \tilde Z.$$
	  Using the same argument as before, we see inductively that all the words $he_1^{-1}\cdots e_i^{-1}$, for $i=1,\dots , n$, belong to $Z'$. Let $h'= he_1^{-1} \cdots e_n^{-1}$. Observe that $d\in Y$ for some $Y \in C_{r(h')}$. Since $Z'_{h'}$ is a maximal local configuration, there is a unique $c\in Y$ such that $c\in Z'_{h'}$. But then 
	  $$he_1^{-1}\cdots e_n^{-1} c = h' c\in Z' .$$
	  This implies that $he_1^{-1} \in Z'\setminus S=Z$, which is a contradiction. 
	  
	  We conclude that $Z$ is maximal in $\mathfrak F (v)$, and thus $Z \in \mathfrak U$, as desired. 
	  
	We now define the inverse map $\psi \colon \mathfrak U \to \mathfrak U'$.
     For $Z\in \mathfrak U$, we set
$$\psi (Z) = Z\bigsqcup \{ gx_1^{-1}\cdots x_n^{-1} \in \FC : g\in Z, n\ge 1, x_i\in E^1, \text{ and } gx_1^{-1} \notin Z \}.$$
One can easily show that $\psi (Z)$ is a non-trivial lower subset of $\FC (v)$, where $v$
is the vertex such that $Z\in \mathfrak F (v)$. Let us check that $Z'_h$ is a maximal local configuration for all $h\in Z'\setminus Z$, where $Z':= \psi (Z)$. 
Write $h=gx_1^{-1} \cdots x_n^{-1}$, where $g\in Z$, $n\ge 1$, $x_i\in E^1$ and $gx_1^{-1}\notin Z$. Observe that $x_n\in Z'_h$. Let $X\in C_{r(h)}$ be such that $x_n\in X$.
If there exists $Y\in C_{r(h)}$ such that $Y\ne X$, then choosing $y\in Y$ we have that $Z\cup \{hy\}^{\downarrow}\in \mathfrak F$ and $Z\subsetneq Z\cup \{hy\}^{\downarrow}$, contradicting the maximality of $Z$ in $\mathfrak F$. Hence we obtain that $C_{r(h)} = \{X\}$, and $x_n\in X\cap Z'_h$. Moreover for each $x\in r^{-1}(r(h))$, we have
$$hx^{-1} = gx_1^{-1} \cdots x_n^{-1} x^{-1} \in Z',$$
hence $x^{-1} \in Z'_{h}$. It follows that $Z'_{h}= \{ x_n\} \cup \{ x^{-1} : x\in r^{-1}(r(h))\}$ is a maximal local configuration. Similarly $Z'_g$ is a maximal local configuration for all $g\in Z$.
This shows that $Z'=\psi (Z)\in \mathfrak U'$. It is easy to show that $\varphi$ and $\psi$ are mutually inverse maps.  
        	\end{proof}

     \begin{remark}
         \label{rem:maximal-elements-of-ultrafilters}
 Note that for any $g\in \max (Z)$, where $Z\in \mathfrak U$, the last edge of $g$ must belong to $E^1$ (if $g$ is non-trivial), so that $r(g)$ is a sink in $E$, but the condition that $g\in \max (Z)$ is much stronger than $r(g)$ being merely a sink of $E$, because it implies that $g$ cannot be {\it properly} extended to any element $h\in \FC$ such that $h$ ends in $E^1$.       
              \end{remark}

Before we continue with our discussion, we need to analyze the topology on the space of filters $\mathfrak F$. To this end we recall that there is a basis of open subsets $\{ \mathcal U (F,G)\}$ for the topology of $\mathfrak F$, where $F,G$ are finite subsets of $\YY$ and 
$$\mathcal U (F,G) = \{ Z\in \mathfrak  F \mid I_0\subseteq Z, J_0\nsubseteq Z \,\, \forall \,  [I]\in F, \,\, \forall \, [J] \in G \},$$
see e.g. \cite{Exel:Inverse_combinatorial}. 

Note that when $F\ne \emptyset $ and $\mathcal U (F,G)\ne \emptyset$, we can replace $\mathcal U (F,G)$ by
$\mathcal U (\wedge F, G)$, so that we can assume that $F$ is a singleton. 

We are going to prove that certain elements of this basis form a basis of open compact subsets of $\mathfrak F$.   This basis is similar to the one described by Webster in \cite{webster}, although the technical details are harder in our situation. 

Recall that the sets $\Y_0(v)$, for $v\in E^0$, and $\Y_0$ have been introduced in Notation \ref{notati:Y-sub-0}.

 By the previous observation, we see that the family $\mathcal U (\{I\}, G)$, where $I\in \Y_0$ and $G$ is a finite subset of $\Y_0$, is a basis for the topology of $\mathfrak F$. 

\begin{notation}
	For an element $g= x_1\cdots x_n \in \F$ of length $n\ge 1$, with $x_j\in E^1\cup E^{-1}$, and $1\le i \le n$, we denote by $[g]_i$ the prefix of $g$ of length $i$, that is, $[g]_i := x_1\cdots x_i\in \F$.
\end{notation}

\begin{notation}
	\label{notati:opencompactbasis}
	\begin{enumerate}
		\item For $v\in E^0$, write
	$$\mathcal N (v) := \{ x_1^{-1}x_2^{-1}\cdots x_n^{-1}y \in \FC (v) : n\ge 0, x_1,\dots , x_n,y\in E^1\}.$$
	\item 	For each $I\in \Y_0$, write  
	$$\mathcal N (I) : = \{gg' \in \FC : g\in I,\,  g' \in \mathcal N (r(g)),\,  g[g']_1\notin I, \, \text{ and }  I\cup \{gg'\}^{\downarrow}\in \Y_0 \}.$$
	\item For each $I\in \Y_0$ and each finite subset $F$ of 
	$\mathcal N (I)$ we set
	$$\mathcal Z (I\setminus F) := \mathcal U (\{I\}, \{I\cup \{ gg' \}^{\downarrow}  : gg'\in F \}),$$
	and $\mathcal Z (I):= \mathcal Z (I\setminus \emptyset) =\mathcal U (\{I\}, \emptyset)$.  
\end{enumerate}
	Note that the condition  $I\cup \{gg'\}^{\downarrow}\in \Y_0$ in (2) requires the $C$-compatibility of $I\cup \{gg'\}^{\downarrow}$, which is not automatic from the rest of the conditions. Observe also that for $I = \{v\}\in \Y_0(v)$, we have $\mathcal N (I)= \mathcal N (v)$, so the notations introduced in (1) and (2) are coherent.
   \end{notation}

 \begin{lemma}
	\label{lem:opencompactbasis}
	With the above notation, the family of sets $\mathcal Z (I\setminus F)$, where $I$ ranges on $\Y _0$ and $F$ ranges on all finite subsets of $\mathcal N (I)$,
	is a basis of open compact subsets of $\mathfrak F$. Each open compact subset of $\mathfrak F$ is a finite disjoint union of sets $\mathcal Z (I\setminus F)$.  
\end{lemma}
\begin{proof} Consider a non-empty set $\mathcal U (\{I\}, G)$, where $I\in \Y_0(v)$ and $G$ is a finite subset of $\Y_0(v)$ for some $v\in E^0$. We need to show that for each $Z\in \mathcal U (\{I\}, G)$ there exists a set of the form $\mathcal Z ( J\setminus F)$ such that $Z\in \mathcal Z ( J\setminus F)$ and $\mathcal Z ( J\setminus F)\subseteq \mathcal U (\{I\}, G)$. 
	
	Take $Z\in \mathcal U (\{I\},G)$.
	 
	Let $A$ be the family of those $h\in \max (L)$, for some $L\in G$, such that $h$ is $C$-incompatible with some $g_h\in Z$. Then we can write 
	$$g_h=p_hx_hr_h \in Z, \qquad h= p_hy_hq_h,$$
	where $x_h,y_h\in X_h\in C$, $x_h\ne y_h$. Observe that $p_hx_h\in Z$, and it ends in $E^1$. Moreover, $h\notin Z$, because $Z$ is $C$-compatible. . 
	
	Let $B$ be the family of those $h\in \max (L)$, for some $L\in G$, such that $h\notin Z$ and $h$ is $C$-compatible with all the elements of $Z$. Let $z_h$ be the greatest prefix of $h$ such that $z_h\in Z$ and $z_h$ does not end in $E^{-1}$. Since $h\notin Z$ and ends in $E^1$, there exists $g'_h\in \mathcal N (r(z_h))$ such that $z_hg'_h\le _p h$. Note that $z_hg'_h\in \FC (v)$ because $h\in \FC(v)$. Note also that $z_hg'_h\notin Z$ by maximality of $z_h$. 
	
	Now set 
	$$J= I \cup (\bigcup_{h\in A} \{p_hx_h\}^{\downarrow})\cup (\bigcup_{h\in B} \{z_h\}^{\downarrow}) \quad \text{ and } \quad F= \{ z_hg'_h : h\in B\}.$$ 
 Observe that $J$ is a finite lower set and that $J\subseteq Z$ because $Z$ is a lower subset of $\FC (v)$. Since $Z$ is $C$-compatible, so is $J$ and so by construction $J\in \Y _0(v)$.  Moreover $F$ is an eligible set for $J$, so that we may consider the  set
	 $\mathcal Z (J\setminus F)$. 
	 
	 We have to check that $Z\in \mathcal Z (J\setminus F)$, and that $\mathcal Z (J\setminus F)\subseteq  \mathcal U (\{I\}, Y)$. We have already observed that $J\subseteq Z$. Also $z_hg'_h\notin Z$ for all $h\in B$ by construction. Hence we obtain $Z\in \mathcal Z (J\setminus F)$. Let $Z'\in \mathcal Z (J\setminus F)$. Then $J\subseteq Z'$ and so $I\subseteq Z'$. Let $L\in G$, and suppose first that there is $h\in \max (L)$ such that $h$ is $C$-incompatible with $g_h\in Z$, that is, $h\in A$. Then we have $p_hx_h \in J\subseteq Z'$ and therefore $p_hy_h\notin Z'$, because $Z'$ is $C$-compatible, and consequently $h\notin Z'$. This shows that $L\nsubseteq Z'$. Finally suppose that  all elements of $\max (L)$ are $C$-compatible with $Z$. Since $Z\in \mathcal U (\{I\}, G)$, there exists $h\in \max (L)$ such that $h\notin Z$. Hence $h\in B$, and we may consider the element $z_hg'_h\in F$.  Since $Z'\in \mathcal Z (J\setminus F)$, it follows that $z_hg'_h \notin Z'$ and thus $h\notin Z'$, because $z_hg_h'\le_p h$. It follows that $L\nsubseteq Z'$. We conclude that $Z'\in \mathcal U (\{I\}, G)$, as desired.   
	
The fact that $\mathcal Z (I)$ is compact for $I\in \Y _0$ follows from the general theory. In fact, the space of filters $\mathfrak F$ can be identified with the space of characters on the semilattice $\YY$, through the consideration of the characteristic functions of the filters, see 
\cite{Exel:Inverse_combinatorial}*{Section~12}. Then the set $\mathcal Z (I)$ corresponds to the set of characters $\chi$ such that $\chi (I) = 1$, and this set is closed in the compact space $\{0,1\}^{\YY}$, and thus it is compact, see \cite{Exel:Inverse_combinatorial}*{Definition~10.2} and the comments after it. Consequently all the sets $\mathcal Z (I\setminus F)$ are open and compact. 

By \cite{ABPS}*{Lemma~3.9}, to prove that every open compact subset of $\mathfrak F$ can be written as a disjoint union of sets in the basis $\mathcal B := \{ \mathcal Z (I\setminus F) \}_{I,F} $, it suffices to check three properties:
\begin{enumerate}
	\item Every set $B\in \mathcal B$ is open and compact.
	\item $\mathcal B$ is closed under finite intersections.
	\item For $B_1,B_2\in \mathcal B$, the set $B_1\setminus B_2$ is a finite disjoint union of sets from $\mathcal B$. 
\end{enumerate}
      
We have already shown (1). To show (2), take $B_1= \mathcal Z (I_1\setminus F_1)$ and $B_2= \mathcal Z (I_2\setminus F_2)$, where $F_i$ is a finite subset of $\mathcal N (I_i)$, for $i=1,2$. We may assume that $B_1 \cap B_2 \neq \emptyset$. If $Z\in \mathfrak F$ and $Z	\in B_1\cap B_2$ then we have $I_1\cup I_2 \subseteq Z$ and consequently $I_1\cup I_2$ is $C$-compatible. It is then clear that $I_1\cup I_2 \in \Y _0$, because $\max (I_1\cup I_2)\subseteq \max (I_1) \cup \max (I_2)$. If $y\in F_1\cap I_2$ then $y\in I_2\subseteq Z$ and $y\notin Z$ because $y\in F_1$, leading to a contradiction. Hence $F_1\cap I_2 =\emptyset$ and similarly $F_2\cap I_1 =\emptyset$.
It follows that $(F_1\cup F_2) \cap (I_1\cup I_2) = \emptyset$ and we readily see that 
$$B_1\cap B_2 = \mathcal Z (I_1\setminus F_1) \cap \mathcal Z (I_2\setminus F_2) = \mathcal Z (I_1\cup I_2\setminus (F_1\cup F_2))\in \mathcal B, $$
as desired.  

Finally we show (3). Take $B_i=\mathcal Z (I_i\setminus F_i)$ as before. We may assume that 
$B_1\setminus B_2 \neq \emptyset$. We may also assume that $B_1\cap B_2 \neq \emptyset$, because if $B_1\cap B_2 = \emptyset$, then $B_1\setminus B_2 = B_1$ and we are done. As in the proof of (2), this gives that $I_1\cup I_2$ is necessarily $C$-compatible  and that $F_1\cap I_2 = F_2\cap I_1 = \emptyset$. 

Suppose first that $I_2\nsubseteq I_1$, and consider the set
$$\mathfrak M  := \max (I_2) \setminus I_1.$$
We can write each $h\in \mathfrak M $ in the form $h= g_0^hg_1^h\cdots g^h_{n(h)}$, where $g^h_0\in I_1$, $g^h_{i+1} \in  \mathcal N (r(g^h_i))$ for all $i=0,\dots , n(h)-1$, and $g^h_0[g^h_1]_1 \notin I_1$.
If $g_0^hg_1^h\in F_1$ for some $h\in \mathfrak M$ then $\mathcal Z (I_1\setminus F_1)\cap \mathcal Z (I_2) = \emptyset$ and hence $B_1\setminus B_2 = \mathcal Z (I_1\setminus F_1)$ as desired. Hence we may assume that
$F_1\cap \{ g^h_0g_h^1: h\in \mathfrak M \} = \emptyset$, which is indeed equivalent to $F_1\cap I_2 = \emptyset$. 

Denoting by $[i,j]$ the set $\{ i,i+1 ,\dots ,j\}$ for integers $i\le j$, we consider the set
$$\mathfrak N = \{ (i(h))_{h\in \mathfrak M} \in \Big( \prod _{h\in \mathfrak M} [0,n(h)]\Big) \setminus \{(n(h))_{h\in \mathfrak M}  \}  \}.$$      
For each $\mathfrak n = (i(h))_{h\in \mathfrak M}\in \mathfrak N$, we consider 
$$I_{\mathfrak n} = I_1 \cup \bigcup_{h\in \mathfrak M} \{g_0\cdots g_{i(h)}\}^{\downarrow}\in \mathcal X_0$$
and 
$$F_{\mathfrak n} = (F_1\cap \mathcal N (I_{\mathfrak n})) \sqcup \{ g^h_0\cdots g^h_{i(h)}g^h_{i(h)+1}: i(h) <n(h), h\in \mathfrak M\}\subseteq \mathcal N (I_{\mathfrak n}),$$
and the corresponding open compact sets $\mathcal Z (I_{\mathfrak n}\setminus F_{\mathfrak n})$.

In case $I_2\subseteq I_1$ we set $\mathfrak N = \emptyset$. 

Independently of whether $I_2\subseteq I_1$ or not, we need to consider some additional basic open compact subsets.  
Set 
$$\mathfrak H :=  \{ H\subseteq  F_2\setminus F_1 : H\ne \emptyset \text{ and } I_1\cup I_2 \cup H^{\downarrow} \in \Y _0  \},$$
and for each $H\in \mathfrak H$
$$F(H) = \mathcal N (I_1\cup I_2 \cup H^{\downarrow}) \cap (F_1\cup F_2).$$

 We claim that $$B_1\setminus B_2 = \mathcal Z (I_1\setminus F_1)\setminus \mathcal Z (I_2\setminus F_2) = \bigsqcup _{\mathfrak n \in \mathfrak N}\mathcal Z (I_{\mathfrak n}\setminus F_{\mathfrak n})\sqcup \bigsqcup _{H\in \mathfrak H} \mathcal Z (I_1\cup I_2\cup H^{\downarrow}\setminus F(H)).$$
Note that the union in the above formula is a disjoint union by construction.

 We prove first that $ \mathcal Z (I_{\mathfrak n}\setminus F_{\mathfrak n})\subseteq B_1\setminus B_2$ for each $\mathfrak n \in \mathfrak N$.  
Take $Z\in \mathcal Z (I_{\mathfrak n}\setminus F_{\mathfrak n})$. Then $I_1\subseteq I_{\mathfrak n}\subseteq Z$ and $Z\cap F_{\mathfrak n} = \emptyset$. So we get $I_1\subseteq Z$. If $z\in Z\cap F_1$, then $I_{\mathfrak n} \cup \{ z \}^{\downarrow}\subseteq Z$, 
so $I_{\mathfrak n} \cup \{z\}^{\downarrow}$ is $C$-compatible because $Z$ is so. Therefore $z\in Z\cap (F_1\cap \mathcal N (I_n))= \emptyset$, which is a contradiction. Therefore $Z\cap F_1= \emptyset$, proving that $Z\in B_1$. Moreover since $h\notin Z$ for at least one $h\in \max (I_2)$, we get that $I_2\nsubseteq Z$ and hence $Z\notin B_2$. Hence $Z\in B_1\setminus B_2$. Now we check that $\mathcal Z (I_1\cup I_2\cup H^{\downarrow}\setminus F(H)) \subseteq B_1\setminus B_2$ for $H\in \mathfrak H$. Take 
$Z\in \mathcal Z (I_1\cup I_2\cup H^{\downarrow}\setminus F(H))$. 
Then $I_1\cup I_2\cup H^{\downarrow}\subseteq Z$ and so $I_2\subseteq Z$ and $H\subseteq  Z$, 
with $\emptyset \ne H \subseteq  F_2$, so that $Z\notin \mathcal Z (I_2\setminus F_2)$. If $h'\in Z\cap F_1$, then $I_1\cup I_2\cup H^{\downarrow} \cup \{h'\}^{\downarrow} \subseteq Z$ and since $Z$ is $C$-compatible, it follows that $I_1\cup I_2\cup H^{\downarrow} \cup \{h'\}^{\downarrow} \in \Y_0$. We claim that $h'\notin I_1\cup I_2 \cup H^{\downarrow}$. Indeed we know that $h'\notin I_1\cup I_2 $ because $F_1\cap (I_1\cup I_2) =  \emptyset$ using our assumptions at the beginning of the proof of (3). Hence $h'\in \mathcal N (I_1\cup I_2)$. If $h'\in H^{\downarrow}$, then $h'$ is an initial segment of an element of $H$, but since the elements of $H$ are also contained in $\mathcal N (I_1\cup I_2)$, it follows that $h'\in H$. But then $h'$ belongs to $F_2\setminus F_1$ by definition of $H$, which contradicts the fact that $h'\in F_1$. We have shown that $h'\notin I_1\cup I_2 \cup H^{\downarrow}$ and therefore $h'\in \mathcal N (I_1\cup I_2 \cup H^{\downarrow})\cap F_1$, which implies that $h'\in F(H)$ by definition of $F(H)$. But, since $h'\in Z$, this contradicts the fact that $Z\in \mathcal Z (I_1\cup I_2\cup H^{\downarrow}\setminus F(H))$. Hence $Z\cap F_1 = \emptyset$ and thus $Z\in \mathcal Z (I_1\setminus F_1)= B_1$. Therefore $Z\in B_1\setminus B_2$, as desired.  

For the reverse containment, take $Z\in B_1\setminus B_2$. Then $I_1\subseteq Z$, $Z\cap F_1 = \emptyset$ and either $I_2\nsubseteq Z$ or $I_2\subseteq Z$ but $Z\cap F_2\ne \emptyset$. If $I_2\nsubseteq Z$ then we are in the case where $I_2\nsubseteq I_1$, and hence $\mathfrak N\ne \emptyset$. For each $h\in \mathfrak M$, let $i(h)\in [0,n(h)]$ be the maximum non-negative integer such that $g_0^h\cdots g^h_{i(h)}\in Z$. Since $I_2\nsubseteq Z$, we get that $i(h)< n(h)$ for a least one $h\in \mathfrak M$, so that $\mathfrak n :=(i(h))_{h\in \mathfrak M} \in \mathfrak N$. Hence we get that $I_{\mathfrak n}\subseteq Z$ and $Z\cap F_{\mathfrak n}= \emptyset$, so that $Z\in \mathcal Z (I_{\mathfrak n}\setminus F_{\mathfrak n})$. 

Finally suppose that $I_2\subseteq Z$ but $Z\cap F_2 \ne \emptyset$. Set $H:= Z\cap F_2= Z \cap (F_2\setminus F_1)$. Then $I_1\cup I_2\cup H^{\downarrow} \subseteq Z$ and since $Z$ is $C$-compatible, it follows that $H\in \mathfrak H$. It is now clear that $Z\in \mathcal Z (I_1\cup I_2 \cup H^{\downarrow}\setminus F(H))$, concluding the proof.    
     \end{proof}

Let $\mathfrak T$ be the closure of $\mathfrak U$ inside $\mathfrak F$. By Proposition \ref{prop:characterizing-filt-ultra-tight} the space $\mathfrak T$ is homeomorphic to the space of tight filters on $\YY$.  

We now characterize the space $\mathfrak T$  of tight filters. For this we need to generalize the notion of maximal configuration. 

Given a separated graph $(E,C)$, we write $\Cfin = \bigcup_{v\in E^0} \Cfin_v$, where  $\Cfin_v$ is the set of all $X\in C_v$ such that $|X|<\infty$. 

\begin{definition}
	Let $Z$ be a local configuration at $v\in E^0$. We say that $Z$ is {\it finite-maximal} if it is admissible 
    and for each $X\in \Cfin_v$ there exists $z_X\in X$ such that 
	$$\{z_X : X\in \Cfin_v \}\bigcup \{x^{-1}: x\in r^{-1}(v)\} \subseteq Z .$$
	 \end{definition}

Recall that an admissible local configuration cannot contain more than one element in each $X\in C_v$. Hence the above definition says that it is maximal with respect to the finite sets $X\in \Cfin_v$, in the sense that the local configuration already contains one element in each of these sets. To obtain a correct characterization of the tight filters, we also need to include some special vertices, which one can think of as ``virtual maximal-finite configurations''.

\begin{definition}
	\label{source-infinite}
	Let $(E,C)$ be a separated graph. Set
	$$\SInf = \{ v\in E^0 : r^{-1}(v)= \emptyset \,\,{\mathrm{ and }} \, \, |X| = \infty \, \mbox{ for all }\, X\in C_v \}.$$
\end{definition}

Since $E$ does not contain isolated vertices, if $v\in E^0$ is a \emph{source}, that is, if $r^{-1}(v)=\emptyset$, then it cannot be a \emph{sink}, meaning $s^{-1}(v)\not=\emptyset$, and therefore $C_v$ is also not empty. In particular we must have $|s^{-1}(v)|=\infty$ whenever $v\in \SInf$; we view these vertices as \emph{infinite sources}, which explains the notation.

\begin{theorem}
	\label{thm:charac-tightfilters}
	 Let $\mathfrak T'$ be the set of all non-trivial lower subsets $Z'$ of $\FC (v)$ such that $Z'_g$ is a finite-maximal local configuration for all $g\in Z$, where $v$ ranges over all vertices of $E$. Then there is a bijection 
	 $$\mathfrak T'\cup \SInf \cong \mathfrak T$$
	 between $\mathfrak T'\cup \SInf $ and the space $\mathfrak T$ of tight filters.      
\end{theorem}
\begin{proof}
We will carefully extend the bijection from $\mathfrak U'$ to $\mathfrak U$ defined in Theorem \ref{cor:ultrafilters} to a bijection from $\mathfrak T'\cup \SInf $ to $\mathfrak T$.

We first define a map $\varphi \colon \mathfrak T'\cup \SInf\to \mathfrak T$. Take $Z'\in \mathfrak T'$, and let $v\in E^0$ be the unique vertex such that $Z'$ is a non-trivial lower subset of $\FC (v)$. By definition, for each $g\in Z'$, the set $Z'_g$ is a finite-maximal local configuration.  Let $S$ be the set of all elements $g\in Z'$ such that $g$ ends in $E^{-1}$ and $g$ cannot be extended to an element $z\in Z'$ such that $z$ ends in $E^1$, and define $\varphi (Z') = Z'\setminus S$. Note that this agrees with the definition used in Theorem \ref{cor:ultrafilters}
for the set $\mathfrak U'$ of non-trivial lower subsets of $\FC$ such that all sets $Z_g$ are maximal local configurations.

  This map is completed by sending $v\in \SInf$ to $\varphi (v) = \{ v\} \in \mathfrak F (v)$.

 Write $Z= \varphi (Z') = Z'\setminus S$. Let us check that $Z\in \mathfrak T (v)$. As in the proof of Theorem \ref{cor:ultrafilters}, we see that $Z\in \mathfrak F (v)$. 
	We now show that $Z$ belongs to the closure of $\mathfrak U$ in $\mathfrak F$. Let $\mathcal Z (I\setminus F)$ be a basic open compact neighborhood of $Z$, where $I\in \Y _0$ and $F$ is a finite subset of $\mathcal N (I)$, see Notation \ref{notati:opencompactbasis}. Then $I\subseteq Z$ and $h\notin Z$ for each $h\in F$. We show that there exists $\widetilde{Z}\in \mathfrak U$ such that $Z\subseteq \widetilde{Z}$ and $h\notin \widetilde{Z}$ for each $h\in F$. 
	Indeed each element $h\in F$ can be written as $h=g_hg_h'\in \FC(v)$, where $g_h\in I$ and $g_h'\in \mathcal N (r(g_h))$, with $g_h[g_h']_1 \notin I$.
We rewrite these elements in  the form $h= g_h''y_h$, where $g_h''= g_hz_h$, where $z_h$ is a product (possibly empty) of inverses of edges of $E$.  
Observe that $g''_h\in Z'$, because $Z'\in \mathfrak T'$, so that each local configuration at an element of $Z'$ contains all inverses of edges ending at the range of that element. 

We now fix an element of the form $g''_h\in Z'$, $h\in F$, and extend the local configuration $Z'_{g_h''}$ of $Z'$ at $g_h''$, which is a finite-maximal local configuration, to a {\it maximal} local configuration $\mathfrak c$. Notice that there might be various elements $h'\in F$ such that $g''_{h'} = g''_{h}$ for the fixed element $h\in F$. Let $F'= F'_h$ be the set of such elements, so that we can write  
$$F'=\{g''_h y_1, \dots , g''_hy_p \},\quad \text{ with } y_1,\dots , y_p\in E^1. $$
We need that the extended local configuration $\mathfrak c$ satisfies $\mathfrak c \cap \{y_1, \dots, y_p\} = \emptyset $. Write $w:= r(g_h'')\in E^0$.
Let $x\in E^1\cup E^{-1}$ be the last letter of $g_h''$. If $x= f^{-1}$ for some $f\in E^1$, then $f\in Z'_{g_h''}$. Let $X_f$ be the unique element of $C_w$ such that $f\in X_f$.
Then $f$ is the unique element of $X_f\cap Z'_{g_h''}$. Let $X\in C_w$, and assume that $X\ne X_f$ in the case where $g_h''$ ends in $E^{-1}$. If $X\in \Cfin_w$, then since 
$Z'_{g_h''}$ is a finite-maximal configuration, there exists a unique $e_X\in X\cap Z'_{g_h''}$, and observe that by definition of $Z=\varphi (Z')$,
we have $g_h''e_X\in Z$. In particular we get that $e_X\ne y_i$ for $i=1,\dots ,p$, because $Z\cap F=\emptyset$. Now suppose that $|X| = \infty$ and that $X\cap Z'_{g_h''}\neq \emptyset$.
Then there is a unique element $e_X$ in $X\cap Z'_{g_h''}$, and as before $e_X\ne y_i$ for all $i$. Finally suppose that $|X| = \infty$ and that $X\cap Z'_{g_h''} = \emptyset$.
In this case, since $X$ is infinite, we may choose an element $e_X\in X$ such that $e_X\ne y_i$ for $i=1,\dots , p$. By means of these choices we obtain a maximal local configuration $\mathfrak c$ extending the local configuration $Z'_{g_h''}$ and such that $\mathfrak c \cap \{y_1,\dots , y_p\} = \emptyset$. 

Let $Z''$ be the union of $Z'$ and the set of all the elements $g_h''e_X$, where $e_X\in X$ are the choices of elements as before, for all $X\in C_w$ with $|X| = \infty$
and $X\cap Z'_{g_h''} =\emptyset$, where $h$ ranges on all elements of $F$. In this way, we have completed the local configuration of $Z'$ at each element $g_h''$ as above to a maximal local configuration $Z''_{g_h''}$, so obtaining a $C$-compatible lower subset $Z''$ of $\FC (v)$ such that $Z'\subseteq Z''$, and such that $F\cap Z'' =\emptyset$. Now $Z''$ can be further extended to $Z'''\in \mathfrak U'$. Set $\widetilde{Z} := \varphi (Z''')\in \mathfrak U$.  Then clearly $I \subseteq Z\subseteq \widetilde{Z}$ and $\widetilde{Z} \cap F= \emptyset$, so that $\widetilde{Z}\in \mathfrak U \cap \mathcal Z (I\setminus F)$, as desired.
	
A similar proof gives that all elements $\{v\}$, where  
$v\in \SInf$, belong to $\mathfrak T$. Indeed here a basic open compact neighborhood $\mathcal Z (I\setminus F)$ of $\{v\}$ must satisfy $I=\{v\}$ and $F$ must be a finite subset of $s^{-1}(v)$. The argument above obviously applies in this situation.   
	
	    We will now define the inverse map $\psi \colon \mathfrak T \to \mathfrak T'\cup \SInf$. First, observe that if $v$ is a sink in $E$ and $\{v \} \in \mathfrak T$, then $v$ must belong to $\SInf$, because if there exists $X\in \Cfin_v$ then taking $\mathcal Z (I\setminus F)$, where $I=\{v\}$ and $F= X$, we see that $\mathcal Z (I\setminus F)$ cannot contain any element from $\mathfrak U$. Hence we obtain a bijection between $\SInf$ and the set of sinks $v$ of $E$ such that $\{v\} \in \mathfrak T$.

Hence we concentrate in defining $\psi$ on the set of elements of $\mathfrak T$ which are not of the form $\{v\}$, for $v$ a sink in $E$. Let $Z\in \mathfrak T (v)$ be such element.  We define 
$$\psi (Z) = Z\bigsqcup \{ gx_1^{-1}\cdots x_n^{-1} \in \FC : g\in Z, n\ge 1, x_i\in E^1, \text{ and } gx_1^{-1} \notin Z \}.$$
Then $\psi (Z)$ is a non-trivial lower subset of $\FC (v)$, where $v$
is the vertex such that $Z\in \mathfrak F (v)$. Note that this extends the definition of $\psi$ in the proof of Theorem \ref{cor:ultrafilters}. We check that $Z'_h$ is a finite-maximal local configuration for all $h\in Z'\setminus Z$, where $Z':= \psi (Z)$. 
Write $h=gx_1^{-1} \cdots x_n^{-1}$, where $g\in Z$, $n\ge 1$, $x_i\in E^1$ and $gx_1^{-1}\notin Z$. Observe that $x_n\in Z'_h$. Let $X\in C_{r(h)}$ be such that $x_n\in X$.
If there exists $Y\in \Cfin_{r(h)}\setminus \{X\}$, then take
$$F: = \{ gx_1^{-1} \cdots x_n^{-1} y : y\in Y \}.$$
Then $Z\in \mathcal Z (g^{\downarrow} \setminus F)$, and since $Z\in \mathfrak T$, there exists $\widetilde{Z}\in \mathfrak U\cap \mathcal Z (g^{\downarrow} \setminus F)$. But this is impossible because the local configuration of $\widetilde{Z}$ at $h$ cannot be maximal, since $F\cap \widetilde{Z} = \emptyset$. Hence $\Cfin_{r(h)}= \{X\}$ if $X\in \Cfin$
 and $\Cfin_{r(h)} = \emptyset$ if $X\notin \Cfin$. In any case we see that the local configuration $Z'_h$ is a finite-maximal local configuration. 
Similarly, if $g\in Z$ and if $Y\in \Cfin _{r(g)}$, with $Y\ne X$ in case that the last letter of $g$ is of the form $e^{-1}$, for $e\in X\in C_{r(g)}$, then there must be an element $y\in Y\cap Z_g$, because otherwise a similar argument as before would contradict the fact that  $Z\in \mathfrak T$.  We therefore get that the local configuration $Z'_g$ is a finite-maximal local configuration also when $g\in Z$. We conclude that $Z' =\psi (Z) \in \mathfrak T'$, as desired. 

It is easy to show that $\varphi$ and $\psi$ are mutually inverse maps.  This concludes the proof. 
\end{proof}

\begin{remark}
    \label{rem:maximals-in-tight filters}
If $Z\in \mathfrak T$ and $g\in \max (Z)$, then $g$ must end in $E^1$ (if $g$ is non-trivial), and it follows from the proof of Theorem \ref{thm:charac-tightfilters} that all edges $y\in s_E^{-1}(r(g))$ belong to sets $X\in C_{r(g)}$ such that $|X|= \infty$. So all maximal elements of $Z$ are either sinks or ``$C$-infinite emitters'' in the above sense. (Compare with Remark \ref{rem:maximal-elements-of-ultrafilters}.)
\end{remark}
	  
 \begin{corollary}
	\label{cor:ultrafilters-closed} Let $(E,C)$ be a separated graph. Then the space $\mathfrak U$ of ultrafilters on the semilattice of idempotents of $\IS(E,C)$ is closed in the space of all filters if and only if $(E,C)$ is finitely separated, that is, all elements of the partition $C$ are finite sets.  
\end{corollary}

\begin{proof}
    Suppose first that $\mathfrak U =\mathfrak T$, that is, the space of ultrafilters is closed in the space of filters. If $v\in \SInf$, then $\{v\} \in \mathfrak T \setminus \mathfrak U$ by Theorem \ref{cor:ultrafilters} and Theorem \ref{thm:charac-tightfilters}. Hence $\SInf = \emptyset$. If $v\in E^0$ and there is $X\in C_v$ such that $|X| = \infty$, 
    then, since $v\notin \SInf$, we can define a finite-maximal local configuration $\mathfrak c$ at $v$ such that $\mathfrak c \cap X=\emptyset$. The finite-maximal local configuration $\mathfrak c$ can be extended to an element $Z'\in \mathfrak T'$ such that $Z'_v = \mathfrak c$, so that $Z'\notin \mathfrak U'$. Hence $Z: = \varphi (Z')\in \mathfrak T\setminus \mathfrak U$, which is a contradiction. Hence we obtain that $C= \Cfin$ and thus $(E,C)$ is finitely separated. 

    Conversely, if $(E,C)$ is finitely separated, then $\SInf = \emptyset$ and all finite-maximal local configurations are indeed maximal local configurations, so the result follows from Theorem \ref{cor:ultrafilters} and Theorem \ref{thm:charac-tightfilters}.
\end{proof}

\begin{example}
    \label{exam:tight spectrum-nonseparated}
Let $E$ be a non-separated graph. Then by Theorem \ref{thm:charac-tightfilters} and the remarks we made in Example \ref{exam:non-separated-filters-and-ultrafilters} and in Remark \ref{rem:maximals-in-tight filters}, the space of tight filters on the semilattice of idempotents of $\IS (E)$ can be identified with the union of the space of infinite paths with the space of finite paths ending in a singular vertex, where a singular vertex of $E$ is a vertex $v\in E^0$ such that $v$ is either a sink or an infinite emitter. Hence we recover a well-known result in the literature (see e.g. \cite{Exel:Partial_dynamical}*{Chapter~37}). Observe that in this case the space of maximal filters coincides with the space of tight filters if and only if the graph is row-finite, in the sense that each vertex emits only a finite number of edges. 
    \end{example}

\section{Algebras associated to separated graphs}
\label{sect:algebras}

In this section we study several algebras associated to a separated graph, and relate them to their (inverse) semigroups and associated topological groupoids. We mainly concentrate our study on the more tractable ``tame'' algebras of separated graphs.

Let $K$ be a commutative ring with unit. We will always assume that $K$ is endowed with a fixed involution that we denote by $\lambda\mapsto \overline \lambda$. We are going to work with $*$-algebras over $K$; these are algebras $A$ over $K$ endowed with an involution, that is, a conjugate $K$-linear anti-multiplicative and involutive map, generally denoted by $a\mapsto a^*$.

Recall from \cite{AG12} that given a separated graph $(E,C)$, we can associate to it the Cohn path algebra $\Co_K(E,C)$ and the Leavitt path algebra $\Le_K(E,C)$. More precisely, $\Co_K(E,C)$ is the universal $*$-algebra over $K$ generated by self-adjoint elements $v\in E^0$ and elements $e\in E^1$ satisfying the relations similar to those defining the semigroup $\S(E,C)$:
\begin{enumerate}
    \item[(V)] $vw=\delta_{v,w}v$ for all $v,w\in E^0$;
    \item[(E1)] $s(e)e= e= er(e)$ for all $e\in E^1$;
    \item[(E2)] $e^ *s(e)=e^* = r(e)e^*$ for all $e\in E^1$;
    \item[(SCK1)] $e^*f=\delta_{e,f}r(e)$ for all $e,f\in X$ with $X\in C$.
\end{enumerate}
The Leavitt path algebra $\Le_K(E,C)$ is the universal $*$-algebra over $K$ generated by the same generators and relations above plus the relation:
\begin{enumerate}
    \item[(SCK2)] $\sum_{e\in X}ee^*=v$ for every finite subset $X\in C_v$ with $v\in E^0$.
\end{enumerate}

Notice that $(E1)$ and $(E2)$ are equivalent as we deal with $*$-algebras. We decided to leave both conditions as they agree with the original definition for algebras in \cite{AG12}, where they are not equivalent.

In what follows we write $(E^1)^*=\{e^*:e\in E^1\}$ and use the convention that $s(e^*)=r(e)$ and $r(e^*)=s(e)$ for $e\in E^1$.
The algebras introduced above are $*$-algebras with the unique involution that makes all $v\in E^0$ self-adjoint and sends $e\mapsto e^*$. 

The (full) \emph{Toeplitz path \cstar{}algebra} $\To(E,C)$ is the enveloping $C^*$-algebra\footnote{Recall that the enveloping \cstar{}algebra of a $*$-algebra is the completion with respect to the largest \cstar{}(semi)norm, whenever this exists -- and it does exist for $*$-algebras generated by partial isometries.} of $\Co_\C (E,C)$, that is, the universal \cstar{}algebra generated by (self-adjoint) projections $v\in E^0$ and partial isometries $e\in E^1$ satisfying the relations (V), (E1), (E2) and (SCK1) above. And the (full) separated graph $C^*$-algebra $C^*(E,C)$ of $(E,C)$ is defined as the enveloping $C^*$-algebra of $\Le_\C (E,C)$, that is, the quotient of $\To(E,C)$ where we impose the extra relation (SCK2), see e.g. \cite{Ara-Goodearl:C-algebras_separated_graphs}.

Notice that the enveloping \cstar{}algebra $\To(E,C)$ of $\Co_\C(E,C)$, and hence also its quotient $C^*(E,C)$ do in fact exist. The reason is that the relations we impose on the generators imply that they are partial isometries, that is, $xx^*x=x$ holds for generators $x\in E^0\cup E^1$ of the $*$-algebra $\Co_\C(E,C)$. 

Let $S$ be a $*$-semigroup with zero. The $*$-algebra of $S$ over $K$, denoted here by $K[S]$ is the $*$-algebra over $K$ with the universal property that every zero-preserving $*$-homomorphism $S\to A$ into another $*$-algebra over $K$ extends uniquely to a $*$-homomorphism $K[S]\to A$. It can be concretely represented by formal finite sums of the form $\sum_{s\in S}\lambda_s s$ with $\lambda_s\in K$ and with the product and involution extending (conjugate) linearly the ones on $S$, where we use the convention that $0\in S$ also represents the zero of $K[S]$.\footnote{Formally $K[S]$ is therefore actually a quotient by the ideal generated by the zero element of $S$ viewed as an element of its usual $K$-algebra (where the zero of $S$ does not play the role of the zero of the algebra).} 

Recall that the involution on our $*$-semigroup $S$ is generally denoted by the pseudo-inverse notation $s\mapsto s^{-1}$, and in a $*$-algebra $A$ this is written as $a\mapsto a^*$. So when we talk about $*$-homomorphisms we mean a multiplicative map $\pi \colon S \to A$ satisfying $\pi(s^{-1})=\pi(s)^*$.

Let us specialize to complex algebras and go into \cstar{}algebras. In general, the $*$-algebra $\C[S]$ of an arbitrary $*$-semigroup $S$ 
need not admit an enveloping \cstar{}algebra as their representations might be ``unbounded''. This happens, for example, for the additive 
semigroup of naturals $(\N,+)$ with the trivial involution. However this is not a problem if we assume that $S$ is presented by generators 
$\G$ and relations $\R$ that contain the partial isometry relation $xx^*x=x$ for all $x\in \G$. Then the enveloping \cstar{}algebra  $C^*(S)$ of $\C[S]$ exists and it is the universal \cstar{}algebra for $*$-homomorphisms of $S$ preserving the zero.
Since $S$ is presented by $(\G,\R)$ we can also say that $C^*(S)$ is presented in the same way, as the universal \cstar{}algebra $C^*(\G,\R)$ generated by $\G$ satisfying the relations $\R$. 

The following result is then immediate from our definitions:

\begin{proposition}
Given a separated graph $(E,C)$, $\Co_K(E,C)$ is canonically isomorphic to the $*$-algebra $K[\S(E,C)]$ of the $*$-semigroup $\S(E,C)$. And $\To(E,C)$ is isomorphic to the universal \cstar{}algebra $C^*(\S(E,C))$ of the $*$-semigroup $\S(E,C)$. 
\end{proposition}

As a conclusion, we can realize the algebras $\Le_K(E,C)$ and $C^*(E,C)$ as certain quotients of $K[\S(E,C)]$ and $C^*(\S(E,C))$, respectively. However it is not clear how to describe their representations directly in terms of the semigroup $\S(E,C)$ as the Cuntz-Krieger relation (SCK2) involves a sum and does not make sense in the category of semigroups. For inverse semigroups, this is solved via the notion of tight representations, and that is what we shall do for algebras associated to the quotient inverse semigroup $\IS(E,C)$. 

One of the main points that make the algebras of $\S(E,C)$ rather intractable is that, although the generators of this $*$-semigroup satisfy the partial isometry relation $xx^*x=x$, this is no longer true for arbitrary elements of $\S(E,C)$. This is related to the commutativity of the elements $e(x):=xx^*$ with $x\in \S(E,C)$, that is, the relation $e(x)e(y)=e(y)e(x)$, also written as $[e(x),e(y)]=0$ for $x,y\in \S(E,C)$; notice that this is exactly the extra relation that gives rise to the quotient inverse semigroup $\IS(E,C)$. This naturally yields the following notion:

\begin{definition}
    Let $A$ be a $*$-algebra that is generated by a set $\G\sbe A$ of partial isometries. Let $S\sbe A$ be the $*$-semigroup generated by $\G$. Given $x\in S$, we write $e(x):=xx^*$. We consider the $*$-ideal $J_S$ of $A$ generated by the commutators $[e(x),e(y)]$ with $x,y\in S$. The \emph{tame quotient} of $A$ with respect to $S$ is defined as the quotient $*$-algebra $A_\ab:=A/J_S$. If $A$ is a \cstar{}algebra, we also define the tame \cstar{}algebra as the quotient \cstar{}algebra $C^*_\ab (A):=A/\overline J_S$. 
\end{definition}

\begin{remark}
\label{rem:equality-of-abs}
One can apply the above, in particular, to the $*$-algebra $K[S]$ of a $*$-semigroup $S$ that is presented by generators $\G$ and relations $\R$ containing the partial isometry relations $xx^*x=x$ for all $x\in \G$. In this case $K[S]_\ab$ is the universal $*$-algebra generated by $\G$ satisfying the relations $\R$ plus the tame relation: $[e(x),e(y)]=0$ for all $x,y\in S$. In other words, $K[S]_\ab\cong K[S_\ab]$, where $S_\ab$ is the quotient of $S$ by the congruence generated by $e(x)e(y)=e(y)e(x)$ for all $x,y\in S$. Notice that $S_\ab$ is an inverse semigroup. A similar result holds for \cstar{}algebras: $C^*_\ab(S)\cong C^*(S_\ab)$. 
\end{remark}

\begin{definition}
	\label{def:algebras}
	Let $(E,C)$ be a separated graph, and let $K$ be a commutative ring with unit. 
	The {\it abelianized (or tame) Cohn path algebra} $\Co_K^\ab(E,C)$ is the tame $*$-algebra associated to the Cohn path algebra $\Co_K(E,C)$ with respect to the image of $\S(E,C)$. 
	
	And with respect to the canonical images of the same $*$-semigroup, we define similarly the {\it abelianized (or tame) Leavitt path algebra} $\Le_K^\ab(E,C)$ as the tame algebra associated to the Leavitt path algebra $\Le_K(E,C)$.
	
	In the context of \cstar{}algebras, the {\it Toeplitz algebra} $\T (E,C)$ is the tame $C^*$-algebra associated to the Toeplitz path $C^*$-algebra $\mathcal{TP}(E,C)$. And the {\it tame graph $C^*$-algebra} $\OO(E,C)$ is the tame $C^*$-algebra associated to the graph $C^*$-algebra $C^*(E,C)$.  
\end{definition}

\begin{proposition}
	\label{prop:abelCohn-is-semigroup-algebra}
	Let $(E,C)$ be a separated graph. Then 
	$$\Co_K^\ab (E,C)\cong K[\IS(E,C)].$$
	Moreover, we have an isomorphism of $C^*$-algebras
	$$\T(E,C)\cong C^*(\IS(E,C)).$$
\end{proposition}
\begin{proof}
	Since $\IS(E,C)= \S(E,C)_\ab$, this follows directly from Remark \ref{rem:equality-of-abs}.  
\end{proof}

By the above proposition, there is a faithful identification of $\IS(E,C)$ with its image in $\Co_K^\ab(E,C)$; this will be not always the case for $\Le_K^\ab(E,C)$. Recall also that if $S$ is an inverse semigroup, then its complex $*$-algebra $\C[S]$ embeds faithfully into its \cstar{}algebra $C^*(S)$; indeed, by Wordingham's Theorem (see \cite{Paterson:Groupoids}*{Theorem~2.1.1}), even the $\ell^1$-algebra of $S$ embeds into the reduced \cstar{}algebra $C^*_r(S)$, defined as the \cstar{}algebra generated by its left regular representation. 

As a direct consequence, we can find a linear basis of the tame Cohn path algebra. For the Cohn path algebra $\Co_K(E,C)$ this was done in \cite{AG12}. In a forthcoming paper we plan to also find a linear basis for the Leavitt path algebra $\Le_K(E,C)$.

\begin{corollary}
	\label{cor:basisforCab} A $K$-basis of $\Co_K^\ab(E,C)$ is given by the set of elements $x$ in $\Co_K^\ab(E,C)$ which are in Scheiblich normal form:
	$$x= (\gamma_1 \gamma_1^*)\cdots (\gamma_n \gamma_n^*)\lambda,$$
	where $\{ \gamma_1,\dots , \gamma_n\}$ is an incomparable $C$-compatible family of $C$-separated (reduced) paths, $\lambda $ is a $C$-separated (reduced) path, $s(\gamma_j)= s(\lambda)$ for $j=1,\dots,n$,  the paths $\lambda_j$ do not end in $(E^1)^*$, and $(\lambda \lambda^*)(\gamma_i\gamma_i^*)= \gamma_i \gamma_i^*$ for some $i$.
\end{corollary}
\begin{proof}
	This follows immediately from Proposition \ref{prop:abelCohn-is-semigroup-algebra} and Theorem \ref{thm:uniqueness-SNF}.
\end{proof}

Given an inverse semigroup $S$ with zero, we denote by $\mathcal G(S)$ the universal groupoid (preserving the zero) of the inverse  semigroup $S$. This can be viewed as the groupoid of germs of the canonical action of $S$ on the (zero-preserving) spectrum $\dual\E$ of its idempotent semilattice $\E=\E(S)$, i.e. the space of zero preserving characters $\E\to \{0,1\}$ endowed with the pointwise convergence topology. 
Alternatively, we may also view this as a space of filters $\xi\sbe \E^\times$, see the beginning of Section \ref{sect:tightsection} and \cite{Exel:Inverse_combinatorial} for more information. The tight groupoid of $S$ is defined as the restriction of  $\G(S)$ to the tight spectrum $\Etight$ of $\E$, which is the closure of the subspace of ultracharacters. We denote this groupoid by $\mathcal G _{\text{tight}}(S)$. It was shown in \cite{Steinberg2010}*{Theorem~6.3} that the semigroup $*$-algebra $K[S]$ can be realized as the 
\emph{Steinberg algebra} of $\G(S)$: this is the $*$-algebra $A_K(\G(S))$ over $K$ linearly generated by characteristic functions $1_U$ of 
compact open bisections $U\sbe \G(S)$ with product and involution generated by the relations
$$1_U\cdot 1_V=1_{UV},\quad 1_U^*=1_{U^{-1}},$$
where $UV=\{gh: g\in U, h\in V, \s(g)=\rg(h)\}\sbe\G(S)$ and $U^{-1}=\{g^{-1}:g\in U\}\sbe \G(S)$ are the product and inversion as subsets of $\G(S)$ (see e.g. \cite{Steinberg2010}*{Section~4}). 

This leads to the following definitions.

\begin{definition}
	\label{def:tight-algebras}
	Let $S$ be an inverse semigroup.
	The {\it tight $*$-algebra} of $S$ over $K$ and the {\it tight $C^*$-algebra} of $S$ are defined as
	$$K_\tight[S]:=A_K(\G_\tight(S)), \qquad C^*_\tight (S) = C^* (\G_\tight (S)).$$
	This is the universal $*$-algebra ($C^*$-algebra) for ``tight'' representations of $S$, see \cite{Exel:Inverse_combinatorial} for more details.
	
	Given a separated graph $(E,C)$ we write
	$$\G(E,C) := \G(\IS(E,C)), \qquad 
	\G_\tight(E,C) := \G_\tight(\IS(E,C)).$$
\end{definition}

It follows directly from Proposition~\ref{prop:abelCohn-is-semigroup-algebra} and from the observations and definitions above that 
$$\Co^\ab_K(E,C)\cong A_K (\G(E,C))\quad\mbox{and}\quad \T (E,C)\cong C^*(\G(E,C)).$$
We are going to get a similar isomorphism for tight algebras later (see Theorem \ref{thm:iso-tight-algebras}). 

The groupoid $\G(E,C)$ has a special structure as a consequence of Theorem \ref{thm:ECMunntrees}, see Corollary \ref{cor:iso-Co-algebras} .
To show this, we need the following result, which follows from Proposition \ref{prop:KelLawson-result} and \cite{MilanSteinberg}*{Theorem~5.1}. We present a short self-contained proof here for the convenience of the reader. 

\begin{proposition}\label{pro:iso-groupoid-co-algebras} 
	If $S=\E\rtimes_\theta^r\Group$ is a restricted semidirect product of a partial action of a group $\Group$ on a semilattice with zero $\E$, then $\theta$ induces a partial action $\hat\theta$ on $\dual\E$ (the spectrum of $\E$, consisting of zero-preserving characters). In this situation, we have a canonical isomorphism $\G(S)\cong \dual\E\rtimes_{\hat\theta}\Group$, and then also
	\begin{equation}\label{eq:isomorphisms-algebras-groupoid}
		K[S]\cong A_K(\G(S))\cong C_K(\dual\E)\rtimes_{\alpha}^\alg\Group,\quad C^*(S)\cong C^*(\G(S))\cong \contz(\dual\E)\rtimes_\alpha\Group,
	\end{equation}
	where $\alpha$ is the induced partial action from $\hat\theta$ on the algebra of functions $C_K(\dual\E)$ or $\contz(\dual\E)$.
\end{proposition}
Here, for a totally disconnected locally compact Hausdorff space $X$, we write $C_K(X)$ for the space of functions $X\to K$ that are continuous and have compact support, where $K$ is endowed with the discrete topology; these are the \emph{locally constant functions}. This space is linearly generated by characteristic functions of compact open subsets of $X$; in other words, $C_K(X)=A_K(X)$ if we view $X$ as the trivial groupoid with only units.
\begin{proof}
	Recall that $\G(S)$ is the groupoid of germs $\hat{\E}\rtimes_{\hat c} S$ for the canonical action $c$ of $S$ on its semilattice of idempotents $\E=\E(S)$. This action is given by the semilattice isomorphisms $c_s\colon \E_{s^{-1}}\congto \E_s$ given by $c_s(e):=ses^{-1}$, where $\E_s$ is the (principal) ideal $ss^{-1}\E\idealin\E$. This induces the dualized action $\hat c$ of $S$ on $\dual\E$ via $\hat c_s\colon \dual\E_{s^{-1}}\to \dual\E_s$ by $\hat c_s(\chi)=\chi\circ c_{s^{-1}}$. Here $\dual\E_s$ is seen as an open subset of $\dual\E$ by the identification $\dual\E_s = \{ \chi \in \dual\E : \chi (ss^{-1})= 1 \}$, see e.g. \cite{LaLonde-ConditionK}.
	
	The groupoid $\G(S)$ is, by definition, the groupoid of germs of the action $\hat c$ (see \cite{Exel:Inverse_combinatorial}); this consists of equivalence classes $[\chi,s]$ of pairs $\chi\in \dual\E$, $s\in S$ with $\chi(ss^{-1})=1$, where $[\chi,s]=[\chi,t]$ if there is $e\in \E$ with $\chi(e)=1$ and $es=et$\footnote{We change the convention in comparison with the usual one, see for instance 
		\cite{Exel:Inverse_combinatorial}, in order to match our convention for semidirect products of partial actions.}.
	If this is the case, then $\sigma(s)=\sigma(es)=\sigma(et)=\sigma(t)$, where $\sigma\colon S^\times =(\E\rtimes^r_\theta\Gamma)^\times  \to \Gamma$ is the canonical partial grading homomorphism, that is, $\sigma(e,g)=g$. This implies that the map $\phi\colon \G(S)\to \dual\E\rtimes_{\hat\theta}\Gamma$, $\phi[\chi,s]:=(\chi,\sigma(s))$, is well defined, and straightforward computations show that $\phi$ is an isomorphism of topological groupoids. Here $\hat\theta$ is the partial action dualized from $\theta$, that is, $\hat\theta_g\colon \dual D_{g^{-1}}\to \dual D_g$, $\hat\theta_g(\chi):=\chi\circ\theta_{g^{-1}}$. Hence $\G (S)$ is the transformation groupoid $\dual\E\rtimes_{\hat\theta}\Gamma$ of the partial action $\hat\theta$ of $\Gamma $ on $\dual\E$. 
	
	Now the isomorphisms~\eqref{eq:isomorphisms-algebras-groupoid} follow from well-known results from the literature describing the algebra of the transformation groupoid of a partial action, see \cite{BeuterGoncalves}*{Theorem~3.2} for the isomorphism involving the Steinberg algebra, and \cite{Abadie}*{Theorem~3.3} for the \cstar{}algebra case.
\end{proof}

Combining the above proposition with Theorems \ref{thm:ECMunntrees} and  \ref{thm:semidirect-product}, we then get the following:

\begin{corollary}\label{cor:iso-Co-algebras}
	Let $(E,C)$ be a separated graph, and set $\E = \E (\IS(E,C))$. Then there are canonical isomorphisms
	\begin{equation}
		\label{eq:iso-Co1}
		\Co^\ab_K(E,C)\cong C_K(\dual\E)\rtimes_{\alpha}^\alg \Free \cong C_K(\dual\YY)\rtimes_{\alpha}^\alg \Free
	\end{equation}
	and
	\begin{equation}
		\label{eq:iso-Co2}
		\T(E,C)\cong \contz(\dual\E)\rtimes_\alpha\Free \cong \contz(\dual\YY)\rtimes_\alpha\Free,
	\end{equation} 
	where $\alpha$ is the partial action of $\Free$ on $\E\cong \YY$ (as in Theorem \ref{thm:ECMunntrees}), dually induced to $\dual\E$ and then to either the algebra $\Co_K(\dual\E)$ or the \cstar{}algebra $C_0(\dual\E)$.
\end{corollary}

\begin{remark}\label{rem:reduced-Toeplitz}
	The \cstar{}algebra $\T(E,C)\cong \contz(\dual\E)\rtimes_\alpha\Free$ is, in general, not nuclear, not even exact. For instance, if $E$ is the graph with a single vertex and a set $X$, $|X|>1$, of edges endowed with the free separation $C$ as in Example~\ref{exam:free-separation-first-one}, then $\T(E,C)$ 
	is not exact because the quotient \cstar{}algebra $\OO(E,C)$ is the full \cstar{}algebra $C^*(\F)$ of the free group $\F=\F(X)$, hence not exact (see \cite{BrownOzawa}*{Corollary~3.7.12}).
	This means that the groupoid $\G(E,C)$ is not amenable and the reduced \cstar{}algebra $\T_r(E,C):=C^*_r(\G(E,C))$ is different from $\T(E,C)$, in general (see \cite{BrownOzawa}*{Corollary~5.6.17 and Theorem~5.6.18}). Indeed, by \cite{XinLi}*{Proposition~3.1}, the isomorphism~\eqref{eq:iso-Co2} factors through the isomorphism
	$$\T_r(E,C)\cong \contz(\dual\E)\rtimes_{\alpha,r}\Free.$$
	Moreover, by \cite{BussFerraroSehnem}*{Theorem~3.10}, we have $\T(E,C)=\T_r(E,C)$ if and only if $\alpha$ is amenable (i.e, has the ``approximation property'' in the language of \cite{BussFerraroSehnem}), which means that the groupoid $\G(E,C)$ is amenable, and this happens if and only if $\T_r(E,C)$ is nuclear.
\end{remark}

We now introduce the idempotents of $\Co_K^{\text{ab}} (E,C)$ associated to the cylinder sets $\mathcal Z (I\setminus F)$ introduced in Notation \ref{notati:opencompactbasis}. 

\begin{definition}
	\label{def:idempotents-IsetminusF}
	Let $(E,C)$ be a separated graph and let $I\in \Y_0$. 
	\begin{enumerate}
		\item The idempotent $\e (I)$ associated to $I$ is 
		$$\e (I) = \prod _{\lambda\in \max (I)} \lambda \lambda^*.$$ 
		\item Suppose now that $\lambda \in \mathcal N (I)$. Then there is a unique decomposition $\lambda = \lambda_0\lambda_1$, where $\lambda_0\in I$, $\lambda_1= x_1^{-1}\cdots x_r^{-1}y$, $r\ge 0$, $x_1,\dots , x_r,y\in E^1$ and $\lambda_0x_1^{-1}\notin I$. Define
		$$\e (\lambda_0\setminus \lambda_1) = \lambda_0\lambda_0^*-\lambda \lambda^* = \lambda_0 (r(\lambda_0) - \lambda_1\lambda_1^*)\lambda_0^* \in \Co_K^\ab(E,C).$$
		\item For a finite subset $F$ of $\mathcal N (I)$, define the idempotent associated to $\mathcal Z (I\setminus F)$ by
		$$\e (I\setminus F) = \e (I)\cdot \prod_{\lambda_0\lambda_1\in F} \e(\lambda_0\setminus \lambda_1) = \prod _{\lambda_0\lambda_1\in F} \e (I\setminus \{\lambda_0\lambda_1\}).$$
		\item For $v\in E^0$ and $e\in s^{-1}(v)$ we define
		$$\e (v\setminus e) = v-ee^*.$$
		\item For $v\in E^0$ and $X\in \Cfin_v$, define
		$$q_X = \e (v\setminus X) = \prod_{e\in X} \e (v\setminus e).$$
	\end{enumerate}
\end{definition}

Notice that $e^*f=0$ for $e\not=f$ in the same set $X\in C_v$. Hence
$$q_X=v-\sum_{e\in X}ee^*.$$
This projection is therefore related to the Cuntz relation $\sum_{e\in X}ee^*=v$, and in particular the image of this projection is zero in the Leavitt path algebra $\Le_K^\ab(E,C)$. Indeed, we have:

\begin{lemma}\label{lem:quotient-ideal-Cuntz-rel}
	$\Le_K^\ab(E,C)\cong \Co_K^\ab(E,C)/\mathcal{Q}$, where $\mathcal Q$ is the $*$-ideal of $\Co_K^\ab(E,C)$ generated by the projections $q_X$ with $X\in \Cfin$.
\end{lemma}
\begin{proof}
	$\Le_K^\ab(E,C)$ is the quotient of $\Co_K^\ab(E,C)$ by the $*$-ideal generated by the Cuntz relations $\sum_{e\in X}ee^*=v$ for $X\in \Cfin_v$, $v\in E^0$, and by the previous observations this equals the ideal $\mathcal Q$ defined as in the statement.
\end{proof}

\begin{remark}
	\label{rem:e-for-X}
	It is convenient, for some computations, to extend the definition of the idempotent map $\e$ to general elements of $\Y$, with the same formulas. Indeed for $I\in \Y$ we have
	$$\e(I) := \prod _{\mu \in \max (I)} \mu \mu^* =\prod _{\lambda \in I} \lambda \lambda^*
	= \prod _{\lambda_0 \in I_0} \lambda_0 \lambda_0^* = \prod _{\mu \in \max (I_0)} \mu \mu^*= \e(I_0),$$
	where we have used the notation introduced in Lemma \ref{lem:maxelements}.
\end{remark}

We now obtain, using  Lemma \ref{lem:opencompactbasis}, a complete description of the Boolean algebra of idempotents of the commutative subalgebra of $\Co_K^\ab(E,C)$ generated by the set of idempotents $\mathcal E$ of $\IS(E,C)$. 

\begin{proposition}
	\label{prop:idempotents-in-CKEhat0} Let $K$ be an indecomposable commutative ring with unit.
	Let $(E,C)$ be a separated graph and let $\mathcal C$ be the subalgebra of $\Co_K^\ab(E,C)= K[\IS(E,C)]$ generated by $\mathcal E= \mathcal E (\IS (E,C))$. Then we have $\mathcal C \cong C_K(\hat{\mathcal E})$ and any idempotent of $\mathcal C$ is an orthogonal finite sum of idempotents of the form $\e (I\setminus F)$, for $I\in \Y_0$ and $F$ a finite subset of $\mathcal N (I)$. 
\end{proposition}

\begin{proof} The algebra $\mathcal C$ is the image of $C_K(\dual\E)$ under the isomorphism $C_K(\dual\E)\rtimes_{\alpha}^\alg \Free
	\cong \Co^\ab_K(E,C)$ of Corollary \ref{cor:iso-Co-algebras}. Since $K$ is indecomposable, the unique idempotents of $K$ are $0$ and $1$ and it follows that the idempotents of $C_K(\dual\E)$ are exactly the characteristic functions of open-compact subsets of $\dual\E$.  
	Moreover each open-compact set $\mathcal Z (I\setminus F)$ corresponds to the idempotent $\e (I\setminus F)$ in $\Co_K^\ab (E,C)$. Hence the result follows from Lemma \ref{lem:opencompactbasis}. 
\end{proof}

\subsection{Leavitt path algebras and tight groupoids}
We now prepare to show our main result in this section, stating that $\Le^\ab_K(E,C)$ is isomorphic to the Steinberg algebra of the tight groupoid $\mathcal G_\tight (E,C)$ of $(E,C)$, and similar results for both the full and the reduced tame $C^*$-algebras $\OEC$ and $\OECR$. To show these results, the essential ingredient is Theorem \ref{thm:Leavitt-universal-tight} below. 

We first define tight representations of an inverse semigroup, see e.g. \cite{Exel:Inverse_combinatorial}.

\begin{definition}
	\label{def:tight-repres} Let $S$ be an inverse semigroup with zero, and denote by $\mathcal E$ its semilattice of idempotents.
	\begin{enumerate}
		\item A {\it cover} of $p\in \mathcal E$ is a finite subset $Z\sbe \{q\in \E: q\leq p\}$ such that for all $0\not= q\leq p$, there exists $z\in Z$ with $zq\not=0$.
		\item Let $\pi\colon S\to A$ be a $*$-homomorphism from an inverse semigroup (with zero) $S$ into a $K$-algebra with involution $A$. The image of the idempotents $\pi(\E)\sbe A$ generates a (generalized)\footnote{Meaning that the Boolean algebra may not have a unit, see \cite{Exel:Tight-cover}.} Boolean algebra of idempotents of $A$ with respect to $p\wedge q:=pq$ and $p\vee q:=p+q-pq$. The $*$-homomorphism $\pi$ is  \emph{tight} if its restriction to $\E=\E(S)\sbe S$ is tight, meaning that it preserves finite \emph{covers}, that is, for every finite cover $Z$ of $p$ in $\E$, we have
		$$\bigvee_{z\in Z}\pi(z)=\pi(p).$$
	\end{enumerate}
	
	We can now state the main technical result of this section.
\end{definition}

\begin{theorem}\label{thm:Leavitt-universal-tight}
	Let $(E,C)$ be a separated graph and consider its associated inverse semigroup $\IS(E,C)$. Let $K$ be a commutative ring with unit and let $A$ be a $*$-algebra over $K$. 
	Then a $*$-homomorphism $\pi\colon \IS(E,C)\to A$ is tight if and only if it extends to a $*$-homomorphism $\Le_K^\ab(E,C)\to A$.
\end{theorem}
\begin{proof}
	First assume that the $*$-homomorphism $\pi\colon \IS(E,C)\to A$ is tight. We prove that for each finite subset $X\in \Cfin_v$, the set $Z:=\{ee^{-1}: e\in X\}$ is a finite cover of $v\in \E=\E(\IS(E,C))$. Indeed, it is clear that $ee^{-1}\leq v$ for all $e\in E^1$ with $s(e)=v$. We know that non-zero idempotents of $\IS(E,C)$ can be uniquely represented in the form of a finite product 
	$$p:=\prod_{\alpha\in \max (I_0)}\alpha\alpha^{-1},$$
	where $[I]\in \YY$, see Theorem~\ref{thm:semidirect-product}, so that each $\alpha\in \max (I_0)$ does not end in $E^{-1}$. If $p\leq v$, then this means that $I_0\in \mathcal X _0(v)$. Suppose first that there is some $\alpha\in \max (I_0)$, with $\alpha=\alpha_1\cdots\alpha_n$ for $\alpha_i\in E^{1}\sqcup E^{-1}$ and $\alpha_n\in E^1$, such that $\alpha_1\in X$.
	Then $$\alpha_1\alpha_1^{-1}p = p\ne 0,$$
	as desired. Suppose now that no $\alpha \in \max (I_0)$ starts with an edge in $X$. Then we may consider the set
	$$I' = I_0 \cup \{x\},$$
	where $x$ is an arbitrary (but fixed) element of $X$. Observe that $I'\in \mathcal X_0 (v)$ and $xx^{-1} p = \e (I') \ne 0$.
	Hence we have shown that $Z$ is a finite cover of $v$. Hence a tight $*$-homomorphism $\pi\colon \IS(E,C)\to A$ satisfies the relation
	$$\bigvee_{e\in X}\pi(ee^{-1})=\pi(v)$$
	for all $v\in E^0$ and $X\in \Cfin_v$. This means that the extension of $\pi$ to $K[\IS(E,C)]\cong \Co^\ab_K(E,C)$ vanishes on the $*$-ideal $\mathcal Q\idealin \Co^\ab_K(E,C)$ generated by the differences $\sum_{e\in X} ee^* -v$ with $v\in E^0$ and $X\in \Cfin_v$. And by Lemma~\ref{lem:quotient-ideal-Cuntz-rel}, we have $\Co^\ab_K(E,C)/\mathcal Q\cong \Le_K^\ab(E,C)$. Therefore $\pi$ extends to a $*$-homomorphism of $K$-algebras $\Le_K^\ab(E,C)\to A$, as desired.
	
	For the converse we have to show that the canonical representation $\rho \colon \IS(E,C) \to \Le_K^\ab (E,C)$ is tight. 
	For this part of the proof, we will extensively use the notion of $C$-compatiblility introduced in Section \ref{sect:InvSemSG}.
	Recall that we say that $I,J\in \mathcal X_0$ are $C$-compatible, or that $I$ is $C$-compatible with $J$, in case $I\cup J$ is $C$-compatible.  
	Given $I,J\in \mathcal X_0$, notice that $\e (I) \e (J) \ne 0$ if and only if $I$ and $J$ are $C$-compatible. 
	
	Let $Z$ be a finite cover of some nonzero idempotent $p$ in $\mathcal E$. We know that we may represent $p$ in the form $p= \e (I) = \prod _{\lambda \in \max (I)} \lambda \lambda^{-1}$ for a unique $I\in \mathcal X _0(v)$. Since $z\le p$ for each $z\in Z$, we have $z= \e (I_z)$ for a unique $I_z\in \mathcal X _0$, with $I\subseteq I_z$. Note that the condition that $Z$ is a finite cover is equivalent to saying that each $J\in \mathcal X_0$ such that $I\subseteq J$ must be $C$-compatible with $I_z$, for some $z\in Z$. 
	
	We start by showing that we can assume that $Z$ does not contain certain types of idempotents. Of course we can assume that $I\subsetneq I_z$ for all $z\in Z$. For $J\in \mathcal X_0$ with $I\subsetneq J$,
	we write every $\alpha \in J\setminus I$ in the form $\alpha = \alpha_I \alpha_{J\setminus I}$, where $\alpha_I$ is the  maximum initial segment of $\alpha$ such that $\alpha_I \in I$. 
	We say that $z\in Z$ is {\it redundant} if $Z\setminus \{z\}$ is a finite cover of $p$. Observe that we can assume that $Z$ does not contain any redundant element. Now we show that any $z\in Z$ such that 
	$I_z$ contains an element $\alpha \in I_z\setminus I$ such that $\alpha_{I_z\setminus I}= \gamma_0 e \gamma_1$, where $e\in Y$ and $|Y|= \infty$ is redundant. Suppose for contradiction that there is $z_0\in Z$ such that $I_{z_0}$ contains an element $\alpha \in I_{z_0}\setminus I$ such that $\alpha_{I_z\setminus I}= \gamma_0 e_0 \gamma_1$, where $e_0\in Y_0$ and $|Y_0|= \infty$, and such that $z_0$ is not redundant. This means that there exists $J\in \mathcal X _0$ such that $I\subseteq J$, $J$ is $C$-compatible with $I_{z_0}$, and $J$ is $C$-incompatible with all $I_{z}$ for $z\in Z\setminus \{z_0\}$.
	Since $I_{z_0}$ and $J$ are $C$-compatible, by replacing $J$ with $J\cup I_{z_0}$, we can assume that $I_{z_0}\subseteq J$. In particular we get that $\alpha_I \gamma _0 e_0 \in J$. For each $f\in Y_0\setminus \{e_0\}$, we consider a new element $J_f \in \mathcal X _0$, with $I\subseteq J_f$, defined as follows:
	$$J_f = (J\setminus \{\alpha_{I} \gamma_0 e_0 \lambda : \alpha_{I} \gamma_0 e_0 \lambda\in J \} ) \cup \{ \alpha_I \gamma_0 f \}.$$
	The element $J_f$ is obtained by killing all the branches of $J$ corresponding to $\alpha_I \gamma_0 e_0$, and replacing $\alpha_I \gamma_0 e_0$ with $\alpha_I \gamma_0 f$. Observe that $J_f$ is $C$-incompatible with $I_{z_0}$, so it must be $C$-compatible with $I_z$ for some $z\in Z\setminus \{z_0\}$, because $Z$ is a cover of $p=\e (I)$. Let $Z(f)$ be the subset of elements of $Z$ which are $C$-compatible with $J_f$. We claim that, for each $z\in Z(f)$, $I_z$ does contain $\alpha_I \gamma_0 f$. If $I_z$ does not contain $\alpha_I \gamma_0 f$, then it cannot contain any path of the form $\alpha_I \gamma_0 f'$, for $f'\in Y_0$, because $J_f$ and $I_z$ are $C$-compatible. In particular $I_z$ does not contain $\alpha_I \gamma_0 e_0$, and hence does not contain any path in the branch of $\alpha_I \gamma_0 e_0$. 
	But since $J$ and $I_z$ are $C$-incompatible there exists two different edges $g_1,g_2\in X$, for some $X\in C$, such that $\beta_1 = \beta_I \tau_0 g_1 \in J$ and  $\beta_2 = \beta_I \tau_0 g_2\in I_z$. But then $\beta_I \tau_0 g_1 \in J\setminus \{\alpha_{I} \gamma_0 e_0 \lambda : \alpha_{I} \gamma_0 e_0 \lambda\in J \}$, and thus $\beta_I \tau_0 g_1 \in J_f$, so that $J_f$ and $I_z$ are $C$-incompatible, which contradicts the fact that $z\in Z(f)$. Hence we get that $\alpha_I \gamma_0 f\in I_z$ for all $z\in Z(f)$.
	But now we see that any two $I_{z_1}$ and $I_{z_2}$, where $z_1\in Z(f_1)$ and $z_2\in Z(f_2)$ for distinct $f_1,f_2\in Y_0\setminus \{e_0\}$ must satisfy that $I_{z_1}$ and $I_{z_2}$ are $C$-incompatible,
	because $\alpha_I \gamma_0 f_1\in I_{z_1}$ and $\alpha_I \gamma_0 f_2\in I_{z_2}$, and $f_1,f_2$ are distinct elements of $Y_0\in C$; and it follows that $z_1z_2 = 0$. We infer that the subsets 
	$$\{Z(f): f\in Y_0\setminus \{e_0\} \}$$ 
	of $Z$ are mutually orthogonal. Since all the sets $Z(f)$ are non-empty, $Z$ is finite, and $Y_0$ is infinite, we obtain a contradiction.    
	We have thus shown that any $z\in Z$ such that 
	$I_z$ contains an element $\alpha \in I_z\setminus I$ such that $\alpha_{I_z\setminus I}= \gamma_0 e \gamma_1$, where $e\in Y$ and $|Y|= \infty$ is redundant. Therefore we can assume that all elements of $Z$ have the property that, for each $\alpha = \alpha_I \alpha_{I_z\setminus I}\in I_z$, all edges appearing (with positive exponent) in the path  $\alpha_{I_z\setminus I}$ belong to some $X\in \Cfin$. 
	
	In preparation for the final part of the proof, we show the following claim.
	
	\medskip
	
	\noindent \underline{Claim 1:}
	Suppose that $Z$ is a finite cover of $p$, and suppose that $z\in Z$ is such that there is a $C$-separated path $\alpha\in \FC$ of the form $\alpha = \alpha_0 \mu e$, with $\alpha \notin I_z$, such that $\alpha_0\in I_z$, $\mu$ is a (possibly empty) product of inverse edges $a^{-1}$, for $a\in E^1$, and $e\in X\in \Cfin$. Suppose that, for all $f\in X$, we have $\alpha_0 \mu f \notin I_z$. Set $\alpha (f) = \alpha_0 \mu f$, which is a $C$-separated path, and $I_{z,\alpha(f)} := I_z\cup \{\alpha(f)\}^{\downarrow}\in \mathcal X _0$, and $z_{\alpha(f)} = \e (I_{z,\alpha(f)})$.
	
	Then 
	$$Z' = (Z\setminus \{z\})\cup \{ z_{\alpha(f)} : f\in X\}$$
	is also a finite cover of $p$, and also $\bigvee_{z\in Z}  \rho (z) = \bigvee_{z'\in Z'} \rho (z')$.
	
	\medskip

	\noindent \underline{Proof of Claim 1:}  Let $J\in \mathcal X_0$ such that $I\subseteq J$. Then $J$ is $C$-compatible with $I_{z_1}$ for some $z_1\in Z$. If $z_1\ne z$ we are done. If $z_1= z$ then $J$ must be $C$-compatible with at least one of the elements $I_{z,\alpha(f)}$, for some $f\in X$. This shows that $Z'$ is a finite cover of $p$. Now observe that in $\Le_K^\ab (E,C)$, we have
	\begin{align*}
		\rho (\e (I_z)) = \alpha_0\alpha_0^{*}\rho (\e (I_z)) &
		= \alpha_0\mu \mu^{*} \alpha_0^* \rho (\e (I_z))\\
		&= \Big( \sum_{f\in X} \alpha_0 \mu ff^*\mu^* \alpha_0^*\Big) \rho (\e (I_z)) =  
		\bigvee _{f\in X} \rho (\e (I_{z,\alpha(f)})).
	\end{align*}
	This shows that $\rho (z) = \bigvee _{f\in X} \rho (\e (I_{z,\alpha (f)}))= \bigvee _{f\in X} \rho (z_{\alpha(f)})$ and hence $\bigvee_{z\in Z}  \rho (z) = \bigvee_{z'\in Z'} \rho (z')$.  
	This concludes the proof of the Claim 1. 
	
	\medskip

	We now will show that, given a finite cover $Z$ of $p$, there exists another finite cover $Z'$ of $p$, with $\bigvee _{z\in Z} \rho (z) = \bigvee_{z'\in Z'} \rho (z')$, and a finite lower subset $\mathfrak Z$ of $\F_C(v)$, not necessarily $C$-compatible, with $I\subseteq \mathfrak Z$, such that the following conditions are satisfied:
	
	\begin{enumerate}
		\item[(P1)]  For each $z'\in Z'$ and each $\alpha = \alpha_I \alpha_{I_{z'}\setminus I}\in I_{z'}$, all edges appearing (with positive exponent) in the path  $\alpha_{I_{z'}\setminus I}$ belong to some $X\in \Cfin$. 
		\item[(P2)] For all $z'\in Z'$ and for all $\alpha \in \mathfrak Z$ such that $\alpha = \alpha_0 \mu e$, with $\alpha_0\in I_{z'}$,
		$\mu$ is a (possibly empty) product of inverse edges, and $e\in X \in \Cfin$, there exists $f\in X$ such that $\alpha_0 \mu f\in I_{z'}$.
		\item[(P3)] $\mathfrak Z \subseteq \bigcup _{z'\in Z'} I_{z'}$, and for each $\alpha = \beta e \in I_{z'}$, with $e\in X\in \Cfin$ and $z'\in Z'$, there exists $f\in X$ such that $\beta f \in \mathfrak Z$.   
	\end{enumerate}
	
	Note that conditions (P1) and (P3) together imply that no maximal element of $\mathfrak Z$ ends in $E^{-1}$. Indeed suppose that $\alpha \in \max (\mathfrak Z )$. Then we have $\alpha \in I_{z'}$ for some $z'\in Z'$ by the first part of (P3). Hence there is a maximal element $\gamma = \beta e$ of $I_{z'}$, with $e\in E^1$, such that $\alpha$ is an initial segment of $\gamma $. If $\gamma \in I$ then $\alpha = \gamma \in \max (I)$, because $I\subseteq \mathfrak Z$, and thus $\alpha$ does not end in $E^{-1}$. If $\alpha \notin I$, then also $\gamma \notin I$, and thus $e\in X\in \Cfin$ by (P1). Therefore by the second part of (P3), there is $f\in X$ such that $\beta f \in \mathfrak Z$. If $\alpha $ is an initial path of $\beta$, then $\alpha $ is a proper initial segment of $\beta f \in \mathfrak Z$, which contradicts the maximality of $\alpha$. Hence $\alpha = \beta e$ does not end in $E^{-1}$.
	
	We now show that, given any finite cover $Z$ of $p$, there is a finite cover $Z'$ and a set $\mathfrak Z$ satisfying (P1)--(P3). By the first part of the proof, we know that we can assume that $Z$ satisfies (P1). Assuming this property, we define
	$$\mathfrak Z = \bigcup_{z\in Z} I_z.$$
	Then $\mathfrak Z$ is a finite lower subset of $\F_C(v)$, not necessarily $C$-compatible, with $I\subseteq \mathfrak Z$, and clearly $\mathfrak Z$ satisfies (P3) if $Z'$ is replaced with $Z$ in (P3). 
	Let $A$ be the set of those pairs $(\alpha, z) \in \mathfrak Z \times Z$ such that $\alpha = \alpha_0 \mu e$ with $\alpha _0\in I_z$, $\mu$ a (possibly empty) product of inverse edges, and $e\in X\in \Cfin$.
	For $(\alpha, z)\in A$, we say that $\text{P2}(\alpha,z)$ {\it holds} if (P2) does hold for the given elements $\alpha \in \mathfrak Z$ and $z\in Z$. 
	
	For an element $\alpha \in \mathfrak Z$, consider the set 
	$$\mathcal P (\alpha) = \{ z\in Z : (\alpha,z) \in A \text{ and } \text{P2}(\alpha,z) \text{ does not hold } \} .$$
	We also set, for $\alpha \in \mathfrak Z$,
	$$d_Z (\alpha) = | \{ z\in Z : (\alpha, z) \in A \text{ and } \text{P2}(\alpha, z) \text{ does not hold } \}| = | \mathcal P (\alpha)|.$$
	If property (P2) does not hold for $\mathfrak Z$ and $Z$, we take $\alpha\in \mathfrak Z$ of minimal length such that $d_Z(\alpha) >0$. If $z\in \mathcal P (\alpha)$, then  $\alpha  = \alpha_0 \mu e$, where $\alpha _0\in I_z$, $\mu $ is a (possibly empty) product of inverse edges, and $e\in X\in \Cfin$, and there is no $f\in X$ such that $\alpha_0 \mu f\in I_z$. Set $\alpha (f) = \alpha_0 \mu f$, for $f\in X$, which is a $C$-separated path, and $I_{z,\alpha(f)} := I_z\cup \{\alpha(f)\}^{\downarrow}\in \mathcal X _0$, and $z_{\alpha(f)} = \e (I_{z,\alpha(f)})$. By Claim 1 we get
	$$\tilde{Z} = (Z\setminus \{z\})\cup \{ z_{\alpha(f)} : f\in X\}$$
	is also a finite cover of $p$, and also $\bigvee_{z\in Z}  \rho (z) = \bigvee_{\tilde{z}\in \tilde{Z}} \rho (\tilde{z})$. Moreover the finite cover $\tilde{Z}$ also satisfies (P1), and it satisfies (P3) with respect to the same set $\mathfrak Z$, because $\alpha = \alpha _0 \mu e\in \mathfrak Z$. Moreover $\alpha_0 \mu f \in I_{z, \alpha(f)}$ for each $f\in X$, hence we have
	$$d_{\tilde{Z}} (\alpha) = d_Z (\alpha) -1.$$
	Moreover the number $d_{\tilde{Z}} (\alpha')$, for $\alpha' \in \mathfrak Z$,
	can only be greater than $d_Z(\alpha')$ if $\alpha '$ has as a {\it proper} initial segment some of the paths $\alpha (f)$ for $f\in X$. Hence after $d_Z(\alpha)-1 $ additional  applications of Claim 1, we will get a finite cover $Z(\alpha)$ of $p$, with $\bigvee _{z\in Z} \rho (z) = \bigvee _{z'\in Z(\alpha)} \rho (z')$, such that $Z(\alpha)$ satisfies (P1) and also satisfies (P3) with respect to the original set $\mathfrak Z$. Moreover we will have that $d_{Z(\alpha)}(\alpha)= 0$ and $d_{Z(\alpha)} (\alpha')\le d_Z(\alpha')$ for all $\alpha'\in \mathfrak Z$ with the only possible exception of those $\alpha'\in \mathfrak Z$ which contain some of the paths $\alpha(f)$, for $f\in X$, as proper initial segments. In particular we have $d_{Z(\alpha)} (\beta) \le d_Z(\beta)$ for all $\beta\in \mathfrak Z$ with $|\beta | \le |\alpha|$. Hence we obtain a new finite cover $Z(\alpha)$ of $p$ with $d_{Z(\alpha)}(\beta) >0$ for one less path of minimal length $\beta\in \mathfrak Z$, and so we have reduced the number of paths of minimal length in $\mathfrak Z$ at which $d$ takes a positive value, or we have increased the minimal length of a path $\beta \in \mathfrak Z$ at which $d$ takes a positive value in case $\alpha$ was the only path of minimal length at which $d_Z$ takes a positive value. Note that the reference set $\mathfrak Z$ remains the same along all this process. Since $\mathfrak Z$ is finite, after a finite number of steps we will arrive at the desired finite cover $Z'$ of $p$.    
	
	This concludes the first part of the proof. Now we observe that if the 
	finite cover $Z$ of $p$ satisfies conditions (P1)--(P3) with respect to a finite reference set $\mathfrak Z$, then it also satisfies the same conditions with respect to the set $\mathfrak Z' :=  \bigcup_{z\in Z} I_z$, and then condition (P3) for the new set $\mathfrak Z'$ is redundant. Hence, changing notation, we will assume that $Z$ is a finite cover of $p$ such that the following conditions hold, where $\mathfrak Z := \bigcup _{z\in Z} I_z$:  
	
	\begin{enumerate}
		\item[(Q1)]  For each $z\in Z$ and each $\alpha = \alpha_I \alpha_{I_{z}\setminus I}\in I_{z}$, all edges appearing (with positive exponent) in the path  $\alpha_{I_{z}\setminus I}$ belong to some $X\in \Cfin$. 
		\item[(Q2)] For all $z\in Z$ and for all $\alpha \in \mathfrak Z$ such that $\alpha = \alpha_0 \mu e$, with $\alpha_0\in I_{z}$,
		$\mu$ is a (possibly empty) product of inverse edges, and $e\in X \in \Cfin$, there exists $f\in X$ such that $\alpha_0 \mu f\in I_{z}$.
	\end{enumerate}
	
	The next step is to show that if $\mathfrak Z \ne I$ then we can reduce the size of $Z$ such that conditions (Q1) and (Q2) still hold for the new finite cover $Z'$. 
	
	For this we need two more claims:
	
	\noindent \underline{\it Claim 2.} Assume that the finite cover $Z$ of $p$ satisfies (Q1) and (Q2). If $J \in \mathcal X _0$, with $I\subseteq J$, satisfies (Q1), (Q2) and the second part of (P3) with respect to the reference set $\mathfrak Z= \bigcup_{z\in Z} I_z$ (that is, (P3) is satisfied with $I_{z'}$ replaced with $J$ in its statement), then there is a unique $z\in Z$ such that $J = I_z$, hence $\e (J) = \e (I_z) = z$. 
	
	\medskip
	
	\noindent \underline{ \it Proof of Claim 2.} 
	Let $z\in Z$ such that $J$ is $C$-compatible with $I_z$. We will show that $J= I_z$. We start by showing that $J\subseteq I_z$. Looking for a contradiction, suppose there is $\alpha'\in \max (J)\setminus I_z$. We can then write $\alpha' = \alpha \gamma$, where $\alpha = \alpha _0 \mu e$, $\mu$ is a (possibly empty) product of inverse edges, and $e\in X\in C$, where $\alpha _0\in I_z$ and $\alpha\notin I_z$ . Observe that $\alpha\in J$ because $J$ is a lower set, and $X\in \Cfin$ by (Q1). By (P3) applied to $J$ and $\mathfrak Z$, there exists $f\in X$ such that $\alpha_0 \mu f \in \mathfrak Z$. Now by (Q2) applied to $I_z$ and $\alpha_0 \mu f\in \mathfrak Z$, there exists $h\in X$ such that $\alpha_0 \mu h\in I_z$. But now since $I_z$ and $J$ are $C$-compatible we must have that $h=e$, and thus $\alpha = \alpha_0 \mu e = \alpha_0 \mu h \in I_z$, a contradiction.
	
	Interchanging the roles of $J$ and $I_z$ in the above, we obtain that  $I_z\subseteq J$. Hence $J=I_z$, and this shows that there is exactly one $z\in Z$ such that $I_z$ is $C$-compatible with $J$, and this set $I_z$ is precisely $J$. This concludes the proof of Claim 2. 
	
	\medskip
	
	\noindent \underline{\it Claim 3.} Assume that the finite cover $Z$ of $p$ satisfies (Q1) and (Q2), and assume that $\mathfrak Z = \bigcup_{z\in Z} I_z \ne I$. Let $\alpha= \beta e \in \max(\mathfrak Z) \setminus I$, with $e\in X\in C$, be an element of maximal length amongst all elements in $\max(\mathfrak Z) \setminus I$.  Then $X\in \Cfin$ and for each $f\in X$, each $z\in Z$ such that
	$\beta f \in I_z$, and each $x\in X$, there exists a unique element $z(x)\in Z$ such that $I_{z(x)} = (I_z\setminus \{\beta f\}) \cup \{\beta x\}$. 
	
	\medskip
	
	\noindent \underline{ \it Proof of Claim 3.}  First observe that $X\in \Cfin$ by (Q1). For each $x\in X$ we set $\alpha (x) = \beta x$.
	We now show that $\alpha (x) \notin I$ for all $x\in X$. Let $z_0\in Z$ be such that $\alpha \in I_{z_0}$. If $\alpha (x)\in I$, then since $I\subset I_{z_0}$ and $\alpha =\beta e\notin I$, we would have that $\beta e, \beta x$ are two distinct elements of $I_{z_0}$, contradicting the fact that $I_{z_0}$ is $C$-compatible. Hence $\alpha (x)\notin I$ for all $x\in X$. 
	
	Now let $f\in X$ and $z\in Z$ be such that $\beta f \in I_z$.  Since $\alpha = \beta e$ is an element of maximal length in $\mathfrak Z \setminus I$, and $\alpha (f) = \beta f\notin I$ has the same length as $\alpha$, we infer that $\alpha (f)$ is a maximal element of $I_z$. Take $x\in X$. We can obviously assume that $x\ne f$. Set $J= J(z,x) = (I_z\setminus \{\alpha (f)\}) \cup \{\alpha(x) \} $.
	Since $\alpha (f)$ is a maximal element of $I_z$, it follows that $J$ is a lower subset of $\FC$. Then $J\in \mathcal X_0$, $I\subseteq J$, and $J$ satisfies (Q1), (Q2) and the second part of (P3) with respect to $\mathfrak Z = \bigcup_{z\in Z} I_z$, so by Claim 2 there exists a unique $z(x)\in Z$ such that $J= I_{z(x)}$. This concludes the proof of Claim 3.
	
	\medskip
	
	We are now ready for our reduction step.  
	Assume that the finite cover $Z$ of $p$ satisfies (Q1) and (Q2), and assume that $\mathfrak Z = \bigcup_{z\in Z} I_z \ne I$. Let $\alpha= \beta e \in \max(\mathfrak Z) \setminus I$, with $e\in X\in C$, be an element of maximal length amongst all elements in $\max(\mathfrak Z) \setminus I$. 
	We set 
	\begin{multline*}
	    Z ':= \{ z\in Z : \beta \notin I_z\} \cup \\ \{ z'\in \mathcal E : I_{z'} = (I_z\setminus \{\beta f\})_0 \text{ for all } z\in Z , f\in X \text{ such that } \beta f\in I_z \}.
	\end{multline*} 
	To show that $Z'$ is a finite cover of $p$, it is enough to observe that for each $z\in Z$ there is $z'\in Z'$ such that $z\le z'$. If $z\in Z$ and $\beta \notin I_z$ this is obvious. If $z\in Z$ and $\beta \in I_z$, then by (Q2) there is $f\in X$ such that $\beta f \in I_z$ and then clearly $z\le z'$, where $I_{z'} = (I_z\setminus \{\beta f\})_0$. This also shows that   $\bigvee_{z\in Z} \rho (z) \le  \bigvee_{z'\in Z'} \rho (z')$. Hence, in order to show that $\bigvee_{z\in Z} \rho (z) =  \bigvee_{z'\in Z'} \rho (z')$, it suffices to prove that for each $z'\in Z'$ we have $ \rho (z') \le  \bigvee_{z\in Z} \rho (z)$. We can assume that $I_{z'} = (I_z\setminus \{\beta f\})_0$, where $z\in Z$, $f\in X$ and $\beta f \in I_z $. By Claim 3, for each $x\in X$ 
	we have
	$$(I_z\setminus \{\beta f\}) \cup \{\beta x\} = I_{z(x)}$$
	for a unique $z(x) \in Z$, and clearly $I_z\setminus \{\beta f \} = I_{z(x)}\setminus \{ \beta x \} $. Hence, using Remark \ref{rem:e-for-X}, we get in $\Le _K ^\ab (E,C)$:
	\begin{align*}
		\rho (z') & = \e ((I_z\setminus \{\beta f\})_0) = \prod _{\gamma \in \max(I_z\setminus \{\beta f\})} \gamma \gamma^* \\
		& = (\beta\beta ^*)(\prod _{\gamma \in \max(I_z\setminus \{\beta f\})} \gamma \gamma^*) \\
		& = (\sum_{x\in X} (\beta x)(\beta x)^*) \Big( \prod _{\gamma \in \max(I_z\setminus \{\beta f\})} \gamma \gamma^*\Big)\\
		& = \sum_{x\in X} \Big[ (\beta x)(\beta x)^* \Big( \prod _{\gamma \in \max(I_z\setminus \{\beta x\})} \gamma \gamma^*\Big)\Big]\\
		& = \sum _{x\in X} \e (I_{z(x)}) = \sum_{x\in X} \rho (z(x)) \le \bigvee _{z\in Z} \rho (z). 
	\end{align*}
	Hence $Z'$ is a finite cover of $p$ with $\bigvee_{z'\in Z'} \rho (z')= \bigvee _{z\in Z} \rho (z)$, and $\bigcup _{z' \in Z'} I_{z'}\subsetneq \bigcup_{z\in Z} I_z$, because $\alpha =\beta e \notin \bigcup _{z' \in Z'} I_{z'}$. Moreover $Z'$ satisfies conditions (Q1) and (Q2) with respect to $\mathfrak Z ' := \bigcup _{z' \in Z'} I_{z'}$. Hence by induction we conclude that 
	$$\bigvee_{z\in Z}\rho (z) = \bigvee_{z'\in Z'} \rho (z') = \e (I) = p.$$
	This concludes the proof of the theorem.
\end{proof}

We can now state and prove our main structural result for the $*$-algebras $\Le_K^\ab (E,C)$ and the $C^*$-algebras $\OO (E,C)$.
Recall that the tight algebras and tight groupoid appearing in the statement have been introduced in Definition \ref{def:tight-algebras}.

\begin{theorem}\label{thm:iso-tight-algebras}
	If $(E,C)$ is a separated graph and $K$ is a commutative unital ring with involution, then 
	\begin{equation}\label{eq:iso:tight-algebras}
		\Le_K^\ab (E,C) \cong K_\tight[\IS(E,C)] = A_K (\G_\tight(E,C)) \cong C_K(\dual\E_\tight) \rtimes \F.
	\end{equation}
	A similar result holds for \cstar{}algebras:
	\begin{equation}\label{eq:iso:tight-cst-algebras}
		\OO(E,C)\cong C^*_\tight(\IS(E,C)) = C^*(\G_\tight(E,C))\cong \contz(\dual\E_\tight)\rtimes\F.    \end{equation}
	Here the crossed products are with respect to the  canonical partial action (restricted to $\dual\E_\tight$), which is the same as the one appearing in Proposition \ref{pro:iso-groupoid-co-algebras} and Corollary~\ref{cor:iso-Co-algebras}.
\end{theorem}
\begin{proof}
Theorem~\ref{thm:Leavitt-universal-tight} shows that $\Le_K^\ab(E,C)$ is the universal $*$-algebra over $K$ for tight $*$\nb-homomorphisms. 
	Hence the isomorphism $$\Le_K^\ab (E,C) \cong K_\tight[\IS(E,C)] := A_K (\mathcal G_\tight(E,C))$$ follows from \cite{SteinbergSzakacs}*{Corollary 2.5}. 
	
	The tight spectrum $\dual\E_\tight$ is invariant under the partial action of $\F$ (see \cite{XinLi}*{Lemma~3.2}) and the isomorphism $\G(E,C)\cong \dual\E\rtimes\F$ from Proposition~\ref{pro:iso-groupoid-co-algebras} factors through the tight quotient: $\G_\tight(E,C)= \dual\E_\tight\rtimes \IS (E,C)\cong \dual\E_\tight\rtimes \F$. The last isomorphism in~\eqref{eq:iso:tight-algebras} then follows from \cite{BeuterGoncalves}*{Theorem~3.2}.
	
	As in the proof of Proposition~\ref{pro:iso-groupoid-co-algebras} we can use \cite{Abadie}*{Theorem~3.3} to deduce the \cstar{}algebra isomorphism~\eqref{eq:iso:tight-cst-algebras}.
\end{proof}

\begin{remark}
	Similar observations as in Remark~\ref{rem:reduced-Toeplitz} hold for $\OO(E,C)$: first, in view of Theorem \ref{thm:iso-tight-algebras}, it is natural to define $\OO_r(E,C):=C^*_{\tight,r}(\IS(E,C))$. By \cite{XinLi}*{Proposition~3.1}, the isomorphism $\OO(E,C)\cong \contz(\dual\E_\tight)\rtimes\F$ from Theorem \ref{thm:iso-tight-algebras} factors through the reduced \cstar{}algebras:
	$$\OO_r(E,C)\cong C^*_r(\G_\tight(E,C))\cong \contz(\dual\E_\tight)\rtimes_r\F.$$
	This is compatible with the definition of $\OO_r(E,C)$ appearing in \cite{Lolk:tame}. In particular, $\OO_r(E,C)$ is always an exact \cstar{}algebra, although it might be not nuclear. And $\OO(E,C)$ is not even exact in general by Remark~\ref{rem:reduced-Toeplitz}. By \cite{BussFerraroSehnem}*{Theorem~3.10}, $\OO_r(E,C)$ is nuclear if and only if $\OO(E,C)=\OO_r(E,C)$, and this happens if and only if the groupoid $\G_\tight(E,C)$ is amenable, which is in turn equivalent to amenability (or approximation property) of the underlying partial action of $\F$ on $\dual\E_\tight$. Related to these observations, Theorem~5.1 in \cite{Lolk:nuclearity} gives several equivalent conditions for nuclearity of $\OO(E,C)$ or, equivalently, $\OO_r(E,C)$ when $E$ is a finitely separated graph.
\end{remark}

\section*{Acknowledgments} The authors are very grateful to Ben Steinberg for his many comments and suggestions on an earlier version of the paper. In particular, the approach to Theorem \ref{thm:ECMunntrees} using Munn trees and the free inverse semigroup $\FIS(E)$ is due to him. We also would like to thank the anonymous referee for their careful reading and constructive comments, which helped us improve the clarity and presentation of the paper.

\end{document}